\documentclass{amsart}
\usepackage{a4wide}
\usepackage{amsmath}
\usepackage{amssymb}
\usepackage{graphicx}
\usepackage[portuguese,english]{babel}
\usepackage[latin1]{inputenc}
\usepackage[T1]{fontenc}
\usepackage{hyperref}
\usepackage{enumitem}
\usepackage{multicol}
\usepackage{mathtools}
\usepackage{multirow}
\usepackage{array}
\usepackage{booktabs}
\usepackage{tikz}
\usetikzlibrary{arrows}
\usetikzlibrary{shapes}
\usetikzlibrary{positioning}
\usepackage[toc,page]{appendix}
\usepackage[all]{xy}
\usepackage{float}

\def\Q{\mathbb{Q}}
\def\X{\mathbb{X}}
\def\Z{\mathbb{Z}}

\def\CP{\mathbb{CP}}

\def\to{\rightarrow}

\def \bbcp{\mathbb C\mathbb P}
\newcommand{\Symp}{\mathrm{Symp}}
\newcommand{\Diff}{\mathrm{Diff}}

\newcommand{\Mcc}{\widetilde{M}_{c_1,c_2}}
\newcommand{\Muccc}{\widetilde{M}_{\mu,c_1,c_2,c_3}}
\newcommand{\Mucccc}{\widetilde{M}_{\mu,c_1,c_2,c_3,c_4}}
\newcommand{\Mccc}{\widetilde{M}_{c_1,c_2,c_3}}

\newcommand{\Jccc}{\mathcal{J}_{c_1,c_2,c_3}}
\newcommand{\Jcccm}{\mathcal{J}_{c_1,c_2,c_3,m}}
\newcommand{\Guccc}{G_{\mu,c_1,c_2,c_3}}
\newcommand{\Gccc}{G_{c_1,c_2,c_3}}
\newcommand{\Gcccc}{G_{c_1,c_2,c_3,c_4}}
\newcommand{\Gcc}{G_{c_1,c_2}}

\newtheorem{theorem}{Theorem}[section]
\newtheorem{lemma}[theorem]{Lemma}
\newtheorem{proposition}[theorem]{Proposition}
\newtheorem{corollary}[theorem]{Corollary}
\newtheorem{definition}[theorem]{Definition}
\newtheorem{remark}[theorem]{Remark}
\newtheorem{conjecture[theorem]}{Conjecture}
\newtheorem*{theorem*}{Theorem}
\newtheorem*{lemma*}{Lemma}
\newtheorem*{proposition*}{Proposition}
\newtheorem*{corollary*}{Corollary}
\newtheorem*{definition*}{Definition}
\newtheorem*{remark*}{Remark}

\usepackage{srcltx}

\newcommand*\circled[1]{\tikz[baseline=(char.base)]{
            \node[shape=circle,draw,inner sep=2pt] (char) {#1};}}
            
\begin{document}

\title[Homotopy algebra of symplectomorphism groups]{The homotopy Lie algebra of symplectomorphism groups of 3--fold blow-ups of   $(S^2 \times S^2, \sigma_{std} \oplus \sigma_{std}) $}

\author[S. Anjos]{S\'ilvia Anjos}
\address{SA: Center for Mathematical Analysis, Geometry and Dynamical Systems \\ Mathematics Department \\  Instituto Superior T\'ecnico \\  Av. Rovisco Pais \\ 1049-001 Lisboa \\ Portugal}
\email{sanjos@math.ist.utl.pt}

\author[S. Eden]{Sinan Eden}
\address{SE: Center for Mathematical Analysis, Geometry and Dynamical Systems \\ Mathematics Department \\  Instituto Superior T\'ecnico \\  Av. Rovisco Pais \\ 1049-001 Lisboa \\ Portugal}
\email{sinaneden@yahoo.com }

\date{\today}

\thanks{The first author is partially supported by FCT/Portugal through projects UID/MAT/04459/2013 and EXCL/MAT-GEO/0222/2012.
The second author is partially supported by FCT/Portugal through fellowship SFRH/BD/87329/2012, and through projects UID/MAT/04459/2013 and PTDC/MAT/098936/2008.}

\begin{abstract}

We consider the 3-point blow-up of the manifold $ (S^2 \times S^2, \sigma \oplus \sigma)$ where $\sigma$ is the standard symplectic form which gives area 1 to the sphere $S^2$, and study its group of symplectomorphisms  $\Symp ( S^2 \times S^2  \#\, 3\overline{ \bbcp}\,\!^2, \omega)$. So far, the monotone case was studied by J. Evans in \cite{Eva}, where he proved that this group is contractible. Moreover, J. Li, T. J. Li and W. Wu in \cite{LiLiWu} showed that the group Symp$_{h}(S^2 \times S^2  \#\, 3\overline{ \bbcp}\,\!^2,\omega)  $ of symplectomorphisms that act trivially on homology is always connected and recently they also computed, in \cite{LiLiWu2}, its fundamental group. We describe, in full detail, the rational homotopy Lie algebra of this group.

We show that some particular circle actions contained in the symplectomorphism group  generate its full topology. More precisely, they give the generators of the homotopy graded Lie algebra of $\Symp (S^2 \times S^2  \#\, 3\overline{ \bbcp}\,\!^2, \omega)$.  Our study depends on Karshon's classification of Hamiltonian circle actions and the inflation technique introduced by Lalonde-McDuff.
As an application, we deduce the rank of the homotopy groups of  $\Symp( \bbcp^2  \#\, 5\overline{ \bbcp}\,\!^2, \widetilde \omega)$,  in the case of small blow-ups. 

\end{abstract}

\subjclass[2010]{Primary 53D35; Secondary 57R17,57S05,57T20}
\keywords{symplectic geometry, symplectomorphism group, homotopy type, Lie graded algebra, J-holomorphic curves}

\maketitle

\tableofcontents

\section{Introduction}\label{section: intro}

The understanding of the homotopy type of symplectomorphism groups is one of the fundamental questions in symplectic topology. If $(M, \omega)$ is a closed simply connected symplectic manifold, then the symplectomorphism group $ \Symp(M, \omega)$ endowed with the standard $C^\infty$-topology, is an infinite dimensional Fréchet Lie group. Except for dimension 4, there are not many tools available to compute the homotopy type of the symplectomorphism group. However, in this case a detailed study of the space of almost complex stuctures compatible with $\omega$, $ \mathcal{J} _{\omega} $, has produced several results. More precisely, the long-known example of  the symplectomorphism group of $ (S ^{2} \times S ^{2},  \sigma \oplus \sigma) $, where $\sigma$ denotes the standard symplectic symplectic form on $S^2$ that gives area 1 to the sphere,  was studied by M. Gromov in \cite{Gro}. Among many other groundbreaking results he proved that the group of symplectomorphisms of this symplectic manifold has the homotopy type of the semi-direct product $(SO(3) \times SO(3)) \ltimes \Z_2$.  Several other authors obtained results about the symplectomorphism of this manifold with different symplectic forms as well as their symplectic blow-ups. More precisely, M. Abreu and D. McDuff  in \cite{AbrMcD} described the rational cohomology ring of the group of symplectomorphisms that act as the identity on homology, on the same manifold equipped with the symplectic form $ \omega_{\mu} = \mu \sigma \oplus \sigma $ where $ \mu \geq 1$. The group of symplectomorphisms of these manifolds with only one blow-up, acting trivially on homology, is calculated by M. Pinsonnault in \cite{Pin} and the 2-fold blow-up is calculated by S. Anjos and M. Pinsonnault in \cite{AnjPin}.
The next simplest thing to do would be to analyze the 3-fold blow-up. So far, the monotone case  was studied by J. Evans in \cite{Eva}, where he proved that the symplectomorphism group in this case is contractible. Furthermore, J. Li, T. J. Li and W. Wu showed in \cite{LiLiWu} that the homology trivial part of $\Symp(S^2 \times S^2  \#\, 3\overline{ \bbcp}\,\!^2, \omega) $ is connected and recently, in \cite{LiLiWu2} they computed its fundamental group. Moreover, they completely determine $\pi_0(\Symp(S^2 \times S^2  \#\, 4\overline{ \bbcp}\,\!^2, \omega))$  and find a lower bound for the rank of the fundamental group of the symplectomorphism group in this case, which together with the upper bound obtained by D. McDuff in \cite{McD4} gives the precise rank in most occasions. In order to obtain these results the authors study the space of tamed almost complex structures using a fine decomposition via smooth rational curves and a relative version of the infinite dimensional Alexander duality. 

The goal of this paper is  to give a complete description of the homotopy  Lie algebra  of the symplectomorphism group $\Symp_h(S^2 \times S^2  \#\, 3\overline{ \bbcp}\,\!^2, \omega) $, for some $\omega$ (see below). As an application we compute the rank of the homotopy groups of $\Symp_h(\bbcp  \#\, 5\overline{ \bbcp}\,\!^2, \omega) $,  for small blow-ups. 

In all cases studied by now, the main idea behind the study of the homotopy type of the symplectomorphism group is to investigate the natural action of this group on  the space of almost complex structures $ \mathcal{J} _{\omega} $. In those cases, namely rational ruled manifolds and some of their blow-ups, the space $ \mathcal{J} _{\omega} $ is a stratified space where each stratum contains an integrable almost complex structure for which its isotropy group is a finite dimensional Lie group. As in the previous cases, we will show that the isotropy groups of the integrable almost complex structures generate the full topology of the symplectomorphism group. More precisely, in our case, the isotropy groups are either $ \mathbb{T}^{2} $ or $ S^{1} $ or the identity, and we will show that the Hamiltonian circle actions contained in them give the generators of the homotopy graded Lie algebra of the symplectomorphism group.

This work adds to a series of projects in the area, focusing on the symplectic blow-ups of the symplectic manifold $ M_\mu = ( S ^{2} \times S ^{2} , \mu \sigma \oplus \sigma )$, $ \mu \geq 1 $. More precisely,  let $ \Muccc$ be obtained from $M_\mu$,  by performing three successive blow-ups of capacities  $ c_{1}$, $ c_{2}$ and $c_{3}$, with $\mu\geq 1 > c_{1} + c_{2} > c_{1} + c_{3} > c_{1} > c_{2} > c_{3} > 0 $.  
In Section \ref{sec m1: reduced} we explain why it is sufficient to consider values of $c_1,c_2$ and $c_3$ in this range. In this paper we will consider the case $\mu=1$. 

Seeing $ S ^{2} \times S ^{2}$ as a trivial fibration over $S^2$, the homology class of the base $ B \in H _{2} ( S ^{2} \times S ^{2} )$ representing $ [ S ^{2} \times \{ pt \} ]$ has area $ \mu =1$, and the homology class of the fiber $ F \in H _{2} ( S ^{2} \times S ^{2} )$ representing $ [ \{ pt \} \times S ^{2} ] $ has area $ 1 $.

Equivalently, we can start with $( \mathbb{CP} ^{2} , \omega _{\nu} ) $, where $ \omega _{\nu} $ is the standard Fubini-Study form rescaled so that $ \omega _{\nu} (\mathbb{CP} ^{1} ) = \nu $, and blow up at four balls of capacities $ \delta_{1}, \delta_{2}, \delta_{3} $ and $ \delta_{4} $. This gives a space conventionally denoted by $ \mathbb{X} _{4} = \mathbb{CP} ^{2} \# 4 \overline{ \mathbb{CP} ^{2} } $, equipped with the symplectic structure $\omega _ { \nu ; \delta_{1}, \delta_{2}, \delta_{3}, \delta_{4} }$.

One easy way to understand the equivalence between $ \mathbb{X} _{4} $ and $\Muccc$ is as follows: let $ \{ L, V_{1}, V_{2}, V_{3}, V_{4} \} $ be the basis for $ H_{2} (\mathbb{X}_{4} ; \mathbb{Z}) $ where $L$ is the class representing a line and the $V_{i} $ are the exceptional classes. Let $ \{ B, F, E_{1}, E_{2}, E_{3} \} $ be the basis for $H_{2}(\widetilde{M}_{\mu , c_{1} , c_{2}, c_{3}} ; \mathbb{Z}) $ where the $ E_{i}$ represent the exceptional spheres arising from the blow-ups. We identify $L$ with $ B+F-E_{1} $, $ V_{1}$ with $ B - E_{1}$, $ V_{2} $ with $ F - E_{1} $, $ V_{3} $ with $ E_{2} $, and $ V_{4} $ with $ E_{3} $. In order to see this birational equivalence in the symplectic category, we recall  {\it the uniqueness of symplectic blow-ups} due to D. McDuff (see \cite[Corollary 1.3]{McD}): the symplectomorphism type of a symplectic blow-up of a rational ruled manifold along an embedded ball of capacity $ c \in (0,1) $ depends only on the capacity $ c $ and not on the particular embedding used in obtaining the blow-up. Using this result and after rescaling, we obtain the equivalence in the symplectic category:
\begin{equation}\label{sympequiv}
 \mu = \dfrac{\nu - \delta_{2}}{\nu - \delta_{1}} , \quad  c_{1}= \dfrac{\nu - \delta_{1} - \delta_{2}}{\nu - \delta_{1}}, \quad  c_{2}= \dfrac{\delta_{3}}{\nu - \delta_{1}}, \quad \mbox{and} \quad  c_{3} = \dfrac{\delta_{4}}{\nu - \delta_{1}} .
 \end{equation}
Let $ G_{c_{1} , c_{2}, c_{3} } = G_{\mu=1, c_{1} , c_{2}, c_{3} } $ denote the group of symplectomorphisms of $ \widetilde{M}_{c_{1} , c_{2}, c_{3} } = \widetilde{M}_{\mu=1, c_{1} , c_{2}, c_{3} }$ acting trivially on homology, and $ \mathcal{J} _{c_{1} , c_{2}, c_{3} } $ the space of $\omega$-compatible almost complex structures on $ \widetilde{M}_{c_{1} , c_{2}, c_{3} }$.

In order to describe the algebraic structure of $ \pi_{*}( G_{c_{1} , c_{2}, c_{3}})$, we introduce the Lie product with which we equip it, namely the Samelson product. Recall, that if $ G $ is a connected topological group, the Samelson product  $ [ \cdot , \cdot ] : \pi _{p} (G) \otimes \pi _{q} (G) \rightarrow \pi _{p+q} (G) $ is defined by the commutator map
$$ S^{p+q} = S^{p} \times S^{q} / S^{p} \vee S^{q} \rightarrow G \ : \ (s,t) \mapsto a(s)b(t)a^{-1}(s)b^{-1}(t) .$$
The Samelson product satisfies the following two properties:
\begin{align*}
 & \mbox{(i)}  \quad [x,y] = (-1)^{\deg x \deg y + 1} [y,x];  &  & \mbox{(antisymmetry)} \\
 & \mbox{(ii)}  \quad [x,[y,z]] = [[x,y],z] + (-1)^{\deg x \deg y}[y,[x,z]]. &  & \mbox{(Jacobi identity)}
\end{align*}
The vector space $ \pi _{*} (G) \otimes \mathbb{Q} $ with the Samelson product is thus a graded Lie algebra (i.e. a graded vector space together with a linear map of degree zero satisfying antisymmetry and the Jacobi identity). 

The main theorem gives a complete description of the homotopy Lie algebra of $G_{c_1,c_2,c_3}.$ 

\begin{theorem}\label{mainthm m1}
Let $ \widetilde{M} _{c_{1} , c_{2}, c_{3} } $ denote the symplectic manifold obtained from $ (S ^{2} \times S ^{2}, \sigma \oplus \sigma )$, by performing three successive blow-ups of capacities of $ c_{1}$, $ c_{2}$ and $c_{3}$, with $ 1 > c_{1} + c_{2} > c_{1} + c_{3} > c_{1} > c_{2} > c_{3}. $ Let $ G_{c_{1} , c_{2}, c_{3} }$ denote the group of symplectomorphisms of $ \widetilde{M}_{c_{1} , c_{2}, c_{3} } $ acting trivially on homology.
Then $ \pi _{*} (G_{c_{1} , c_{2}, c_{3} }) \otimes \mathbb{Q} $ is isomorphic to an algebra generated by nine elements of degree 1 (which we define in Section \ref{sec m1: generators}) satisfying  a set of relations between their Lie products, which are specified in Section \ref{proof m1}. In particular,  $\pi _1 (G_{c_{1} , c_{2}, c_{3} }) \otimes \mathbb{Q} \simeq \Q^9$.
\end{theorem}

We will see in Section \ref{sec m1: generators} that these generators are represented by Hamiltonian $ S^{1} $-actions on the symplectic manifold, so the rational homotopy type of the symplectomorphism group $ G_{c_{1} , c_{2}, c_{3}} $ is generated by these circle actions.
While the general structure of this work is parallel to the study of the homotopy groups of $ G_{c_{1} , c_{2}} $, there are a few new phenomena occuring when we blow-up once more. More precisely, due to a symmetry between the roles of $c_2$ and $c_3$, two relations  of a new nature between circle actions on $\Mccc$, emerge via the use of auxiliary toric pictures (see Section \ref{section: new relation m1}). We will show that these relations appear more naturally when $ \mu >1$.

\subsubsection*{Organization of the paper} The paper is divided in four  sections and three appendices containing auxiliary technical material. In Section \ref{chp m1: Structure} we explain the stratification of the space of compatible almost complex structures on $\Mccc$.  In Section \ref{chp m1: homotopy type}, we study the homotopy type of the symplectomorphism group $G_{c_1,c_2,c_3}$. This means that we first state the main result about the stability of the symplectomorphism groups, giving the main idea of the proof and postponing the complete details to Appendix A.  Then we compute the homotopy groups of $G_{c_1,c_2,c_3}$ and study the circle actions on $\Mccc$, using Karshon's classification of Hamiltonian circle actions.  Finally, at the end of this section, we state  and prove  in full detail Theorem \ref{mainthm m1}. In particular we give a geometric interpretation of the generators of the fundamental group of $G_{c_1,c_2,c_3}$ as circle actions on the manifold $\Mccc$. Section \ref{applications} is devoted to the applications of the main theorem. Up to this point we just studied the generic case, that is, we studied the topology of the group of symplectomorphisms $\Gccc$ when $ 1 > c_{1} + c_{2} > c_{1} + c_{3} > c_{1} > c_{2} > c_{3} $. In Section \ref{SpecialCases} we give some remarks about some of the remaining cases. In Appendix B we state some results on the differential and topological aspects of the space $ \mathcal{J}_{\omega} = \tilde{\mathcal{J}}_{c_{1},c_{2},c_{3}} $ of compatible almost complext structures on $ \widetilde{M}_{c_{1} , c_{2}, c_{3} }$.  Our main goal is to illustrate the correspondence between isometry groups and configurations of $J$--homomorphic curves. Appendix C is devoted to the study of some particular toric actions in the case $\mu >1$ that relate with the case $\mu=1$ and therefore are useful in the proof of the main theorem. 

\subsubsection*{Acknowledgements}
Both authors are very grateful to Jarek Kedra and Martin Pinsonnault for reading a preliminary version of this paper carefully and making several important suggestions and corrections that have helped to clarify various arguments and improve the paper. In particular, we would like to thank them for a very useful suggestion on how to simplify the proof of Lemma \ref{rank proposition} and pointing out that the same argument could be applied to compute the homotopy groups of $\Symp(\bbcp^2  \#\, 5\overline{ \bbcp}\,\!^2, \tilde\omega)$ (see Corollary \ref{mainapplication}). We also would like to thank the referee for his/her thorough work in reviewing the paper, in particular, for identifying a gap in the proof of Theorem 3.2 and for giving a suggestion to fix it.

\section{The structure of \texorpdfstring{$J$}{Lg}-holomorphic curves}\label{chp m1: Structure}

\subsection{Reduced classes}\label{sec m1: reduced}

We start by justifying the conditions in Theorem \ref{mainthm m1} on the sizes of the capacities, namely the assumption that $ 1 > c_{1} + c_{2} > c_{1} + c_{3} > c_{1} > c_{2} > c_{3} > 0 $. We will show that this assumption is sufficient to cover the generic case. The key component of the following Lemma is the notion of a \textit{reduced class}: 

\begin{definition}
A class $ A = a_{0} L - \sum_{1 \leq i \leq 4 } a_{i} V_{i} $ is \textit{reduced} with respect to the basis $ \{ L, V_{1} , ... , V_{4} \} $ if $ a_{1} \geq a_{2} \geq a_{3} \geq a_{4} $ and $ a_{0} \geq a_{1} + a_{2} + a_{3} $.
\end{definition}

\begin{lemma}\label{reduced m1}
Every symplectic form on $ \widetilde{M}_{c_{1} , c_{2}, c_{3} }$ is, after rescaling, diffeomorphic to a form Poincar\'{e} dual to $ B + F - c_{1} E_{1} - c_{2} E_{2} - c_{3} E_{3} $ with $ 0 < c_{3} \leq c_{2} \leq c_{1} < c_{1} + c_{3} \leq c_{1} + c_{2} \leq 1 $.
\end{lemma}
\begin{proof}
The arguments in \cite{AnjPin} work equally here to give a proof of the lemma. 
Note first that diffeomorphic symplectic forms define symplectomorphism groups that are homeomorphic, and symplectomorphisms are invariant under rescalings of symplectic forms. So we need to describe a fundamental domain for the action of $ \Diff^{+} \times \mathbb{R_*} $ on the space $ \Omega _{+} $ of orientation-compatible symplectic forms defined on the 4--fold blow-up $ \mathbb{X}_{4} $. Taking $ \{ L, V_{1} , ... , V_{4} \} $ as the standard basis of $ H _{2} ( \mathbb{X}_{4} ; \mathbb{Z} ) $ as in the Introduction, the Poincar\'{e} dual of $ [ \omega_{\nu, \delta_{1}, \delta_{2}, \delta_{3}, \delta_{4}}]$ is $ \nu L - \Sigma _{i} \delta_{i} V_{i} $. Similarly, the first Chern class of any compatible almost-complex structure on $ \mathbb{X}_{4} $ is Poincar\'{e} dual to $ K : = 3L - \Sigma _{i} V _{i} $.

Now if $C$ stands for the Poincar\'{e} dual of the symplectic cone of $ \mathbb{X}_{4} $, then by  {\it uniqueness of symplectic blow-ups} in \cite{McD}, the diffeomorphism class of the form $ \omega_{\nu, \delta_{1}, \delta_{2}, \delta_{3}, \delta_{4}}$ only depends on its cohomology class. Therefore, it is enough to describe a fundamental domain of the action of $ \Diff^{+} \times \mathbb{R_*} $ on $C$.
Moreover, the canonical class $K$ is unique up to orientation preserving diffeomorphisms \cite{LiLiu2}, so it suffices to describe the action of the diffeomorphisms fixing $K$, $\Diff_K$, on
$$ C_{K} = \{ A \in  H_{2} ( \mathbb{X}_{4} ; \mathbb{R} ) \ : \ A = PD [ \omega ]  \ \mbox{for some} \ \omega \in \Omega_{K} \} $$

\noindent where $ \Omega _{K}$ is the set of orientation-compatible symplectic forms with $K$ as the symplectic canonical class.
By a result due to \cite{LiLiu3}, the set of reduced classes is a fundamental domain of $C_{K} ( \mathbb{X}_{4}  )$ under the action of $ \Diff_{K}$ so by the equivalence in the symplectic category given in \ref{sympequiv} we obtain the desired result. 
\end{proof}

\begin{remark}\label{c1c2c3 assumption m1}
Note that there are several possibilities within this condition. We don't know the relation between $ c_{1} + c_{2} + c_{3}$ and $1$, and between $ c_{2} + c_{3} $ and $ c_{1}$. We will show below that this does not make a difference in the homotopy type of $ G_{c_{1} , c_{2}, c_{3} }$ (see Remark \ref{No need for c1c2c3<1}).
\end{remark}

\subsection{Configurations of \texorpdfstring{$J$}{Lg}-holomorphic curves}\label{sec m1: configs}

Throughout this section we will assume that $ 1 > c_{1} + c_{2} > c_{1} + c_{3} > c_{1} > c_{2} > c_{3} > 0 $ (that is, without allowing for the equalities) and study the possible configurations of $ J $-holomorphic curves. In this study, we use two basic properties of $ J $-holomorphic curves on symplectic manifolds of dimension 4, namely the positivity of intersections and the adjunction formula (see Theorem 2.6.3 and Theorem 2.6.4 in \cite{McDSal}).

\textbf{Positivity of Intersections}: Two distinct closed $ J $-curves $ S $ and $ S' $ in an almost complex 4-manifold $(M, J)$ have only a finite number of intersection points. Each such point $ x $ contributes a number $ k_{x} \geq 1 $ to the algebraic intersection number $ [S] \cdot [S'] $. Moreover, $ k_{x} = 1 $ if and only if the curves $ S $ and $ S' $ intersect transversally at $ x $. Thus, $ [S] \cdot [S'] = 0 $ if and only if $S$ and $S'$ are disjoint, and $ [S] \cdot [S'] =1 $ if and only if the curves meet exactly once transversally and at a point which is non-singular on both curves.

\textbf{Adjunction Formula}: For $ S $ the image of a $ J $-holomorphic map $ u : \Sigma \rightarrow M$, we define the virtual genus of $ S $ as the number
$$ g_v(S) = 1 + \frac12([S] \cdot [S] - c_{1} ([S])).$$
Then, if $ u $ is somewhere injective, $ g_v(S) $ is an integer. Moreover $ g_v(S) \geq
\text{genus}(\Sigma) $ with equality if and only if $ S $ is embedded. If $\Sigma = S^2$ and $c_{1} ([S])- [S] \cdot [S] =2$ then $u$ is a smooth embedding. 

\begin{lemma}\label{B-curve m1}
Let $ J \in  \mathcal{J} _{c_{1} , c_{2}, c_{3} } $. Suppose $ A = pB + qF - r_{1} E_{1} - r_{2} E_{2} - r_{3} E_{3} \in H_{2}(  \widetilde{M}_{c_{1} , c_{2}, c_{3} } , \mathbb{Z} ) $ has a simple $J$-holomorphic representative. Then $ p \geq 0 $. Moreover
\begin{itemize}
\item if $p=0$, then A is one of the followings: $F$, $F-E_{i}$, $F-E_{i}-E_{j} $, $ E_{i} $, $ E_{i} - E_{j}$, $E_{1} - E_{2} - E_{3}$, $i,j \in \{ 1,2,3 \}$ and $i<j$;
\item if $p=1 $, then $r_{1}, r_{2}, r_{3} \in \{ 0,1 \} $.
\end{itemize}
Since $\mu=1$, a similar result holds for $q$. 
\end{lemma}
\begin{proof}
By the adjunction inequality, we have
$$ 2 g_{v} ( A ) = 2(p-1) (q-1) - r_{1} (r_{1} - 1) - r_{2} (r_{2}-1) - r_{3} (r_{3}-1) \geq 0.$$

Also, since $ \omega (B) = \omega (F) = 1 \geq \omega (E_{1}) = c_{1} \geq \omega (E_{2}) = c_{2} \geq \omega (E_{3}) = c_{3} $, we have
 $ \omega (A) = p + q - c_{1} r_{1} - c_{2} r_{2} - c_{3} r_{3} > 0 $.

Now we show $ p \geq 0 $: Suppose $ p < 0$. Then  $ p < \dfrac{1}{2} $ and so
\begin{align*}
  \phantom{-2 g_{v}(A)}
  &\begin{aligned}
    \mathllap{-2 g_{v} (A) } > (q-1) + r_{1} (r_{1} - 1) + r_{2} (r_{2}-1) + r_{3} (r_{3}-1)
  \end{aligned}\\
  &\begin{aligned}
     \geq (c_{1} r_{1} + c_{2} r_{2} + c_{3} r_{3} - p -1) + r_{1} (r_{1} - 1) + r_{2} (r_{2}-1) + r_{3} (r_{3}-1)
  \end{aligned}\\
  &\begin{aligned}
   \geq c_{1} r_{1} + c_{2} r_{2} + c_{3} r_{3} + r_{1} (r_{1} - 1) + r_{2} (r_{2}-1) + r_{3} (r_{3}-1)
  \end{aligned}\\
  &\begin{aligned}
   = r_{1} (r_{1} -1 + c_{1}) + r_{2} (r_{2} -1 + c_{2}) + r_{3} (r_{3} -1 + c_{3}) \geq 0
  \end{aligned}
\end{align*}
which yields $ g_{v} (A) <0 $, contradicting the adjunction inequality. Now suppose $ p=0$. Then the adjunction formula reads
$$ 2 (q-1) + r_{1} (r_{1} - 1) + r_{2} (r_{2}-1) + r_{3} (r_{3}-1) \leq 0,$$
and the positivity of intersections yields $ q - c_{1} r_{1} - c_{2} r_{2} - c_{3} r_{3} \geq 0 $. The values for $q, r_{1}, r_{2}$ and $r_{3}$ that satisfy these inequalities are given in Table \ref{values qr}.

\begin{table}[thp]
\begin{center}
\begin{tabular}{ |m{3em}|m{3em}|m{3em}|m{3em}|m{8em}| } 
 \hline
 $ q $ & $ r_{1} $ & $ r_{2} $ & $ r_{3} $ & A \\ \hline
 1 & 0 & 0 & 0 & $ F $  \\ \hline
 1 & 0 & 0 & 1 & $ F-E_{3} $  \\ \hline
 1 & 0 & 1 & 0 & $ F-E_{2} $  \\ \hline 
 1 & 0 & 1 & 1 & $ F-E_{2}-E_{3} $  \\ \hline
 1 & 1 & 0 & 0 & $ F-E_{1} $  \\ \hline
 1 & 1 & 0 & 1 & $ F-E_{1}-E_{3} $  \\ \hline
 1 & 1 & 1 & 0 & $ F-E_{1}-E_{2} $  \\ \hline
 1 & 1 & 1 & 1 & $ F-E_{1}-E_{2}-E_{3} $  \\ \hline
 0 & -1 & 0 & 0 & $ E_{1} $  \\ \hline
 0 & -1 & 0 & 1 & $ E_{1}-E_{3} $  \\ \hline
 0 & -1 & 1 & 0 & $ E_{1}-E_{2} $  \\ \hline
 0 & -1 & 1 & 1 & $ E_{1}-E_{2}-E_{3} $  \\ \hline
 0 & 0 & 0 & -1 & $ E_{3} $  \\ \hline
 0 & 0 & -1 & 0 & $ E_{2} $  \\ \hline
 0 & 0 & -1 & 1 & $ E_{2}-E_{3} $  \\ \hline
\end{tabular}
\end{center}
\smallskip
\caption{The values for $q, r_{1}, r_{2}$ and $r_{3}$, and the resulting curves.}\label{values qr}
\end{table}
Note that the curves $F-E_{1}-E_{2}-E_{3}$ and $E_{1}-E_{2}-E_{3}$ are represented only if  $ c_{1} + c_{2} + c_{3} < 1 $ and $ c_{2} + c_{3} < c_{1}$, respectively.

Finally if $p=1$, we get $r_{1} (r_{1} - 1) + r_{2} (r_{2}-1) + r_{3} (r_{3}-1) \leq 0$ so that $ r_{1} , r_{2} , r_{3} \in \{0,1 \} $.
\end{proof}

\begin{lemma}\label{E3 m1}
Let $ J \in \mathcal{J} _{c_{1} , c_{2}, c_{3} } $. Then $E_{3}$ is represented by a unique embedded J-curve.
Hence, if $ A = pB + qF - r_{1} E_{1} - r_{2} E_{2} - r_{3} E_{3}$, with $A \neq E_3$,  has a simple J-representative, then $r_{3} \geq 0$.
\end{lemma}
\begin{proof}
Suppose, by contradiction, that $E_{3}$ is represented by a cusp-curve

\begin{center}
 $ C = \bigcup \limits_{i=1}^{N} m_{i} C_{i}$, $N \geq 2$ and $C_{i}$ simple. 
\end{center}

Each $C_{i}$ is one of the options given in Lemma \ref{B-curve m1}. By area considerations ($ \omega (E_{3}) > \omega (C_{i})$ for all $i$), $C_{i}$ cannot be representing $F, F-E_{1}, F-E_{2},  E_{1},  E_{1}-E_{2},  F-E_{1}-E_{2}$ or $E_{3}$. But then the options we are left with all include $-E_{3}$, which is also impossible as $ m_{i} >0 $ and $N \geq 2$.
\end{proof}

\begin{remark} \label{E3 is important} The fact that the class $ E_{3} $ is always represented by a unique embedded $ J $-holomorphic curve will be extremely useful hereinafter. It allows us to analyze almost-complex structures as well as the symplectomorphisms on the manifold $ \widetilde{M}_{c_{1} , c_{2}, c_{3} }$ via the ones on $ \widetilde{M}_{c_{1} , c_{2}}$, for we can blow-down $ E_{3} $ to obtain the structures for the 2-point blow-up.
\end{remark}

\begin{remark} If we start with $E_{2}$ (instead of $E_{3}$) in Lemma \ref{E3 m1}, a similar argument shows that we must have $C_{i} = E_{2}-E_{3}$ unless we have a simple representative of $E_{2}$.
\end{remark}

Recall that if $(X, \omega)$ is a symplectic 4-manifold and $A \in H_2(X, \Z)$ is a homology class such that $k(A)=\frac12 (A \cdot A +c_1(A)) \geq 0$, then the Gromov invariant of $A$ counts, for a generic complex structure $J$ tamed by $\omega$, the algebraic number of embedded $J$-holomorphic curves in class $A$ passing through $k(A)$ generic points. These curves need not be connected. Therefore, components of non-negative self-intersection are embedded holomorphic curves of some genus $g \geq 0$, while all other components are exceptional spheres. It follows from Gromov's Compactness Theorem that given a tamed $J$, any class $A$ with non-zero Gromov invariant Gr($A$), is represented by a collection of $J$-holomorphic curves or cusp curves. It is also known that if $(X, \omega)$ is some blow-up of a rational ruled surface, then for every class $A \in H_2(X, \Z)$ verifying $k(A) \geq  0$ we have Gr($A$)$=\pm 1$ (it follows from \cite[Corollary 1.4]{LiLiu1}). 

Now consider the classes $D_i \in H_2 ( \Muccc; \Z)$, $ i \in \Z$, by setting 
\begin{align*}
  & D_{4k}= B+kF;  & &  D_{4k-2(j)}= B+kF-E_\ell -E_{\ell'}, \quad  j\neq \ell \neq \ell'; \\
 & D_{4k-1(j)}= B+kF-E_j, \quad j=1,2,3;  & &  D_{4k-3}= B+kF-E_1- E_2-E_3. 
\end{align*}

Note that we have 
$$
D_i \cdot D_i  = \left\{ 
\begin{array}{ll}
{ \frac{i}{2} }& \mbox{if}  \quad i=4k  \\ 
 & \\
{ \frac{i-1}{2} }& \mbox{if}  \quad i=4k-1  \\
  & \\
{ \frac{i}{2}-1} & \mbox{if}  \quad i=4k-2   \\
  & \\
{\frac{i-3}{2} }  & \mbox{if}  \quad i=4k-3
  \end{array} \right. 
  \quad \mbox{and} \quad
c_1(D_i ) = \left\{ 
\begin{array}{ll}
{ \frac{i}{2} +2}& \mbox{if}  \quad i=4k  \\ 
 & \\
{ \frac{i+3}{2} }& \mbox{if}  \quad i=4k-1  \\
  & \\
{ \frac{i}{2}+1} & \mbox{if}  \quad i=4k-2   \\
  & \\
{\frac{i+1}{2} }  & \mbox{if}  \quad i=4k-3.
  \end{array} \right. $$

  Thus the adjunction formula $g_v(D_i)= 1 + \frac12 (D_i \cdot D_i -c_1(D_i ))=0$ implies that any $J$-holomorphic sphere in class $D_i$ must be embedded. Moreover, the virtual dimension of the spaces of unparameterized $J$-curves representing $D_i$, $\dim {\mathcal M}(D_i, J) =2c_1(D_i)-2$, is 
$$
\dim {\mathcal M}(D_i, J) = \left\{ 
\begin{array}{ll}
{i+2 }& \mbox{if}  \quad i=4k  \\ 
 & \\
{ i+1 }& \mbox{if}  \quad i=4k-1  \\
  & \\
{ i} & \mbox{if}  \quad i=4k-2   \\
  & \\
{i-1}  & \mbox{if}  \quad i=4k-3
  \end{array} \right. 
  \quad \mbox{and} \quad
k(D_i)= \left\{ 
\begin{array}{ll}
{ \frac{i}{2} +1}& \mbox{if}  \quad i=4k  \\ 
 & \\
{ \frac{i+1}{2} }& \mbox{if}  \quad i=4k-1  \\
  & \\
{ \frac{i}{2}} & \mbox{if}  \quad i=4k-2   \\
  & \\
{\frac{i-1}{2} }  & \mbox{if}  \quad i=4k-3.
  \end{array} \right. $$

\smallskip

\begin{lemma}\label{existenceJ-curves_generic}
Given $i \geq 1$, the Gromov invariant of the class $D_i$ is non-zero. Therefore, for a generic tamed almost complex structure $J$, there exist an embedded $J$-holomorphic sphere representing $D_i$ and passing through a generic set of $k(D_i)$ points in $\Muccc$. 
\end{lemma}
\begin{proof}
The argument is the same used in \cite[Lemma 2.9]{AnjPin} or in \cite[Lemma 2.4]{Pin}. As in these papers we can conclude that Gr($D_i$) $= \pm 1$. Hence given a generic $J$, $D_i$ is represented by the disjoint union of some embedded $J$-holomorphic curve of genus $g \geq 0$ with some exceptional $J$-spheres. Since $D_i$ has non-negative intersection with the exceptional classes $E_i, B-E_i, F-E_i$, $ i=1,2,3$, this representative is connected. Moreover, it is a sphere because its genus is zero, by the adjunction formula. 
\end{proof}

\begin{lemma}\label{J-holomorphics m1}
The set of tamed almost complex structures on $ \widetilde{M}_{c_{1} , c_{2}, c_{3} }$ for which the classes $B - E_{i} $, $i=1,2,3$ are represented by an embedded $J$-holomorphic sphere is open and dense in  $ \mathcal{J} _{c_{1} , c_{2}, c_{3} } $.
If for a given $J$ there are no such spheres, then either one of the classes $B - E_{i} - E_{j} $, $i<j$, or $B - E_{1} - E_{2} - E_{3} $ is represented by a unique embedded $J$-holomorphic sphere.
\end{lemma}
\begin{proof}
The proof of this Lemma is similar to the proof of Lemma 2.5 in \cite{Pin} and of Lemma 2.10 in \cite{AnjPin}. We only need to modify their proofs to include the various curves available in our case. We will confine ourselves with giving a sketch of the proof.

Since the curves $B - E_{i} $, $i=1,2,3$, are exceptional, the set of almost complex structures $ J $ for which they have an embedded $ J $-holomorphic representative is open and dense. Since they are primitive, they cannot have multiply-covered representatives. Suppose, then, that the representative of $ B - E_{i} $ is a cusp curve, which due to Lemma \ref{B-curve m1} and Lemma \ref{E3 m1} would decompose as $ B - r_{1} E_{1} - r_{2} E_{2} - r_{3} E_{3} $ and a linear combination of the following classes: $ F $, $ F - E_{j} $, $ F - E_{j} - E_{k} $, $ F - E_{1} - E_{2} - E_{3}$, $ E_{j} $, $ E_{j} - E_{k} $, $ E_{1} - E_{2} - E_{3}$. Here all the coefficients in the linear combination would be non-negative and $ r_{i} \in \{ 0,1 \} $. This implies that at least two of the  $ r_{i} $ should be non-zero, hence one of the classes $B - E_{i} - E_{j} $ $i<j$ or $B - E_{1} - E_{2} - E_{3} $ is represented by a unique embedded $J$-holomorphic sphere.
\end{proof}

\begin{remark}
By positivity of intersections, we cannot have $ F - E_{i} $ and $ F - E_{i} - E_{j} $ simultaneously represented by embedded $J$-holomorphic spheres. Similarly, it is not possible for the classes $ F - E_{i} $ and $ F - E_{1} - E_{2} - E_{3} $ to simultaneously have embedded $ J $-holomorphic representatives.
\end{remark}

Using the positivity of intersections and the results of this section, we thus get a set of possible configurations of $J$-holomorphic curves. Note that all the configurations can be obtained from the previous case of two blow-ups, that is, from the configurations of J-holomorphic curves in $\Mcc$, by blowing-up either inside a curve, or at an intersection between curves, or outside (see Remark \ref{E3 is important}). We will recall the configurations in \cite{AnjPin} and draw the configurations resulting from the next blow-up. For instance,  we reproduce in Figure \ref{Config1 of AP} configuration  (1) of Figure 2 in \cite{AnjPin}. Blowing-up at an intersection point (points 1 to 6 in Figure \ref{Config1 of AP}) we obtain the six configurations given in Figure \ref{Config1 of AP blown up 1-6}. In Figure  \ref{Config1 of AP blown up 7-12} we show the 6 configurations obtained by blowing-up at a point inside a curve (points 7 to 12 in Figure \ref{Config1 of AP}), and finally, blowing-up outside any of these curves (point 13) we obtain  configuration 1.13 in Figure \ref{Config1 of AP blown up 13}.

\begin{figure}[thp]

\begin{center}
\begin{tikzpicture}[scale=0.8,font=\scriptsize,squarednode/.style={rectangle, draw=blue!60, fill=blue!5, very thick, minimum size=5mm}]

	\node[squarednode] at (-3,3) (maintopic) {1} ;
         \draw (-1.4,2) -- (1.4,2);
	\draw (-0.8,2.4) -- (-2.4,-0.2);
	\draw (-2.4,0.2) -- (-0.8,-2.4);
	\draw (-1.4,-2) -- (1.4,-2);
	\draw (0.8,2.4) -- (2.4,-0.2);
	\draw (0.8,-2.4) -- (2.4,0.2);

	\node at (0,2.3) {$B-E_{2}$};
	\node at (2.1,1) {$E_{2}$};
	\node at (-2.5,1) {$F-E_{1} $};
	\node at (-2.1,-1) {$E_{1} $};
	\node at (0,-2.3) {$B-E_{1} $};
	\node at (2.4, -1) {$F-E_{2}$};
	\node at (1.2,-2.3){$\color{red} 1$};
	\node at (1.04,-2){$\color{red} \bullet $};
	\node at (2.6,0){$\color{red} 2$};
	\node at (2.28,0){$\color{red} \bullet $};
	\node at (1.2,2.3){$\color{red} 3$};
	\node at (1.04,2){$\color{red} \bullet $};
	\node at (-1.2,2.3){$\color{red} 4$};
	\node at (-1.04,2){$\color{red} \bullet $};
	\node at (-2.6,0){$\color{red} 5$};
	\node at (-2.28,0){$\color{red} \bullet $};
	\node at (-1.2,-2.3){$\color{red} 6$};
	\node at (-1.04,-2){$\color{red} \bullet $};
	\node at (0,-1.7) {$\color{red} 7$};
	\node at (0,-2){$\color{red} \bullet $};
	\node at (1.4, -1) {$\color{red} 8$};
	\node at (1.66,-1){$\color{red} \bullet $};
	\node at (1.4,1){$\color{red} 9$};
	\node at (1.66,1){$\color{red} \bullet $};
	\node at (0,1.7) {$\color{red} 10$};
	\node at (0,2){$\color{red} \bullet $};
	\node at (-1.3,1) {$\color{red} 11$};
	\node at (-1.66,1){$\color{red} \bullet $};
	\node at (-1.3,-1) {$\color{red} 12$};
	\node at (-1.66,-1){$\color{red} \bullet $};
	\node at (0,0){$\color{red} \bullet $};
	\node at (0.3,0){$\color{red} 13$};

\end{tikzpicture}
\end{center}

\caption{Configuration (1) in \cite{AnjPin} for $J$-holomorphic curves in $\Mcc$.}
\label{Config1 of AP}
\end{figure}

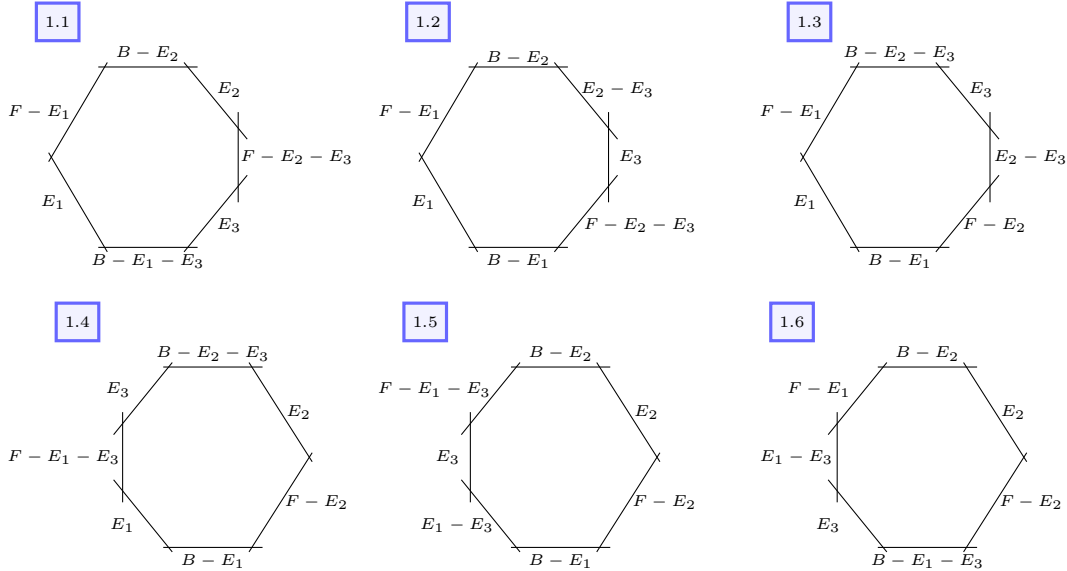
\begin{figure}[thp]

\bigskip

\begin{minipage}{.33\textwidth}
\begin{tikzpicture}[scale=0.6,squarednode/.style={rectangle, draw=blue!60, fill=blue!5, very thick, minimum size=5mm}, font=\tiny]

	\node[squarednode] at (-3,3) (maintopic) {1.1} ; 
    \draw (-2.1,2) -- (0.1,2);
	\draw (-1.9,2.1) -- (-3.2,-0.1);
	\draw (-3.2,0.1) -- (-1.9,-2.1);
	\draw (-2.1,-2) -- (0.1,-2);
	\draw (-0.2,-2.1) -- (1.2,-0.4);
	\draw (1,1) -- (1,-1);
	\draw (-0.2,2.1) -- (1.2,0.4);

	\node at (-1,2.3) {$B-E_{2}$};
	\node at (0.8,1.5) {$E_{2}$};
	\node at (-3.4,1) {$F-E_{1} $};
	\node at (-3.1,-1) {$E_{1} $};
	\node at (-1,-2.3) {$B-E_{1}-E_{3} $};
	\node at (0.8, -1.5) {$E_{3}$};
	\node at (2.3,0) {$F-E_{2}-E_{3}$};
\end{tikzpicture}
\end{minipage}%
\begin{minipage}{.34\textwidth}
\begin{tikzpicture}[scale=0.6,squarednode/.style={rectangle, draw=blue!60, fill=blue!5, very thick, minimum size=5mm}, font=\tiny]

	\node[squarednode] at (-3,3) (maintopic) {1.2} ; 
    \draw (-2.1,2) -- (0.1,2);
	\draw (-1.9,2.1) -- (-3.2,-0.1);
	\draw (-3.2,0.1) -- (-1.9,-2.1);
	\draw (-2.1,-2) -- (0.1,-2);
	\draw (-0.2,-2.1) -- (1.2,-0.4);
	\draw (1,1) -- (1,-1);
	\draw (-0.2,2.1) -- (1.2,0.4);

	\node at (-1,2.2) {$B-E_{2}$};
	\node at (1.2,1.5) {$E_{2}-E_{3}$};
	\node at (-3.4,1) {$F-E_{1} $};
	\node at (-3.1,-1) {$E_{1} $};
	\node at (-1,-2.3) {$B-E_{1} $};
	\node at (1.7, -1.5) {$F-E_{2}-E_{3}$};
	\node at (1.5,0) {$E_{3}$};

\end{tikzpicture}
\end{minipage}%
\begin{minipage}{.33\textwidth}
\begin{tikzpicture}[scale=0.6,squarednode/.style={rectangle, draw=blue!60, fill=blue!5, very thick, minimum size=5mm}, font=\tiny]

	\node[squarednode] at (-3,3) (maintopic) {1.3} ;  
    \draw (-2.1,2) -- (0.1,2);
	\draw (-1.9,2.1) -- (-3.2,-0.1);
	\draw (-3.2,0.1) -- (-1.9,-2.1);
	\draw (-2.1,-2) -- (0.1,-2);
	\draw (-0.2,-2.1) -- (1.2,-0.4);
	\draw (1,1) -- (1,-1);
	\draw (-0.2,2.1) -- (1.2,0.4);

	\node at (-1,2.3) {$B-E_{2}-E_{3}$};
	\node at (0.8,1.5) {$E_{3}$};
	\node at (-3.4,1) {$F-E_{1} $};
	\node at (-3.1,-1) {$E_{1} $};
	\node at (-1,-2.3) {$B-E_{1}$};
	\node at (1.1, -1.5) {$F-E_{2}$};
	\node at (1.9,0) {$E_{2}-E_{3}$};

\end{tikzpicture}

\end{minipage}

\bigskip

\begin{minipage}{.33\textwidth}

\begin{tikzpicture}[scale=0.6,squarednode/.style={rectangle, draw=blue!60, fill=blue!5, very thick, minimum size=5mm}, font=\tiny]

	\node[squarednode] at (-4,3) (maintopic) {1.4} ;
    \draw (-2.1,2) -- (0.1,2);
	\draw (-1.9,2.1) -- (-3.2,0.5);
	\draw (-3.2,-0.5) -- (-1.9,-2.1);
	\draw (-2.1,-2) -- (0.1,-2);
	\draw (-0.2,-2.1) -- (1.2,0.1);
	\draw (-0.2,2.1) -- (1.2,-0.1);
	\draw (-3,1) -- (-3,-1);

	\node at (-1,2.3) {$B-E_{2}-E_{3}$};
	\node at (0.9,1) {$E_{2}$};
	\node at (-3.1,1.5) {$E_{3} $};
	\node at (-3,-1.5) {$E_{1}$};
	\node at (-1,-2.3) {$B-E_{1}$};
	\node at (1.3, -1) {$F-E_{2}$};
	\node at (-4.3, 0) {$F-E_{1}-E_{3} $};

\end{tikzpicture}
\end{minipage}%
\begin{minipage}{.34\textwidth}
\begin{tikzpicture}[scale=0.6,squarednode/.style={rectangle, draw=blue!60, fill=blue!5, very thick, minimum size=5mm}, font=\tiny]

	\node[squarednode] at (-4,3) (maintopic) {1.5} ;

    \draw (-2.1,2) -- (0.1,2);
	\draw (-1.9,2.1) -- (-3.2,0.5);
	\draw (-3.2,-0.5) -- (-1.9,-2.1);
	\draw (-2.1,-2) -- (0.1,-2);
	\draw (-0.2,-2.1) -- (1.2,0.1);
	\draw (-0.2,2.1) -- (1.2,-0.1);
	\draw (-3,1) -- (-3,-1);

	\node at (-1,2.3) {$B-E_{2}$};
	\node at (0.9,1) {$E_{2}$};
	\node at (-3.8,1.5) {$F-E_{1}-E_{3} $};
	\node at (-3.3,-1.5) {$E_{1}-E_{3} $};
	\node at (-1,-2.3) {$B-E_{1}$};
	\node at (1.3, -1) {$F-E_{2}$};
	\node at (-3.5, 0) {$E_{3} $};

\end{tikzpicture}

\end{minipage}%
\begin{minipage}{.33\textwidth}

\begin{tikzpicture}[scale=0.6,squarednode/.style={rectangle, draw=blue!60, fill=blue!5, very thick, minimum size=5mm}, font=\tiny]

	\node[squarednode] at (-4,3) (maintopic) {1.6} ;
    \draw (-2.1,2) -- (0.1,2);
	\draw (-1.9,2.1) -- (-3.2,0.5);
	\draw (-3.2,-0.5) -- (-1.9,-2.1);
	\draw (-2.1,-2) -- (0.1,-2);
	\draw (-0.2,-2.1) -- (1.2,0.1);
	\draw (-0.2,2.1) -- (1.2,-0.1);
	\draw (-3,1) -- (-3,-1);

	\node at (-1,2.3) {$B-E_{2}$};
	\node at (0.9,1) {$E_{2}$};
	\node at (-3.4,1.5) {$F-E_{1} $};
	\node at (-3.2,-1.5) {$E_{3} $};
	\node at (-1,-2.3) {$B-E_{1}-E_{3} $};
	\node at (1.3, -1) {$F-E_{2}$};
	\node at (-3.9, 0) {$E_{1}-E_{3} $};

\end{tikzpicture}
\end{minipage}
         
\caption{Configurations 1.1 --  1.6}
\label{Config1 of AP blown up 1-6}
\end{figure}

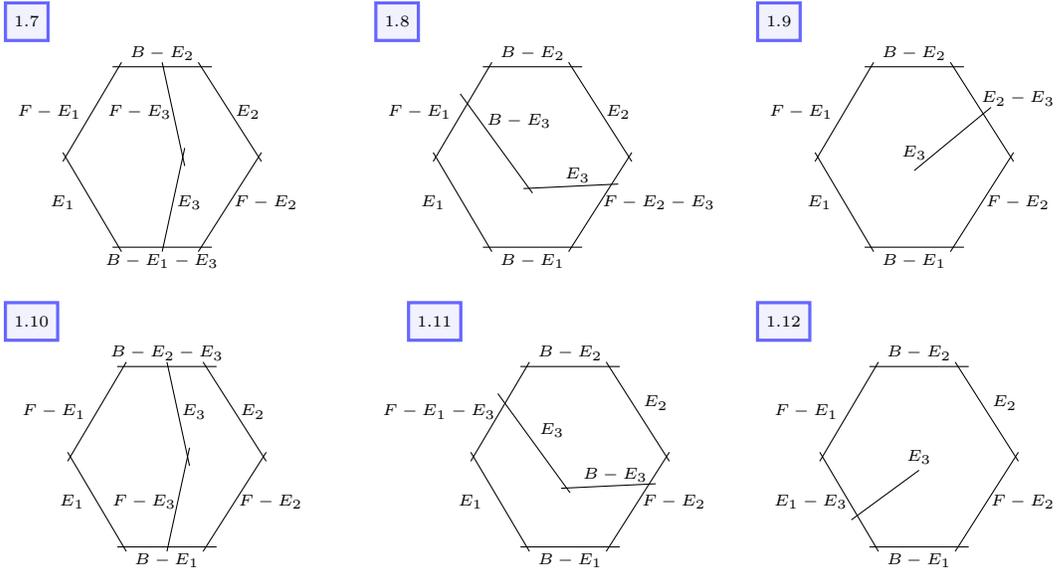
\begin{figure}[thp]

\bigskip

\begin{minipage}{.33\textwidth}
\begin{tikzpicture}[scale=0.6,squarednode/.style={rectangle, draw=blue!60, fill=blue!5, very thick, minimum size=5mm}, font=\tiny]

	\node[squarednode] at (-4,3) (maintopic) {1.7} ;
    \draw (-2.1,2) -- (0.1,2);
	\draw (-2.1,-2) -- (0.1,-2);
	\draw (-0.2,-2.1) -- (1.2,0.1);
	\draw (-0.2,2.1) -- (1.2,-0.1);
	\draw (-1.9,2.1) -- (-3.2,-0.1);
	\draw (-3.2,0.1) -- (-1.9,-2.1);
	\draw (-0.5,-0.2) -- (-1,2.1);
	\draw (-0.5,0.2) -- (-1,-2.1);

	\node at (-1,2.3) {$B-E_{2}$};
	\node at (0.9,1) {$E_{2}$};
	\node at (-3.5,1) {$F-E_{1} $};
	\node at (-3.2,-1) {$E_{1}$};
	\node at (-1,-2.3) {$B-E_{1}-E_{3}$};
	\node at (1.3, -1) {$F-E_{2}$};
	\node at (-1.5,1) {$F-E_{3} $};
	\node at (-0.4,-1) {$E_{3} $};

\end{tikzpicture}
\end{minipage}%
\begin{minipage}{.34\textwidth}

\begin{tikzpicture}[scale=0.6,squarednode/.style={rectangle, draw=blue!60, fill=blue!5, very thick, minimum size=5mm}, font=\tiny]

	\node[squarednode] at (-4,3) (maintopic) {1.8} ;
    \draw (-2.1,2) -- (0.1,2);
	\draw (-2.1,-2) -- (0.1,-2);
	\draw (-0.2,-2.1) -- (1.2,0.1);
	\draw (-0.2,2.1) -- (1.2,-0.1);
	\draw (-1.9,2.1) -- (-3.2,-0.1);
	\draw (-3.2,0.1) -- (-1.9,-2.1);
	\draw (-1.2,-0.7) -- (0.9,-0.6);
	\draw (-1,-0.8) -- (-2.6,1.4);

	\node at (-1,2.3) {$B-E_{2}$};
	\node at (0.9,1) {$E_{2}$};
	\node at (-3.5,1) {$F-E_{1} $};
	\node at (-3.2,-1) {$E_{1}$};
	\node at (-1,-2.3) {$B-E_{1}$};
	\node at (1.8, -1) {$F-E_{2}-E_{3}$};
	\node at (0, -0.4) {$E_{3} $};
	\node at (-1.3,0.8) {$B-E_{3} $};

\end{tikzpicture}
\end{minipage}%
\begin{minipage}{.33\textwidth}
\begin{tikzpicture}[scale=0.6,squarednode/.style={rectangle, draw=blue!60, fill=blue!5, very thick, minimum size=5mm}, font=\tiny]

	\node[squarednode] at (-4,3) (maintopic) {1.9} ;

    \draw (-2.1,2) -- (0.1,2);
	\draw (-2.1,-2) -- (0.1,-2);
	\draw (-0.2,-2.1) -- (1.2,0.1);
	\draw (-0.2,2.1) -- (1.2,-0.1);
	\draw (-1.9,2.1) -- (-3.2,-0.1);
	\draw (-3.2,0.1) -- (-1.9,-2.1);
	\draw (-1,-0.3) -- (0.7,1.1);

	\node at (-1,2.3) {$B-E_{2}$};
	\node at (1.3,1.3) {$E_{2}-E_{3}$};
	\node at (-3.5,1) {$F-E_{1} $};
	\node at (-3.1,-1) {$E_{1}$};
	\node at (-1,-2.3) {$B-E_{1}$};
	\node at (1.3, -1) {$F-E_{2}$};
	\node at (-1, 0.1) {$E_{3} $};

\end{tikzpicture}

\end{minipage}

\bigskip

\begin{minipage}{.33\textwidth}

\begin{tikzpicture}[scale=0.6,squarednode/.style={rectangle, draw=blue!60, fill=blue!5, very thick, minimum size=5mm}, font=\tiny]

	\node[squarednode] at (-4,3) (maintopic) {1.10} ;

    \draw (-2.1,2) -- (0.1,2);
	\draw (-2.1,-2) -- (0.1,-2);
	\draw (-0.2,-2.1) -- (1.2,0.1);
	\draw (-0.2,2.1) -- (1.2,-0.1);
	\draw (-1.9,2.1) -- (-3.2,-0.1);
	\draw (-3.2,0.1) -- (-1.9,-2.1);
	\draw (-0.5,-0.2) -- (-1,2.1);
	\draw (-0.5,0.2) -- (-1,-2.1);

	\node at (-1,2.3) {$B-E_{2}-E_{3}$};
	\node at (0.9,1) {$E_{2}$};
	\node at (-3.5,1) {$F-E_{1} $};
	\node at (-3.1,-1) {$E_{1}$};
	\node at (-1,-2.3) {$B-E_{1}$};
	\node at (1.3, -1) {$F-E_{2}$};
	\node at (-0.4,1) {$E_{3} $};
	\node at (-1.5,-1) {$F-E_{3} $};

\end{tikzpicture}
\end{minipage}%
\begin{minipage}{.34\textwidth}
\begin{tikzpicture}[scale=0.6,squarednode/.style={rectangle, draw=blue!60, fill=blue!5, very thick, minimum size=5mm}, font=\tiny]

	\node[squarednode] at (-4,3) (maintopic) {1.11} ;

    \draw (-2.1,2) -- (0.1,2);
	\draw (-2.1,-2) -- (0.1,-2);
	\draw (-0.2,-2.1) -- (1.2,0.1);
	\draw (-0.2,2.1) -- (1.2,-0.1);
	\draw (-1.9,2.1) -- (-3.2,-0.1);
	\draw (-3.2,0.1) -- (-1.9,-2.1);
	\draw (-1.2,-0.7) -- (0.9,-0.6);
	\draw (-1,-0.8) -- (-2.6,1.4);

	\node at (-1,2.3) {$B-E_{2}$};
	\node at (0.9,1.2) {$E_{2}$};
	\node at (-3.9,1) {$F-E_{1}-E_{3} $};
	\node at (-3.2,-1) {$E_{1}$};
	\node at (-1,-2.3) {$B-E_{1}$};
	\node at (1.3, -1) {$F-E_{2}$};
	\node at (0, -0.4) {$B-E_{3} $};
	\node at (-1.4,0.6) {$E_{3} $};

\end{tikzpicture}

\end{minipage}%
\begin{minipage}{.33\textwidth}

\begin{tikzpicture}[scale=0.6,squarednode/.style={rectangle, draw=blue!60, fill=blue!5, very thick, minimum size=5mm}, font=\tiny]

	\node[squarednode] at (-4,3) (maintopic) {1.12} ;
   \draw (-2.1,2) -- (0.1,2);
	\draw (-2.1,-2) -- (0.1,-2);
	\draw (-0.2,-2.1) -- (1.2,0.1);
	\draw (-0.2,2.1) -- (1.2,-0.1);
	\draw (-1.9,2.1) -- (-3.2,-0.1);
	\draw (-3.2,0.1) -- (-1.9,-2.1);
	\draw (-1,-0.3) -- (-2.5,-1.4);

	\node at (-1,2.3) {$B-E_{2}$};
	\node at (0.9,1.2) {$E_{2}$};
	\node at (-3.5,1) {$F-E_{1} $};
	\node at (-3.4,-1) {$E_{1}-E_{3}$};
	\node at (-1,-2.3) {$B-E_{1}$};
	\node at (1.3, -1) {$F-E_{2}$};
	\node at (-1, 0) {$E_{3} $};

\end{tikzpicture}
\end{minipage}

\caption{Configurations 1.7 -- 1.12}
\label{Config1 of AP blown up 7-12}

\bigskip

\end{figure}

\bigskip

\begin{figure}[thp]

\begin{center}
\begin{tikzpicture}[scale=0.8,squarednode/.style={rectangle, draw=blue!60, fill=blue!5, very thick, minimum size=5mm}, font=\scriptsize]

	\node[squarednode] at (-3,3) (maintopic) {1.13} ;
 
    \draw (-1.4,2) -- (1.4,2);
	\draw (-0.8,2.4) -- (-2.4,-0.2);
	\draw (-2.4,0.2) -- (-0.8,-2.4);
	\draw (-1.4,-2) -- (1.4,-2);
	\draw (0.8,2.4) -- (2.4,-0.2);
	\draw (0.8,-2.4) -- (2.4,0.2);
	\draw (-0.7,2.6) -- (1,-3.7);
	\draw (2.4,-2.4) -- (0.25,-0.25);
	\draw[dashed] (-0.25,0.25) -- (0.25,-0.25);
	\draw (-0.25,0.25) -- (-1.8,1.8);
	\draw (0.7,-3.6) -- (2.4,-2);

	\node at (0.2,2.2) {$B-E_{2}$};
	\node at (2.1,1) {$E_{2}$};
	\node at (-2.4,1) {$F-E_{1} $};
	\node at (-2.1,-1) {$E_{1} $};
	\node at (-0.2,-2.3) {$B-E_{1} $};
	\node at (2.4, -1) {$F-E_{2}$};
	\node at (1.7,-3) {$E_{3}$};
	\node[rotate=-75] at (0,-1.1) {$F-E_{3}$};
	\node[rotate=-45] at (1,-0.7) {$B-E_{3}$};

\end{tikzpicture}
\end{center}
\caption{Configuration 1.13}
\label{Config1 of AP blown up 13}
\end{figure}
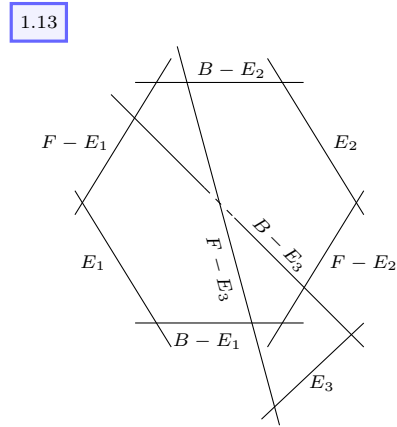

Choosing now configuration (3) of Figure 2 in \cite{AnjPin},\footnote{In order to simplify the presentation in further sections, we will change the labelling of the configurations.} which is represented in Figure \ref{Config3 of AP}, after blowing-up we obtain ten different configurations which are represented in Figures \ref{Config3 of AP blown up 1-5} and \ref{Config3 of AP blown up 6-10}. 

\begin{figure}[thp]

\begin{center}
\begin{tikzpicture}[scale=0.8,squarednode/.style={rectangle, draw=blue!60, fill=blue!5, very thick, minimum size=5mm}, font=\scriptsize]

	\node[squarednode] at (-3,3) (maintopic) {2} ;

    \draw (-1.4,2) -- (1.5,2.5);
	\draw (-0.8,2.4) -- (-2.4,-0.2);
	\draw (-2.4,0.2) -- (-0.8,-2.4);
	\draw (-1.4,-2) -- (1.4,-2);
	\draw (0.8,2.7) -- (2.9,0.7);

	\node at (-0.1,2.5) {$E_{1}$};
	\node at (-2.9,1) {$B-E_{1}-E_{2} $};
	\node at (-2.1,-1) {$E_{2} $};
	\node at (0.2,-2.3) {$F-E_{2} $};
	\node at (2.6,2) {$F-E_{1}$};

\end{tikzpicture}
\caption{Configuration (3) in \cite{AnjPin} for $J$-homomorphic curves in $\Mcc$}
\label{Config3 of AP}
\end{center}
\end{figure}

\begin{figure}[thp]

\begin{minipage}{.33\textwidth}
\begin{tikzpicture}[scale=0.6,squarednode/.style={rectangle, draw=blue!60, fill=blue!5, very thick, minimum size=5mm}, font=\tiny]

	\node[squarednode] at (-3,3) (maintopic) {2.1} ; 
    \draw (-2.1,2) -- (0.1,2);
	\draw (-1.9,2.1) -- (-3.2,-0.1);
	\draw (-3.2,0.1) -- (-1.9,-2.1);
	\draw (-2.1,-2) -- (0.1,-2);
	\draw (-0.2,-2.1) -- (1.2,-0.4);
	\draw (1,1) -- (1,-1);
	\draw (-0.2,2.1) -- (1.2,0.4);

	\node at (-1,2.3) {$E_{1}$};
	\node at (-3.9,1) {$B-E_{1}-E_{2} $};
	\node at (-3.1,-1) {$E_{2} $};
	\node at (-1,-2.3) {$F-E_{2}-E_{3} $};
	\node at (0.9, -1.5) {$E_{3}$};
	\node at (1.1,1.5) {$F-E_{1}$};
	\node at (1.8,0) {$B-E_{3}$};

\end{tikzpicture}

\end{minipage}%
\begin{minipage}{.34\textwidth}

\begin{tikzpicture}[scale=0.6, squarednode/.style={rectangle, draw=blue!60, fill=blue!5, very thick, minimum size=5mm}, font=\tiny]

	\node[squarednode] at (-3,3) (maintopic) {2.2} ; 
    \draw (-2.1,2) -- (0.1,2);
	\draw (-1.9,2.1) -- (-3.2,-0.1);
	\draw (-3.2,0.1) -- (-1.9,-2.1);
	\draw (-2.1,-2) -- (0.1,-2);
	\draw (-0.2,-2.1) -- (1.2,-0.4);
	\draw (1,1) -- (1,-1);
	\draw (-0.2,2.1) -- (1.2,0.4);

	\node at (-1,2.3) {$E_{1}$};
	\node at (-3.9,1) {$B-E_{1}-E_{2} $};
	\node at (-3.1,-1) {$E_{2} $};
	\node at (-1,-2.3) {$F-E_{2}$};
	\node at (1.8, 0) {$E_{3}$};
	\node at (1.6,1.5) {$F-E_{1}-E_{3}$};
	\node at (1.1, -1.5) {$B-E_{3}$};

\end{tikzpicture}
\end{minipage}%
\begin{minipage}{.33\textwidth}
\begin{tikzpicture}[scale=0.6,squarednode/.style={rectangle, draw=blue!60, fill=blue!5, very thick, minimum size=5mm}, font=\tiny]

	\node[squarednode] at (-3,3) (maintopic) {2.3} ; 
    \draw (-2.1,2) -- (0.1,2.2);
	\draw (-1.9,2.1) -- (-3.2,-0.1);
	\draw (-3.2,0.1) -- (-1.9,-2.1);
	\draw (-2.1,-2) -- (0.1,-2);
	\draw (1.2,1) -- (1.2,-1);
	\draw (-0.2,2.3) -- (1.3,0.7);

	\node at (-1,2.5) {$E_{1}-E_{3}$};
	\node at (-3.9,1) {$B-E_{1}-E_{2} $};
	\node at (-3.1,-1) {$E_{2} $};
	\node at (-1,-2.3) {$F-E_{2}$};
	\node at (-0.1, 0) {$F-E_{1}-E_{3}$};
	\node at (0.9,1.7) {$E_{3}$};

\end{tikzpicture}

\end{minipage}

\bigskip

\begin{minipage}{.5\textwidth}

\begin{center}
\begin{tikzpicture}[scale=0.6,squarednode/.style={rectangle, draw=blue!60, fill=blue!5, very thick, minimum size=5mm}, font=\tiny]

	\node[squarednode] at (-3,3) (maintopic) {2.4} ; 
    \draw (-2.1,2) -- (0.1,2.2);
	\draw (-1.9,2.1) -- (-3.2,-0.1);
	\draw (-3.2,0.1) -- (-1.9,-2.1);
	\draw (-2.1,-2) -- (0.1,-2);
	\draw (1.2,1) -- (1.2,-1);
	\draw (-0.2,2.3) -- (1.3,0.7);

	\node at (-1,2.3) {$E_{3}$};
	\node at (-4.4,1.2) {$B-E_{1}-E_{2}-E_{3} $};
	\node at (-3.1,-1) {$E_{2} $};
	\node at (-1,-2.3) {$F-E_{2}$};
	\node at (2, 0) {$F-E_{1}$};
	\node at (1.4,1.7) {$E_{1}-E_{3}$};

\end{tikzpicture}
\end{center}
\end{minipage}%
\begin{minipage}{.5\textwidth}
\begin{center}
\begin{tikzpicture}[scale=0.6,squarednode/.style={rectangle, draw=blue!60, fill=blue!5, very thick, minimum size=5mm}, font=\tiny]

	\node[squarednode] at (-3,3) (maintopic) {2.5} ; 
    \draw (-2.1,2) -- (0.1,2.2);
	\draw (-1.9,2.1) -- (-3.2,-0.1);
	\draw (-3.2,0.1) -- (-1.9,-2.1);
	\draw (-2.1,-2) -- (0.1,-2);
	\draw (1.2,1) -- (1.2,-1);
	\draw (-0.2,2.3) -- (1.3,0.7);

	\node at (-0.6,2.5) {$B-E_{1}-E_{2}-E_{3}$};
	\node at (-3.5,1) {$E_{3} $};
	\node at (-3.4,-1.1) {$E_{2}-E_{3}$};
	\node at (-1,-2.3) {$F-E_{2}$};
	\node at (0.4, 0) {$F-E_{1}$};
	\node at (0.9,1.7) {$E_{1}$};

\end{tikzpicture}
\end{center}
\end{minipage}

\caption{Configurations 2.1 -- 2.5}
\label{Config3 of AP blown up 1-5}
\end{figure}
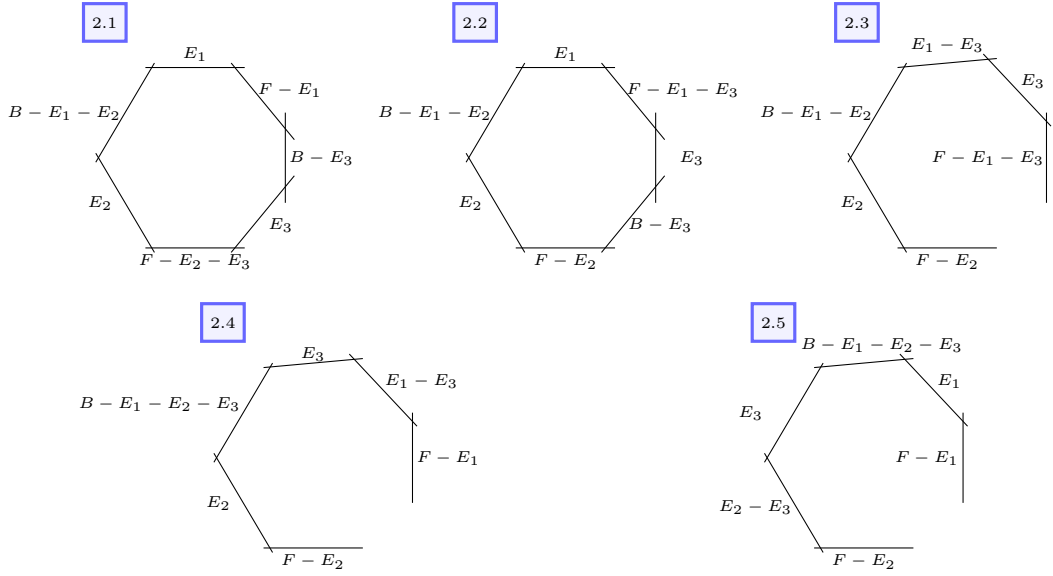

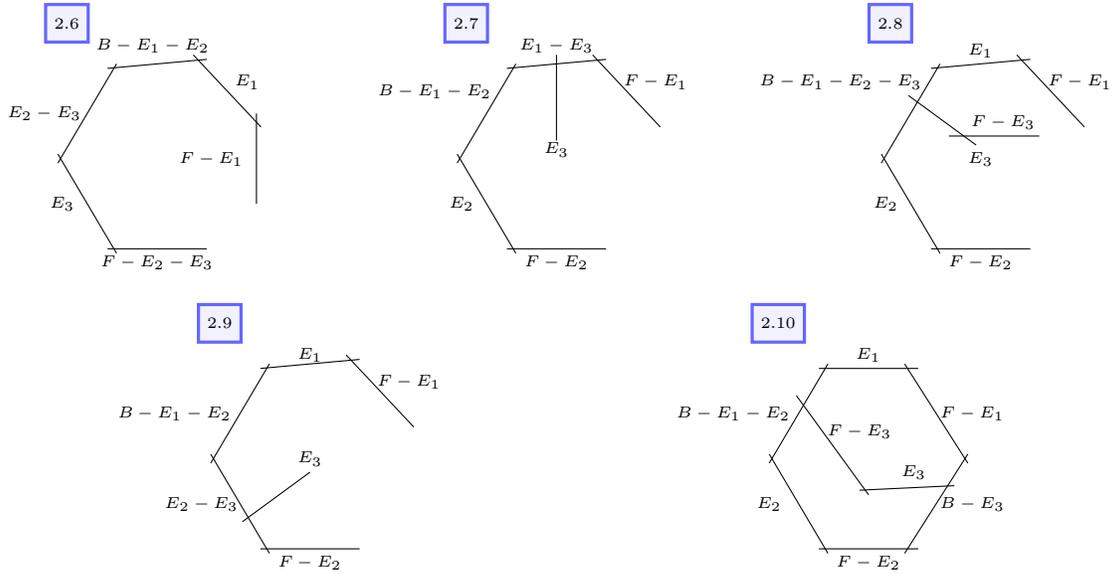
\begin{figure}[thp]

\begin{minipage}{.33\textwidth}

\begin{tikzpicture}[scale=0.6,squarednode/.style={rectangle, draw=blue!60, fill=blue!5, very thick, minimum size=5mm}, font=\tiny]

	\node[squarednode] at (-3,3) (maintopic) {2.6} ; 
    \draw (-2.1,2) -- (0.1,2.2);
	\draw (-1.9,2.1) -- (-3.2,-0.1);
	\draw (-3.2,0.1) -- (-1.9,-2.1);
	\draw (-2.1,-2) -- (0.1,-2);
	\draw (1.2,1) -- (1.2,-1);
	\draw (-0.2,2.3) -- (1.3,0.7);

	\node at (-1.1,2.5) {$B-E_{1}-E_{2}$};
	\node at (-3.5,1) {$E_{2}-E_{3} $};
	\node at (-3.1,-1) {$E_{3} $};
	\node at (-1,-2.3) {$F-E_{2}-E_{3}$};
	\node at (0.2, 0) {$F-E_{1}$};
	\node at (1,1.7) {$E_{1}$};

\end{tikzpicture}
\end{minipage}%
\begin{minipage}{.34\textwidth}
\begin{tikzpicture}[scale=0.6,squarednode/.style={rectangle, draw=blue!60, fill=blue!5, very thick, minimum size=5mm}, font=\tiny]

	\node[squarednode] at (-3,3) (maintopic) {2.7} ; 
    \draw (-2.1,2) -- (0.1,2.2);
	\draw (-1.9,2.1) -- (-3.2,-0.1);
	\draw (-3.2,0.1) -- (-1.9,-2.1);
	\draw (-2.1,-2) -- (0.1,-2);
	\draw (-0.2,2.3) -- (1.3,0.7);
	\draw (-1,0.4) -- (-1,2.3);

	\node at (-1,2.5) {$E_{1}-E_{3}$};
	\node at (-3.7,1.5) {$B-E_{1}-E_{2} $};
	\node at (-3.1,-1) {$E_{2} $};
	\node at (-1,-2.3) {$F-E_{2}$};
	\node at (1.2,1.7) {$F-E_{1}$};
	\node at (-1, 0.2) {$E_{3} $};

\end{tikzpicture}
\end{minipage}%
\begin{minipage}{.33\textwidth}

\begin{tikzpicture}[scale=0.6,squarednode/.style={rectangle, draw=blue!60, fill=blue!5, very thick, minimum size=5mm}, font=\tiny]

	\node[squarednode] at (-3,3) (maintopic) {2.8} ; 
    \draw (-2.1,2) -- (0.1,2.2);
	\draw (-1.9,2.1) -- (-3.2,-0.1);
	\draw (-3.2,0.1) -- (-1.9,-2.1);
	\draw (-2.1,-2) -- (0.1,-2);
	\draw (-0.2,2.3) -- (1.3,0.7);
	\draw (-1.1,0.3) -- (-2.6,1.4);
	\draw (-1.7,0.5) -- (0.3,0.5);
	
	\node at (-1,2.4) {$E_{1}$};
	\node at (-4.1,1.7) {$B-E_{1}-E_{2}-E_{3} $};
	\node at (-3.1,-1) {$E_{2} $};
	\node at (-1,-2.3) {$F-E_{2}$};
	\node at (1.2,1.7) {$F-E_{1}$};
	\node at (-1, 0) {$E_{3} $};
	\node at (-0.5,0.8) {$F-E_{3}$};

\end{tikzpicture}
\end{minipage}

\bigskip

\begin{minipage}{.5\textwidth}
\begin{center}
\begin{tikzpicture}[scale=0.6,squarednode/.style={rectangle, draw=blue!60, fill=blue!5, very thick, minimum size=5mm}, font=\tiny]

	\node[squarednode] at (-3,3) (maintopic) {2.9} ; 
    \draw (-2.1,2) -- (0.1,2.2);
	\draw (-1.9,2.1) -- (-3.2,-0.1);
	\draw (-3.2,0.1) -- (-1.9,-2.1);
	\draw (-2.1,-2) -- (0.1,-2);
	\draw (-0.2,2.3) -- (1.3,0.7);

	\node at (-1,2.3) {$E_{1}$};
	\node at (-4,1) {$B-E_{1}-E_{2} $};
	\node at (-3.4,-1) {$E_{2}-E_{3} $};
	\node at (-1,-2.3) {$F-E_{2}$};
	\node at (1.2,1.7) {$F-E_{1}$};

	\draw (-1,-0.3) -- (-2.5,-1.4);

	\node at (-1, 0) {$E_{3} $};

\end{tikzpicture}
\end{center}
\end{minipage}%
\begin{minipage}{.5\textwidth}
\begin{center}
\begin{tikzpicture}[scale=0.6,squarednode/.style={rectangle, draw=blue!60, fill=blue!5, very thick, minimum size=5mm}, font=\tiny]

	\node[squarednode] at (-3,3) (maintopic) {2.10} ; 

    \draw (-2.1,2) -- (0.1,2);
	\draw (-2.1,-2) -- (0.1,-2);
	\draw (-0.2,-2.1) -- (1.2,0.1);
	\draw (-0.2,2.1) -- (1.2,-0.1);
	\draw (-1.9,2.1) -- (-3.2,-0.1);
	\draw (-3.2,0.1) -- (-1.9,-2.1);
	\draw (-1.2,-0.7) -- (0.9,-0.6);
	\draw (-1,-0.8) -- (-2.6,1.4);

	\node at (-1,2.3) {$E_{1}$};
	\node at (1.3,1) {$F-E_{1}$};
	\node at (-4,1) {$B-E_{1}-E_{2} $};
	\node at (-3.2,-1) {$E_{2}$};
	\node at (-1,-2.3) {$F-E_{2}$};
	\node at (1.3,-1) {$B-E_{3}$};
	\node at (0, -0.3) {$E_{3} $};
	\node at (-1.2,0.6) {$F-E_{3} $};

\end{tikzpicture}
\end{center}
\end{minipage}

\caption{Configurations 2.6 -- 2.10}
\label{Config3 of AP blown up 6-10}
\end{figure}

Then  configuration (5) of Figure 2 in \cite{AnjPin}, reproduced in Figure \ref{Config5 of AP}, gives configurations 3.1 to 3.8 represented in Figures \ref{Config5 of AP blown up 1-4} and \ref{Config5 of AP blown up 5-8}. 

\begin{figure}[thp]
\begin{center}
\begin{tikzpicture}[scale=0.7,squarednode/.style={rectangle, draw=blue!60, fill=blue!5, very thick, minimum size=5mm}, font=\scriptsize]

	\node[squarednode] at (-3,3) (maintopic) {3} ;
    \draw (-1.4,2) -- (1.4,2);
	\draw (-0.8,2.4) -- (-2.4,-0.2);
	\draw (-2.4,0.2) -- (-0.8,-2.4);
	\draw (-1.4,-2) -- (1.4,-2);

	\node at (0.4,2.2) {$B-E_{1}-E_{2}$};
	\node at (-2.3,1) {$E_{2} $};
	\node at (-2.4,-1.2) {$E_{1}-E_{2} $};
	\node at (0,-2.3) {$F-E_{1} $};

\end{tikzpicture}
\end{center}

\caption{Configuration (5) in \cite{AnjPin} for $J$-homolorphic curves in $\Mcc$. }
\label{Config5 of AP}
\end{figure}

\begin{figure}[thp]

\begin{minipage}{.5\textwidth}
\begin{center}
\begin{tikzpicture}[scale=0.6, squarednode/.style={rectangle, draw=blue!60, fill=blue!5, very thick, minimum size=5mm}, font=\tiny]

	\node[squarednode] at (-3,3) (maintopic) {3.1} ; 
    \draw (-2.1,2) -- (0.1,2);
	\draw (-1.9,2.1) -- (-3.2,-0.1);
	\draw (-3.2,0.1) -- (-1.9,-2.1);
	\draw (-2.1,-2) -- (0.1,-2);
	\draw (-0.2,-2.1) -- (1.2,-0.4);
	\draw (1,1) -- (1,-1);
	\draw (-0.2,2.1) -- (1.2,0.4);

	\node at (-1,2.3) {$B-E_{1}-E_{2}$};
	\node at (1.2,1.5) {$F-E_{3}$};
	\node at (-3.1,1) {$E_{2} $};
	\node at (-3.4,-1) {$E_{1}-E_{2} $};
	\node at (-1,-2.3) {$F-E_{1}$};
	\node at (1.1, -1.5) {$B-E_{3}$};
	\node at (1.5,0) {$E_{3}$};

\end{tikzpicture}
\end{center}
\end{minipage}%
\begin{minipage}{.5\textwidth}
\begin{center}
\begin{tikzpicture}[scale=0.6,squarednode/.style={rectangle, draw=blue!60, fill=blue!5, very thick, minimum size=5mm}, font=\tiny]

	\node[squarednode] at (-3,3) (maintopic) {3.2} ; 
    \draw (-2.1,2) -- (0.1,2);
	\draw (-1.9,2.1) -- (-3.2,-0.1);
	\draw (-3.2,0.1) -- (-1.9,-2.1);
	\draw (-2.1,-2) -- (0.1,-2);
	\draw (-0.2,-2.1) -- (1.2,-0.4);
	\draw (1,1) -- (1,-1);

	\node at (-1,2.3) {$B-E_{1}-E_{2}$};
	\node at (-3.1,1) {$E_{2}$};
	\node at (-3.4,-1) {$E_{1}-E_{2}$};
	\node at (-1,-2.3) {$F-E_{1}-E_{3}$};
	\node at (0.8, -1.5) {$E_{3}$};
	\node at (1.8,0) {$B-E_{3}$};

\end{tikzpicture}
\end{center}
\end{minipage}

\bigskip

\begin{minipage}{.5\textwidth}
\begin{center}
\begin{tikzpicture}[scale=0.6,squarednode/.style={rectangle, draw=blue!60, fill=blue!5, very thick, minimum size=5mm}, font=\tiny]

	\node[squarednode] at (-3,3) (maintopic) {3.3} ; 
    \draw (-2.1,2) -- (0.1,2);
	\draw (-2.1,-2) -- (0.1,-2);
	\draw (-0.2,-2.1) -- (1.2,0.1);
	\draw (-1.9,2.1) -- (-3.2,-0.1);
	\draw (-3.2,0.1) -- (-1.9,-2.1);

	\node at (-1,2.3) {$B-E_{1}-E_{2}$};
	\node at (-3.2,1) {$E_{2}$};
	\node at (-3.8,-1.2) {$E_{1}-E_{2}-E_{3}$};
	\node at (-0.9,-2.3) {$E_{3}$};
	\node at (1.9, -1) {$F-E_{1}-E_{3}$};

\end{tikzpicture}
\end{center}
\end{minipage}%
\begin{minipage}{.5\textwidth}
\begin{center}
\begin{tikzpicture}[scale=0.6,squarednode/.style={rectangle, draw=blue!60, fill=blue!5, very thick, minimum size=5mm}, font=\tiny]

	\node[squarednode] at (-3,3) (maintopic) {3.4} ; 
    \draw (-2.1,2) -- (0.1,2);
	\draw (-2.1,-2) -- (0.1,-2);
	\draw (-0.2,-2.1) -- (1.2,0.1);
	\draw (-1.9,2.1) -- (-3.2,-0.1);
	\draw (-3.2,0.1) -- (-1.9,-2.1);

	\node at (-1,2.3) {$B-E_{1}-E_{2}$};
	\node at (-3.5,1) {$E_{2}-E_{3}$};
	\node at (-3.2,-1) {$E_{3}$};
	\node at (-1,-2.3) {$E_{1}-E_{2}-E_{3}$};
	\node at (1.3, -1) {$F-E_{1}$};

\end{tikzpicture}
\end{center}
\end{minipage}

\caption{Configurations 3.1 -- 3.4}
\label{Config5 of AP blown up 1-4}
\end{figure}
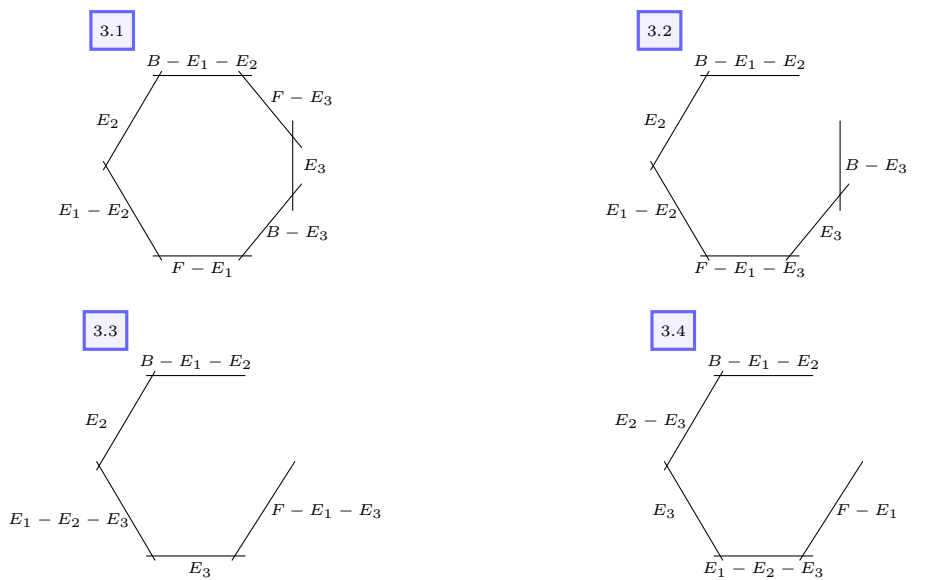

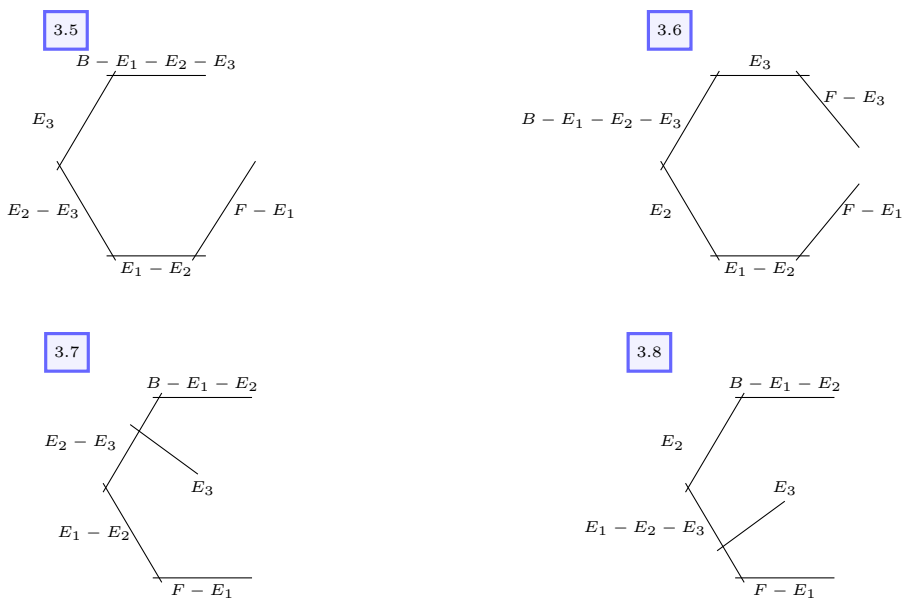
\begin{figure}[thp]

\bigskip

\begin{minipage}{.5\textwidth}
\begin{center}
\begin{tikzpicture}[scale=0.6,squarednode/.style={rectangle, draw=blue!60, fill=blue!5, very thick, minimum size=5mm}, font=\tiny]

	\node[squarednode] at (-3,3) (maintopic) {3.5} ; 
    \draw (-2.1,2) -- (0.1,2);
	\draw (-2.1,-2) -- (0.1,-2);
	\draw (-0.2,-2.1) -- (1.2,0.1);
	\draw (-1.9,2.1) -- (-3.2,-0.1);
	\draw (-3.2,0.1) -- (-1.9,-2.1);

	\node at (-1,2.3) {$B-E_{1}-E_{2}-E_{3}$};
	\node at (-3.5,1) {$E_{3}$};
	\node at (-3.5,-1) {$E_{2}-E_{3}$};
	\node at (-1,-2.3) {$E_{1}-E_{2}$};
	\node at (1.4, -1) {$F-E_{1}$};

\end{tikzpicture}
\end{center}
\end{minipage}%
\begin{minipage}{.5\textwidth}
\begin{center}
\begin{tikzpicture}[scale=0.6,squarednode/.style={rectangle, draw=blue!60, fill=blue!5, very thick, minimum size=5mm}, font=\tiny]

	\node[squarednode] at (-3,3) (maintopic) {3.6};
    \draw (-2.1,2) -- (0.1,2);
	\draw (-1.9,2.1) -- (-3.2,-0.1);
	\draw (-3.2,0.1) -- (-1.9,-2.1);
	\draw (-2.1,-2) -- (0.1,-2);
	\draw (-0.2,-2.1) -- (1.2,-0.4);
	\draw (-0.2,2.1) -- (1.2,0.4);

	\node at (-1,2.3) {$E_{3}$};
	\node at (-4.5,1) {$B-E_{1}-E_{2}-E_{3}$};
	\node at (-3.2,-1) {$E_{2}$};
	\node at (-1,-2.3) {$E_{1}-E_{2}$};
	\node at (1.5, -1) {$F-E_{1}$};
	\node at (1.1,1.5) {$F-E_{3}$};

\end{tikzpicture}
\end{center}

\end{minipage}

\bigskip 

\bigskip

\begin{minipage}{.5\textwidth}
\begin{center}
\begin{tikzpicture}[scale=0.6,squarednode/.style={rectangle, draw=blue!60, fill=blue!5, very thick, minimum size=5mm}, font=\tiny]

	\node[squarednode] at (-4,3) (maintopic) {3.7} ;

    \draw (-2.1,2) -- (0.1,2);
	\draw (-2.1,-2) -- (0.1,-2);
	\draw (-1.9,2.1) -- (-3.2,-0.1);
	\draw (-3.2,0.1) -- (-1.9,-2.1);
	\draw (-1.1,0.3) -- (-2.6,1.4);

	\node at (-1,2.3) {$B-E_{1}-E_{2}$};
	\node at (-3.7,1) {$E_{2}-E_{3} $};
	\node at (-3.4,-1) {$E_{1}-E_{2}$};
	\node at (-1,-2.3) {$F-E_{1}$};
	\node at (-1, 0) {$E_{3} $};

\end{tikzpicture}
\end{center}
\end{minipage}%
\begin{minipage}{.5\textwidth}
\begin{center}
\begin{tikzpicture}[scale=0.6,squarednode/.style={rectangle, draw=blue!60, fill=blue!5, very thick, minimum size=5mm}, font=\tiny]

	\node[squarednode] at (-4,3) (maintopic) {3.8} ;

    \draw (-2.1,2) -- (0.1,2);
	\draw (-2.1,-2) -- (0.1,-2);
	\draw (-1.9,2.1) -- (-3.2,-0.1);
	\draw (-3.2,0.1) -- (-1.9,-2.1);
	\draw (-1,-0.3) -- (-2.5,-1.4);

	\node at (-1,2.3) {$B-E_{1}-E_{2}$};
	\node at (-3.5,1) {$E_{2} $};
	\node at (-4.1,-0.9) {$E_{1}-E_{2}-E_{3}$};
	\node at (-1,-2.3) {$F-E_{1}$};
	\node at (-1, 0) {$E_{3} $};

\end{tikzpicture}
\end{center}
\end{minipage}

\caption{Configurations 3.5 -- 3.8}
\label{Config5 of AP blown up 5-8}
\end{figure}

Next we look at configuration (4) of Figure 2 in \cite{AnjPin}, which is reproduced in Figure \ref{Config42 of AP}. Note that this configuration coincides with the one in Figure \ref{Config5 of AP} if we interchange the role of $B$ and $F$. Therefore, blowing-up we will get configurations 4.1 -- 4.8 which can be obtained simply from 3.1 -- 3.8 by interchanging also the roles of $B$ and $F$. 

Similarly from configuration (2) of Figure 2 in \cite{AnjPin}, represented in Figure \ref{Config42 of AP}, we obtain configurations 5.1 -- 5.10 by interchanging the roles of $B$ and $F$ in configurations 2.1 -- 2.10. 

\begin{figure}[thp]

\begin{minipage}{.5\textwidth}
\begin{center}
\begin{tikzpicture}[scale=0.6,squarednode/.style={rectangle, draw=blue!60, fill=blue!5, very thick, minimum size=5mm}, font=\scriptsize]

	\node[squarednode] at (-3,3) (maintopic) {4} ;
    \draw (-1.4,2) -- (1.4,2);
	\draw (-0.8,2.4) -- (-2.4,-0.2);
	\draw (-2.4,0.2) -- (-0.8,-2.4);
	\draw (-1.4,-2) -- (1.4,-2);

	\node at (0.4,2.3) {$B-E_{1}$};
	\node at (-2.6,1) {$E_{1}-E_{2}$};
	\node at (-2.2,-1.2) {$E_{2} $};
	\node at (0.5,-2.3) {$F-E_{1}-E_{2}$};

\end{tikzpicture}
\end{center}
\end{minipage}%
\begin{minipage}{.5\textwidth}
\begin{center}
\begin{tikzpicture}[scale=0.6,squarednode/.style={rectangle, draw=blue!60, fill=blue!5, very thick, minimum size=5mm}, font=\scriptsize]

	\node[squarednode] at (-3,3) (maintopic) {5} ;

    \draw (-1.4,2) -- (1.5,2.5);
	\draw (-0.8,2.4) -- (-2.4,-0.2);
	\draw (-2.4,0.2) -- (-0.8,-2.4);
	\draw (-1.4,-2) -- (1.4,-2);
	\draw (0.8,2.7) -- (2.9,0.7);

	\node at (-0.1,2.5) {$E_{2}$};
	\node at (-3,1) {$F-E_{1}-E_{2} $};
	\node at (-2.1,-1) {$E_{1} $};
	\node at (0.2,-2.3) {$B-E_{1} $};
	\node at (2.6,2) {$B-E_{2}$};

\end{tikzpicture}
\end{center}
\end{minipage}
\caption{Configurations (4) and (2), respectively,  in \cite{AnjPin}  of $J$-holomorphic curves in $\Mcc$}
\label{Config42 of AP}

\end{figure}
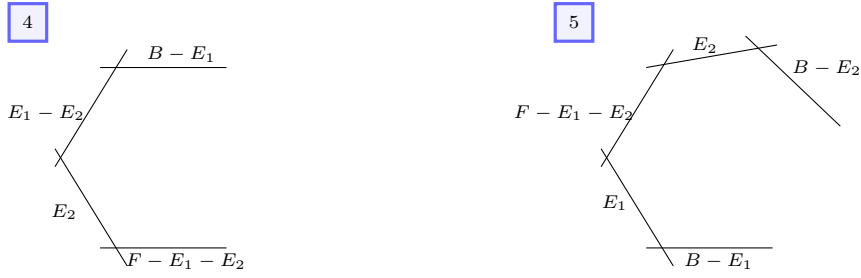

Finally, using configuration (6) of Figure 2 in \cite{AnjPin},  blowing-up at the points indicated in Figure \ref{Config6 of AP}, we obtain eight new configurations of $J$-holomorphic curves  in $\Mccc$, denoted by 6.1 -- 6.8. At this point it should be clear how to draw these eight configurations, so we do not show them in here. 

\begin{figure}[thp]

\begin{center}
\begin{tikzpicture}[scale=0.7,squarednode/.style={rectangle, draw=blue!60, fill=blue!5, very thick, minimum size=5mm}, font=\scriptsize]

	\node[squarednode] at (-3,3) (maintopic) {6} ;

    \draw (-1.4,2) -- (1.4,2);
	\draw (-1.2,2.2) -- (-1.2,-2.2);
	\draw (-1.4,-2) -- (1.4,-2);
	\draw (-1.4,0) -- (1.4,0);

	\node at (0.1,2.3) {$F-E_{1}$};
	\node at (-2.2,1) {$E_{1}-E_{2} $};
	\node at (0.1,-2.3) {$B-E_{1} $};
	\node at (1,0.3) {$E_{2}$};
	\node at (2.3,0){$\color{red} 1$};
	\node at (2,0){$\color{red} \bullet $};
	\node at (1.2,2.3){$\color{red} 2$};
	\node at (1.2,2){$\color{red} \bullet $};
	\node at (-1.5,2.3){$\color{red} 3$};
	\node at (-1.2,2){$\color{red} \bullet $};
	\node at (-1.5,0.3){$\color{red} 4$};
	\node at (-1.2,0){$\color{red} \bullet $};
	\node at (0,0.3){$\color{red} 5$};
	\node at (0,0){$\color{red} \bullet $};
	\node at (-1.5,-2.3){$\color{red} 6$};
	\node at (-1.2,-2){$\color{red} \bullet $};
	\node at (1.2,-1.7){$\color{red} 7$};
	\node at (1.2,-2){$\color{red} \bullet $};
	\node at (-1.5,-1){$\color{red} 8$};
	\node at (-1.2,-1){$\color{red} \bullet $};

\end{tikzpicture}
\end{center}
\caption{Configuration (6) in \cite{AnjPin} of  $J$-holomorphic curves in $\Mcc$.}
\label{Config6 of AP}
\end{figure}

Let $\Jcccm$ be the space of almost complex structures $J \in \Jccc$ such that the class $D_{-m}$,  $m=0,1,2,3$, is represented by an embedded $J$-holomorphic sphere. Moreover, let $\mathcal{J}_1$, $\mathcal{J}_{12,23}$ and $\mathcal{J}_{123}$ be, respectively, the spaces of almost complex structures 
$J \in \Jccc$ for which: at least two of the classes $B-E_i$, $i=1,2,3$; $E_1-E_2$ and $E_2-E_3$; and $E_1- E_2 -E_3$ are represented by embedded $J$-holomorphic spheres. The above considerations and Lemma \ref{J-holomorphics m1} show that the space $\Jccc$ is the disjoint union of the subspaces
\begin{multline}\label{disjointunion}
\Jccc = \mathcal{J}_{c_1,c_2,c_3,3} \sqcup  \mathcal{J}_{c_1,c_2,c_3,2(1)}  \sqcup  \mathcal{J}_{c_1,c_2,c_3,2(2)}  \sqcup  \mathcal{J}_{c_1,c_2,c_3,2(3)} \\
 \sqcup  \mathcal{J}_{1}  \sqcup  \left( \left(\mathcal{J}_{c_1,c_2,c_3,1(1)} \cap  \mathcal{J}_{123}\right) \cup \left(\mathcal{J}_{c_1,c_2,c_3,1(1)} \cap  \mathcal{J}_{12,23}\right) \right).
\end{multline}
Note that the last space in \eqref{disjointunion} is simply the union of the strata of almost complex structures $J$ for which the curves $B-E_1$ and $E_1-E_2-E_3$ are both represented by  $J$-holomorphic spheres with the strata for which the curves $B-E_1$ and $E_1-E_2$ and $E_2-E_3$ are all represented. 

\begin{lemma}\label{existenceJ-curves}
Given a tamed almost complex structure $J$ in one of the strata  of the disjoint union  \eqref{disjointunion}, a class $D_i$ $ i\neq m$, is represented by some embedded $J$-holomorphic sphere iff $D_i \cdot D_{-m} \geq 0$, $m=0,1,2,3$. In particular, the classe $D_i$ always has embedded $J$-holomorphic representatives for all $J \in \Jccc$ except if $A$, given as in the following list,  is also represented by an embedded $J$-holomorphic sphere:
\begin{itemize}
\item $i=4k-1(2)$ and $A$ is equal to $E_1-E_2$ or $E_1-E_2-E_3$;
\item $i=4k-1(3)$ and $A$  is  equal to $E_1-E_3$,  $E_2-E_3$ or $E_1-E_2-E_3$;
\item $i=4k-2(1)$ and $A$  is equal to $ F-E_2-E_3$,  $E_1-E_2-E_3$, $E_1-E_2$ or  $E_1-E_3$;
\item $i=4k-2(2)$ and $A$  is equal to $ F-E_1-E_3$,   or  $E_2-E_3$;
\item $i=4k-2(3)$ and $A= F-E_1-E_2$;
\item $i=4k-3$ and $A$  is equal to $ F-E_1-E_2-E_3$, $F-E_i-E_j$ or $E_1-E_2-E_3$.
\end{itemize}
\end{lemma}
\begin{proof}
The proof is an adaptation of the proof of Lemma 2.6 in \cite{Pin}, so we just recall the main steps. By positivity of intersections if $D_i$ is represented by an embedded $J$-holomorphic curve then we have $D_i \cdot D_{-m} \geq 0$. On the other hand, given $J$ in one of the strata of the decomposition \eqref{disjointunion}, consider $D_i$ such that $D_i \cdot D_{-m} \geq 0$ and $D_i$ does not have any embedded $J$-holomorphic representative. This implies that $i \geq 0$ and therefore $\mathrm{Gr} (D_i) \neq 0$. Hence there must be a $J$-holomorphic cusp-curve representing $D_i$ and passing through $k=k(D_i)$  generic points $\{ p_1, \hdots, p_k\}$. Such a curve 
$$ C = \bigcup_{i=1}^N m_iC_i \quad \mbox{where}  \quad N \geq 2$$
defines a decomposition of $D_i$ as a sum of some $D_r$ and multiples of the classes $E_i$, $E_i-E_j$, $E_1-E_2-E_3$, $F$, $F-E_i$, $ F-E_i-E_j$ and $ F-E_1-E_2-E_3$, where the multiplicities of each curve are all nonnegative and $D_r \cdot D_{-m} \geq 0$ unless $r=-m$.  Then we remove all but one copy of the repeated components of $C$ and replace multiply-covered components by their underlying simple cover in order to obtain the reduced cusp curve $\bar{C}$. Note that the curves $C$ and $\bar{C}$ have the same image and therefore contain the same subset of special points in $\{ p_1, \hdots, p_k\}$. Now the idea is to show that the dimension of the space of reduced cusp curves is too small. More precisely, 
$$\dim \mathcal{M}(\bar{C}, J) = 2c_1(\bar{C}) -2N,$$
where $N$ is the number of components of $\bar{C}$. Since $c_1(\bar{C}) \leq c_1(D_i)$ it follows that 
$$ \dim \mathcal{M}(\bar{C}, J) < 2k(D_i)$$ 
and this implies that no such cusp-curve can pass through $k(D_i)$ points. The list of exceptions follows from the fact that for each pair $D_i$ and $A$ in the list we have $D_i \cdot A < 0 $, which is not allowed by positivity of intersections. 
\end{proof}

\section{Homotopy type of symplectomorphism groups}\label{chp m1: homotopy type}

In this section first we recall a result by J. Li, T. Li and W. Wu about the mapping class group of $\Guccc$, then state the main result about the stability of symplectomorphisms and study the homotopy groups of $\Gccc$. In particular we obtain an estimate for the rank of its fundamental group. In Subsections \ref{sec m1: generators} and \ref{sec m1:Samelsonproducts} we then study the generators of the fundamental group and their Samelson products, respectively. Putting all this together, finally in Subsection \ref{proof m1}, we state Theorem \ref{mainthm m1} in full detail and prove it. 

\begin{proposition} (\cite[Theorem 1.1]{LiLiWu})
The space $ G_{\mu , c_{1} , c_{2}, c_{3} } $ is connected for any $ \mu \geq 1 $.
\end{proposition}
\begin{proof} (Outline)
We will give the main ingredients of the proof. We start by establishing the basic concepts, definitions and notation.

 A \textit{stable (spherical symplectic) configuration} in a symplectic manifold $ (M, \omega) $ is an ordered configuration of symplectic spheres $$ C = \bigcup \limits_{i=1}^{N} C_{i}$$ such that
\begin{itemize}
\item $ [C_{i}] \cdot [C_{i}] \geq 1 $ for all $ i $,
\item for $ i \neq j $, $ [C_{i}] \neq [C_{j}] $ and $ [C_{i}] \cdot [C_{j}] $ is zero or one,
\item there is a compatible almost complex structure $ J $ such that all the $ C_{i} $'s are $ J $-holomorphic.
\end{itemize}
A stable configuration is \textit{standard} if the $ C_{i} $'s intersect $ \omega $-orthogonally at every intersection point of the configuration. We denote by $ \mathcal{C}_{0} $ the space of all standard stable configurations (see \cite{Eva} for more details on this configuration space, namely, its topology). A standard stable configuration $ C \in \mathcal{C}_{0} $ is \textit{full} if $ H^{2}(M,C;\mathbb{R}) = 0 $.
 Note that the configuration 1.13 is full,  standard and stable. From now on, we will restrict our attention to this choice of $ C $ as our working configuration for the case $ \mu =1 $.

The first step in the proof is to show that the group $ \text{Symp} _{c} (U) $ of compactly supported symplectomorphisms of $ (U = M \setminus C , \omega | _{U} )$ is contractible. By Moser's theorem, $ \text{Symp} _{c} (U) $ is homotopy equivalent to the space $ \text{Stab}^{1}(C) $ of symplectomorphisms on $ M $ that fix $ C $ pointwise and that act trivially on the normal bundles of the $ C_{i} $'s.

Next, we consider the fibration $ \text{Stab}^{1}(C) \rightarrow \text{Stab}^{0}(C) \rightarrow \mathcal{G}(C) $, where $ \text{Stab}^{0}(C) $ is the set of symplectomorphisms on $ M $ that fix $ C $ pointwise, and $\mathcal{G}(C) $ is defined as follows: let $ \mathcal{G}_{k_{i}}(C_{i}) $ denote the group of gauge transformations of the normal bundle to $ C_{i} $ which equal the identity at the $ k_{i} $ intersection points. In our case, each of these are homotopy equivalent to $ \mathcal{G}_{k_{i}}(S^{2}) $. Then, for $ C = \bigcup \limits_{i=1}^{N} C_{i}$, we define $\mathcal{G}(C) = \prod_{i=1}^{N} \mathcal{G}_{k_{i}}(C_{i}) $. So, by the first step, we obtain $ \text{Stab}^{0}(C) \simeq \mathcal{G}(C) $.

The third step is to consider the set $ \text{Stab}(C) $ of symplectomorphisms on $ M $ that fix $ C $ as a set. For our choice of $ C $, it can be shown that $ \text{Stab}^{0}(C) \rightarrow \text{Stab}(C) \rightarrow \text{Symp}(C) $ is a homotopy fibration. Considering the long exact sequence
$$ ... \rightarrow \pi_{1}(\text{Stab}(C)) \rightarrow \pi_{1} (\text{Symp}(C)) \xrightarrow{\psi} \pi_{0}(\text{Stab}^{0}(C)) \rightarrow \pi_{0}(\text{Stab}(C)) \rightarrow \pi_{0}(\text{Symp}(C)),$$
it is possible to show that $ \psi $ is surjective and $ \text{Stab}(C) $ is connected.

Finally, we consider $ \text{Stab}(C) \rightarrow \text{Symp}_{h}(M,\omega) \rightarrow \mathcal{C}_{0} $. Here, the action of $ \text{Symp}_{h}(M,\omega) $ on $ \mathcal{C}_{0} $ is transitive and $ \mathcal{C}_{0} $ is connected. Since by the third step $ \text{Stab}(C) $ is connected, the homotopy fibration yields that $ \text{Symp}_{h}(M,\omega) $ is also connected. This finishes the proof of the proposition.

To summarize, the proof relies on the analysis of the following diagram, for a clever choice of $ C $:

\begin{displaymath}
    \xymatrix{
         \text{Symp}_{c}(U) \ar@{->}[r] & \text{Stab}^{1}(C) \ar@{->}[r] & \text{Stab}^{0}(C) \ar@{->}[d] \ar@{->}[r] & \text{Stab}(C) \ar@{->}[d] \ar@{->}[r] & \text{Symp}_{h}(M,\omega) \ar@{->}[d] \\
          &  &  \mathcal{G}(C) & \text{Symp}(C) & \mathcal{C}_{0}.}
\end{displaymath}

In fact, this diagram gives a uniform approach to show the connectedness of the symplectomorphism group  $ \text{Symp}_{h}(M,\omega) $, for $ M $ a n-fold blow-up of $ \mathbb{CP}^{2} $ for $ n \leq 4 $. For each case, the authors give the configuration that can be used to run the proof (\cite[Remark 3.4]{LiLiWu}).
\end{proof}

\subsection{Stability of symplectomorphism groups}\label{sec m1: stability}

\begin{theorem}\label{slight changes m1}
Consider $c_{1} , c_{2}, c_{3}, c'_{1}, c'_{2}, c'_{3}$ such that
\begin{center}
$ 0<c_{3} < c_{2} < c_{1} < c_{1} + c_{3} < c_{1} + c_{2} < 1$,

$ 0<c'_{3} < c'_{2} < c'_{1} < c'_{1} + c'_{3} < c'_{1} + c'_{2} < 1$,
\end{center}
Then the symplectomorphism groups $ G_{c_{1} , c_{2}, c_{3} } $ and $ G_{c'_{1} , c'_{2}, c'_{3} } $ are homotopy equivalent.
\end{theorem}
\begin{proof}
The proof of this theorem follows the same line of thought as the proof of Proposition 3.1 in \cite{Pin} and uses the inflation technique introduced by Lalonde and McDuff in \cite{LalMcD}.

We start by considering the natural action of the identity component of the diffeomorphism group of $ \widetilde{M}_{c_{1} , c_{2} , c_{3}}$ on the space $ \mathcal{S}_{c_{1} , c_{2} , c_{3}} $ of all symplectic forms on $ \widetilde{M}_{c_{1} , c_{2} , c_{3}}$ in the cohomology class $ [ \omega_{c_{1} , c_{2} , c_{3}}]$ that are isotopic to $ \omega_{c_{1} , c_{2} , c_{3}}$. By Moser's theorem, this action is transitive, and we obtain the following fibration:
$$ G_{c_{1} , c_{2}, c_{3} } = \text{Symp}(\widetilde{M} _{c_{1} , c_{2} , c_{3}}) \cap \text{Diff} _{0} (\widetilde{M}_{c_{1} , c_{2} , c_{3}}) \longrightarrow \text{Diff} _{0} (\widetilde{M}_{c_{1} , c_{2} , c_{3}}) \longrightarrow \mathcal{S}_{c_{1} , c_{2} , c_{3}}.$$
Doing the same with $c'_{1} , c'_{2}, c'_{3}$, we see that to show the required homotopy equivalence, it is sufficient to find a homotopy equivalence between the spaces $ \mathcal{S}_{c_{1} , c_{2} , c_{3}} $ and $ \mathcal{S}_{c'_{1} , c'_{2} , c'_{3}} $ that makes the below diagram commutative up to homotopy.

\begin{displaymath}
    \xymatrix{
         G_{c_{1} , c_{2}, c_{3} } \ar@{->}[r]  &  \text{Diff} _{0}  \ar@{->}[r] \ar@{=}[d] &  \mathcal{S}_{c_{1} , c_{2} , c_{3}} \ar@{<->}[d]  \\
         G_{c'_{1} , c'_{2}, c'_{3} }  \ar@{->}[r]  &  \text{Diff} _{0}  \ar@{->}[r] &  \mathcal{S}_{c'_{1} , c'_{2} , c'_{3}} }
\end{displaymath}
To this end, we follow an idea introduced by McDuff \cite{McD2} and consider the larger space $ \mathcal{X} _{c_{1} , c_{2} , c_{3}} $ defined as
$$ \mathcal{X} _{c_{1} , c_{2} , c_{3}} = \{ (\omega,J) \in \mathcal{S}_{c_{1} , c_{2} , c_{3}} \times \mathcal{A}_{c_{1} , c_{2} , c_{3}} \ : \ \omega \text{ tames } J \},$$
where $ \mathcal{A}_{c_{1} , c_{2} , c_{3}} $ is the space of almost complex structures that are tamed by some form in $ \mathcal{S}_{c_{1} , c_{2} , c_{3}} $. Then both projection maps $ \mathcal{X}_{c_{1} , c_{2} , c_{3}} \rightarrow \mathcal{A}_{c_{1} , c_{2} , c_{3}} $ and $ \mathcal{X}_{c_{1} , c_{2} , c_{3}} \rightarrow \mathcal{S}_{c_{1} , c_{2} , c_{3}} $ are fibrations with contractible fibers, and therefore they are homotopy equivalences. Below, we will prove that $ \mathcal{A}_{c_{1} , c_{2} , c_{3}} $ and $ \mathcal{A}_{c'_{1} , c'_{2} , c'_{3}} $ are in fact equal. We start by recalling the following result due to Lalonde and McDuff.

\begin{lemma}
[Inflation Lemma, \cite{LalMcD}] Let $ J $ be an $ \tau _{0} $-tame almost complex structure on a symplectic 4-manifold $ (M, \tau _{0} ) $ that admits a $ J $-holomorphic curve $ Z $ with non-negative self-intersection. Then there is a family $ \tau _{t} $, $ t \geq 0 $, of symplectic forms that all tame $ J $ and have cohomology class $ [ \tau _{t} ] = [ \tau _{0} ] + t \text{PD}(Z)  $, where $ \text{PD}(Z) $ is the Poincar\'{e} dual to the homology class $ Z $.
\end{lemma} 

This result was generalized by Buse to curves with negative self-intersection.

\begin{lemma}
[Buse, \cite{Bus}] Let $ J $ be an $ \tau _{0} $-tame almost complex structure on a symplectic 4-manifold $ (M, \tau _{0} ) $ that admits a $ J $-holomorphic curve $ Z $ with $ Z \cdot Z = -m $, $ m \in \mathbb{N} $. Then for all $ \epsilon > 0 $ there is a family $ \tau _{t} $ of symplectic forms that all tame $ J $ and have cohomology class $ [ \tau _{t} ] = [ \tau _{0} ] + t \text{PD}(Z)  $ for all $ 0 \leq t \leq \frac{\tau_{0}(Z)}{m} - \epsilon $, where $ \text{PD}(Z) $ is the Poincar\'{e} dual to the homology class $ Z $.
\end{lemma}

Note that to show the spaces of almost complex structures $ \mathcal{A}_{c_{1} , c_{2} , c_{3}}$ and $ \mathcal{A}_{c'_{1} , c'_{2} , c'_{3}} $ coincide for different choices of generic capacities, it is sufficient to show  the two spaces are the same when two of the capacities coincide: indeed, if we  consider two generic sets of capacities $\{c_3  < c_2 < c_1\}$ and  $\{c_3' < c_2' < c_1'\}$, first, fixing $c_2,c_2'$ and $c_3,c_3'$, we have $\mathcal{A}_{c_{1} , c_{2} , c_{3}} = \mathcal{A}_{c_{M} , c_{2} , c_{3}}$ and $ \mathcal{A}_{c'_{1} , c'_{2} , c'_{3}} =\mathcal{A}_{c_{M} , c'_{2} , c'_{3}}$, where $c_M=\mathrm{max}(c_1,c_1')$. Then fixing $c_2,c_2'$ and $c_M$ we obtain  $\mathcal{A}_{c_{M} , c_{2} , c_{3}} = \mathcal{A}_{c_{M} , c_{2} , c_{m}}$ and $ \mathcal{A}_{c_{M} , c'_{2} , c'_{3}} =\mathcal{A}_{c_{M} , c'_{2} , c_{m}}$ where  $c_m=\mathrm{min}(c_3,c_3')$. Finally fixing $c_M$ and $c_m$ we have $\mathcal{A}_{c_{M} , c_{2} , c_{m}} = \mathcal{A}_{c_{M} , c'_{2} , c_{m}}$. Therefore we obtain $ \mathcal{A}_{c_{1} , c_{2} , c_{3}}=\mathcal{A}_{c'_{1} , c'_{2} , c'_{3}}$.

Hence, considering any two sets of capacities $\{c_3  < c_2 < c_1\}$ and  $\{c_3' < c_2' < c_1'\}$, we just need to prove the following three  equalities: 

\textit{Step 1:} $ \mathcal{A}_{c'_{1} , c_{2} , c_{3}} = \mathcal{A}_{c_{1} , c_{2} , c_{3}} $ for $ c_{1} \leq c'_{1} $.

\textit{Step 2:} $ \mathcal{A}_{c_{1} , c'_{2} , c_{3}} = \mathcal{A}_{c_{1} , c_{2} , c_{3}} $ for $ c_{2} \leq c'_{2} $.

\textit{Step 3:} $ \mathcal{A}_{c_{1} , c_{2} , c'_{3}} = \mathcal{A}_{c_{1} , c_{2} , c_{3}} $ for $ c_{3} \leq c'_{3} $.

Here, we will only prove $ \mathcal{A}_{c'_{1} , c_{2} , c_{3}} \subset \mathcal{A}_{c_{1} , c_{2} , c_{3}} $ for $ c_{1} \leq c'_{1} $ in full detail. The remaining steps and inclusions are similar and are left to Appendix \ref{pf inflation}.

Take $ J \in \mathcal{A}_{c'_{1} , c_{2} , c_{3}} $ such that the embedded $ J $-holomorphic spheres satisfy  one of the following configurations from Section \ref{chp m1: Structure}:  1.1 -- 1.4, 1.7 -- 1.11, 1.13, 2.1,  2.2, 2.5, 2.6, 2.8 -- 2.10 (Figures \ref{Config1 of AP blown up 1-6},  \ref{Config1 of AP blown up 7-12}, \ref{Config1 of AP blown up 13}, \ref{Config3 of AP blown up 1-5} and \ref{Config3 of AP blown up 6-10}). Note that the curve $ E_{1} $ is represented by a $ J $-holomorphic curve in all these configurations. By definition, there is a symplectic form $ \omega_{c'_{1},c_{2},c_{3}} $ taming $ J $ and satisfying $ \omega_{c'_{1},c_{2},c_{3}}(B)=\omega_{c'_{1},c_{2},c_{3}}(F)=1 $, $ \omega_{c'_{1},c_{2},c_{3}}(E_{1})=c'_{1} $ and $ \omega_{c'_{1},c_{2},c_{3}}(E_{i})=c_{i} $ for $ i=2,3 $. To show that $ J \in \mathcal{A}_{c_{1} , c_{2} , c_{3}} $ we need to find a symplectic form $ \omega $ such that $ \omega(B)=\omega(F)=1 $ and $ \omega(E_{i})=c_{i} $ for $ i=1,2,3 $. We can use negative inflation along the curve $ E_{1} $ to define one-parameter family of symplectic forms all taming $ J $, $ \omega_{t} = \omega_{c'_{1},c_{2},c_{3}} + t \text{PD}(E_{1}) $ for $ 0 \leq t \leq c'_{1}- \epsilon $, where $ \epsilon $ can be chosen small enough so that we have $ t = t_{0} = c'_{1} - c_{1} $. For this value of $ t $, we obtain $ \omega = \omega_{c'_{1},c_{2},c_{3}} + (c'_{1}-c_{1}) \text{PD}(E_{1}) $, which satisfies $ \omega(B)=\omega(F)=1 $ and $ \omega(E_{i})=c_{i} $ for $ i=1,2,3 $, as desired.

Next, we consider $ J \in \mathcal{A}_{c'_{1} , c_{2} , c_{3}} $ with the following configurations: 1.5, 1.6, 1.12, 2.3, 2.4, 2.7. In this case, it follows from Lemma \ref{existenceJ-curves} that first we can inflate along $ B+F-E_{2} $ and $ B+F $, and then apply negative inflation to the resulting symplectic form along $ E_{1}-E_{3} $:

$$ \omega_{b,e} = \dfrac{\omega_{c'_{1},c_{2},c_{3}} + b \text{PD}(B+F-E_{2}) + e \text{PD}(B+F)}{1+b+e}, \quad  b,e \geq 0, $$
$$ \mbox{and} \quad \omega_{a}=\omega_{b,e} + a \text{PD}(E_{1}-E_{3}), \quad \mbox{with}  \quad a < \dfrac{\omega_{b,e}(E_{1}-E_{3})}{2}.$$
We get the desired symplectic form by setting $$ e= \dfrac{(1-c_{2})(c'_{1}-c_{1})}{(c_{1}+c_{3})}, \quad b = \dfrac{c_{2} e}{ 1-c_{2}} \quad \mbox{and} \quad a = \dfrac{c_{3} (b+e)}{1+b+e}. $$
These choices cover all the possible configurations of type 1 and 2. Note that, as explained at the end of Section \ref{chp m1: Structure}, the configurations of type 5 can be obtained by interchanging the roles of $ B $ and $ F $ in the configurations of type 2. Hence, we also covered the configurations of type 5.

For the configurations 3.1, 3.2 and 6.1, 6.2, again it follows from Lemma \ref{existenceJ-curves} that  we can use the curves $ B+F-E_{3} $, $ B+F $ in the inflation procedure together with the curve $ E_{1} - E_{2} $ which is represented in these configurations. This case is similar to the last one, and we only need to interchange the roles of $ E_{2} $ and $ E_{3} $ (and therefore of $ c_{2} $ and $ c_{3} $). For the configurations 3.3, 3.4, 3.8 and 6.3, 6.4, 6.6 we inflate along the curves $B+2F-E_{1}-E_{2} $, $ B+F $, $B$  and $ E_{1}-E_{2}-E_{3} $. In this case, we set

$$ \omega_{b,e,f} = \dfrac{\omega_{c'_{1},c_{2},c_{3}} + b \text{PD}(B+2F-E_{1}-E_{2}) + e \text{PD}(B+F)+f \text{PD}(B)}{1+b+e+f}, \quad b,e,f \geq 0, $$

$$ \mbox{and} \quad \omega_{a}=\omega_{b,e,f} + a \text{PD}(E_{1}-E_{2}-E_{3}) \quad \mbox{with}  \quad a < \dfrac{\omega_{b,e,f}(E_{1}-E_{2}-E_{3})}{3}.$$
We get the desired symplectic form by setting

$$ e= \dfrac{(1-2c_{2}+2c_{3})(c'_{1}-c_{1})}{(c_{1}-c_{2}+2c_{3})}, \quad  f=b = \dfrac{(c_{2}-c_{3}) e}{ 1-2c_{2}+2c_{3}},  \quad \mbox{and}\quad a = \dfrac{c_{3} (2b+e)}{1+2b+e}.$$

\noindent Note that here, to prove that 
$ 3a < \, \omega_{b,e}(E_{1}-E_{2}-E_{3}) $, we use the fact that $ c_{1}>c_{2}+c_{3} $ as otherwise $ E_{1}-E_{2}-E_{3} $ could not have a $ J $-holomorphic representative in the first place.

Then, for the configurations 3.5 -- 3.7 and 6.5, 6.7  we inflate along the curves $ B+3F-E_{1}-E_{2}-E_{3} $, $ B+F $, $B$ and $ E_{1}-E_{2}$. This case is similar to the previous case, and we use the fact that $ c_{1}+c_{2}+c_{3}<1 $, which makes it possible to have $ B-E_{1}-E_{2}-E_{3} $ in the configuration.

Since the configurations of type 4 can be obtained from configurations of type 3 by interchanging the roles of $B$ and $ F $, this finishes the proof of $ \mathcal{A}_{1 , c'_{1} , c_{2} , c_{3}} \subset \mathcal{A}_{c_{1} , c_{2} , c_{3}} $ for $ c_{1} \leq c'_{1} $. The inverse inclusion, as well as Step 2 and 3, are similar. In Appendix \ref{pf inflation}, we give a list of curves along which the inflation procedure can be used to produce the symplectic forms required in each case. This completes the proof of the theorem. 
\end{proof}

Next, we notice that one can get information on $ G_{c_{1} , c_{2}, c_{3} } $ via the group $ \text{Symp}  _{p} (\widetilde{M}_{c_{1} , c_{2} })$ of symplectomorphisms that fix a point $p$ in the manifold $ \widetilde{M}_{c_{1} , c_{2} }$, which is the manifold studied by Anjos and Pinsonnault (see \cite[Proposition 1.2]{AnjPin}). 

\begin{lemma}\label{small c3 m1}
Consider $c_1 , c_2 , c_3 \in (0,1) $ such that $ c_{3} < c_{2} \leq c_{1} < c_{1} + c_{3} < c_{1} + c_{2} \leq 1$. Then $ G_{c_{1} , c_{2}, c_{3} } $ is homotopy equivalent to $ Symp  _{p} (\widetilde{M}_{c_{1} , c_{2} })$.
\end{lemma}
\begin{proof}
The proof is an adaptation of the proof of Proposition 2.1 in \cite{LalPin} to our case. We recall here the main steps. Denote by $\Sigma$ a fixed exceptional 2-sphere, whose homology class is $E_3$. 
Let $ \text{Symp}^{U(2)} (\widetilde{M}_{c_{1} , c_{2}, c_{3}} , E_{3}) $ denote the group of symplectomorphisms of the manifold $\widetilde{M}_{c_{1} , c_{2}, c_{3}}$ which act linearly in a small neighbourhood  $\mathcal N (\Sigma)$ of the exceptional fiber, where we also fix a symplectic identification of $\mathcal N (\Sigma)$ with a neighbourhood  of the zero section in the tautological bundle over $\CP^1$. The first step in the proof is that $ \text{Symp}^{U(2)} (\widetilde{M}_{c_{1} , c_{2}, c_{3}} , E_{3}) $ is homotopy equivalent to $ \text{Symp}_{p}(\widetilde{M}_{c_{1} , c_{2}}) $: every element of $ \text{Symp}^{U(2)} (\widetilde{M}_{c_{1} , c_{2}, c_{3}} , E_{3}) $ gives rise to a symplectomorphism of $ \widetilde{M}_{c_{1} , c_{2}} $ acting linearly near the embedded ball $ B_{c_{3}} $ and fixing its center $p$. Conversely, every homotopy class of the stabilizer $ \text{Symp}  _{p} (\widetilde{M}_{c_{1} , c_{2} })$ can be realized by a family of symplectomorphisms that act linearly on a ball $ B_{c'_{3}} $ of sufficiently small capacity $ c'_{3} $ centered at $p$. So we can lift such a representative to the group $ \text{Symp}^{U(2)} (\widetilde{M}_{c_{1} , c_{2}, c_{3}} , E_{3}) $. If $r$ is the restriction map, then the directed system of homotopy maps

$$ r_{c'_{3},c_{3}} \ : \ \text{Symp}^{U(2)} (\widetilde{M}_{c_{1} , c_{2}, c'_{3}} , E_{3}) \rightarrow \text{Symp}^{U(2)} (\widetilde{M}_{c_{1} , c_{2}, c_{3}} , E_{3}) $$

\noindent gives homotopy equivalences when $ c_{3} \leq c'_{3} $ are sufficiently small, by Theorem \ref{slight changes m1}. This system together with maps

$$ g_{c_{3}} \ : \ \text{Symp}^{U(2)} (\widetilde{M}_{c_{1} , c_{2}, c_{3}} , E_{3}) \rightarrow \text{Symp}_{p}(\widetilde{M}_{c_{1} , c_{2}})$$

\noindent yield a commutative diagram  for each pair $ c_{3} \leq c'_{3} $. It is then obvious that each $ g_{c_{3}} $ is a weak homotopy equivalence.

The second step in the proof is to prove that $ \text{Symp}^{U(2)} (\widetilde{M}_{c_{1} , c_{2}, c'_{3}} , E_{3}) $ is homotopy equivalent to $\text{Symp} (\widetilde{M}_{c_{1} , c_{2}, c'_{3}} , E_{3})$. This is Lemma 2.3 in \cite{LalPin}.

Finally, by Lemma \ref{E3 m1}, we know that that the curve $E_{3}$ cannot degenerate, that is, for every $J$ in  $\mathcal{J} _{c_{1} , c_{2}, c_{3} }$ there is a unique embedded $J$-holomorphic representative of $E_{3}$ and no cusp-curve in class $E_{3}$. Hence by Lemma 2.4 in \cite{LalPin} we conclude that $ G_{c_{1} , c_{2}, c_{3} } $ retracts onto its subgroup $\text{Symp} (\widetilde{M}_{c_{1} , c_{2}, c'_{3}} , E_{3})$.
\end{proof}

\subsection{Homotopy groups of \texorpdfstring{$\Gccc$}{Lg}}
\label{sec m1: homotopy}

We use the same tools (Poincaré series, Poincaré-Birkoff-Witt Theorem, spectral sequence) as in Section 4.1 of \cite{AnjPin} to calculate

\begin{center}
$ \widetilde{r} _{n} = \dim \pi _{n} ( \Omega \widetilde{M}_{c_{1} , c_{2}} ) \otimes \mathbb{Q} = \dim \pi _{n+1} (\widetilde{M}_{c_{1} , c_{2}}) \otimes \mathbb{Q} $.
\end{center}

It is known \cite{FelHalTho} that a simply-connected space with rational cohomology of finite type and finite category is either \textit{rationally elliptic} (i.e. its rational homotopy is finite dimensional) or \textit{rationally hyperbolic} (i.e. the dimensions of the rational homotopy groups grow exponentially). Also, since $b^{2}(\widetilde{M}_{c_{1} , c_{2} }) = \dim H^{2} (\widetilde{M}_{c_{1} , c_{2} }; \mathbb{Q}) > 3$ and since a 4-dimensional simply-connected finite CW-complex has $cat(X) \leq 2 $, the following theorem shows that $\widetilde{M}_{c_{1} , c_{2}} $ is rationally hyperbolic.

\begin{theorem}[\cite{FelHalTho}, Part VI] If an n-dimensional simply-connected space $X$ is rationally elliptic, then
\begin{enumerate}
\item $ \dim \pi _{even} (X) \otimes \mathbb{Q} \leq \dim \pi _{odd} (X) \otimes \mathbb{Q} \leq cat(X) $,
\item $\chi (X) \geq 0$, where $\chi$ denotes the Euler characteristic.
\end{enumerate}
\end{theorem}

For a rationally hyperbolic space $ X $,  the Poincar\'{e} series is the formal series
$$ P_{X} = \sum_{0}^{\infty} \dim H_{n}(X;\mathbb{Q})z^{n}.$$

The Poincar\'{e}-Birkoff-Witt Theorem applied to topological spaces (see \cite[Section 33]{FelHalTho}) yields
$$ P_{\Omega X} = \dfrac{\prod_{n=0}^{\infty} (1+z^{2n+1})^{\widetilde{r}_{2n+1}}}{ \prod_{n=1}^{\infty} (1-z^{2n})^{\widetilde{r}_{2n}} }$$
so that we can recover the rank of the homotopy groups $ \widetilde{r}_{n} $ from $ \widetilde{h}_{n} = \dim H_{n}(\Omega \widetilde{M}_{c_{1} , c_{2}} ; \mathbb{Q} ) $. We compute the latter using the Serre spectral sequence for the path fibration
$$ \Omega \widetilde{M}_{c_{1} , c_{2}} \rightarrow P \widetilde{M}_{c_{1} , c_{2} } \rightarrow \widetilde{M}_{c_{1} , c_{2} }, $$
where $ P \widetilde{M}_{c_{1} , c_{2}} $ is the space of paths starting at a point $ x_{0} $ and the projection map $ p : P \widetilde{M}_{c_{1} , c_{2} } \rightarrow \widetilde{M}_{c_{1} , c_{2}} $ sends each path to its endpoint. The spectral sequence yields $ \widetilde{h}_{0}=1 $, $\widetilde{h}_{1}=4 $ and $ \widetilde{h}_{n} = 4 \widetilde{h}_{n-1} - \widetilde{h}_{n-2}$, which in turn gives, for instance, $ \widetilde{r}_{1} = 4 $, $ \widetilde{r}_{2} = 9 $, $ \widetilde{r}_{3} = 16 $ and $ \widetilde{r}_{4} = 27 $.

\begin{proposition}\label{rank upper limit}
Let $ 0 < c_{3} < c_{2} < c_{1} < c_{1} + c_{3} < c_{1} + c_{2} \leq 1$. Then
\begin{center}
$ \dim \pi_{1} (G_{c_{1}, c_{2}, c_{3}}) \leq 9 $,
\end{center}
and the homotopy groups satisfy
$$ \dim \pi_{n} (G_{c_{1}, c_{2}, c_{3}}) \leq \dim \pi_{n+1} (\widetilde{M}_{c_{1} , c_{2} }) + \dim \pi _{n}(G_{c_{1}, c_{2}}), $$
where $\Gcc$ denotes the group of symplectomorphisms of $\Mcc$ acting trivially in homology. 
\end{proposition}
\begin{proof}

We recall part of Proposition 4.2 in \cite{AnjPin} that would be relevant in our case:

\begin{proposition} (\cite[Proposition 4.2]{AnjPin})
\label{homotopygroups}
Let $0 < c_{2} < c_{1} < c_{2} + c_{1} < 1$. Then
$$ \pi_{1} (\Gcc ) = \mathbb{Q}^{5},  \quad \mbox{and}$$
$$ \pi_{n} (\Gcc ) = \mathbb{Q}^{r_{n}} \quad \mbox{for} \quad  n\geq 2,$$
where $ r_{n} = \dim \pi _{n}(\Omega \widetilde{M}_{c_{1}}) \otimes \mathbb{Q} = \dim \pi_{n+1}(\widetilde{M}_{c_{1}}) \otimes \mathbb{Q} $. For example, we have  $ r_{2}=r_{3}=5 $, $ r_{4}=10 $, and $ r_{5}=24 $.
\end{proposition}

From the evaluation fibration
$$ \Symp  _{p} (\widetilde{M}_{c_{1} , c_{2} }) \longrightarrow \Symp(\widetilde{M}_{c_{1} , c_{2} }) \xrightarrow{ev} \widetilde{M}_{c_{1} , c_{2} }, $$
using Lemma \ref{small c3 m1} and Theorem \ref{slight changes m1}, we get a long exact sequence that yields the upper bound for the rank of the homotopy groups $ \pi_{n} (G_{c_{1}, c_{2}, c_{3}})$. In particular the first terms of the long exact sequence are given by 
\begin{gather*}
\hdots   \rightarrow \pi_{2} (  \widetilde{M}_{c_{1} , c_{2} } ) \rightarrow \pi _{1} (G_{c_{1} , c_{2}, c_{3} } ) \rightarrow \pi _{1} (G_{c_{1} , c_{2}} ) \rightarrow \pi_{1} (  \widetilde{M}_{c_{1} , c_{2} } ) \rightarrow \hdots
\end{gather*}
Using Proposition \ref{homotopygroups}, we can write it as
\begin{gather*}
\rightarrow \mathbb{Q}^{4} \rightarrow \pi _{1} (G_{c_{1} , c_{2}, c_{3} } ) \rightarrow \mathbb{Q}^{5} \rightarrow  0 \rightarrow \hdots
\end{gather*}
so we can conclude that indeed $ \dim \pi_{1} (G_{c_{1}, c_{2}, c_{3}}) \leq 9 $.
\end{proof}

Therefore we expect to find at most 9 generators for $ \pi _{1} $. In the next section, we will look at toric pictures and see that the methods hitherto used in literature do not reduce the rank to less than 11 generators. We address this in Subsection \ref{section: new relation m1} .

\subsection{The generators}\label{sec m1: generators}

This section is dedicated to the description of the generators of the homotopy Lie algebra of $ G_{c_{1} , c_{2}, c_{3} }$. As these generators will appear as Hamiltonian circle actions, which then can be seen as elements of $ \pi_{1}(G_{c_{1} , c_{2}, c_{3} }) $, we will make extensive use of Delzant's classification of toric actions and Karshon's classification of Hamiltonian circle actions on symplectic manifolds. Therefore we first give a quick overview how these classifications work.

A \textit{Delzant polytope} in $ \mathbb{R}^{2} $ is a convex polytope satisfying

\begin{itemize}
\item simplicity, i.e., there are two edges meeting each vertex; 
\item rationality, i.e., the edges meeting at the vertex $ p $ are rational in the sense that each edge is of the form $ p + t u_{i} $, $ t \geq 0 $, where $ u_{i} \in \mathbb{Z}^{2} $;
\item smoothness, i.e., for each vertex, the corresponding $ u_{1} $, $ u_{2} $ can be chosen to be a $ \mathbb{Z} $-basis of $ \mathbb{Z}^{2} $.
\end{itemize}

Delzant's classification tells us that toric manifolds up to equivariant symplectomorphisms are classified by Delzant polytopes up to transformations by $ GL(2,\mathbb{Z}) $. The Delzant polytope corresponding to a toric manifold $ (M, \omega, \mathbb{T}^{2}, \phi ) $ is given by the image $ \phi(M) $ of the moment map (see \cite[Chapter 29]{daS} for the reverse direction in the correspondence). Note that in our case since we work with symplectomorphisms that act trivially on homology, we do not allow symplectomorphisms that interchange the roles of $ B $ and $ F $. Therefore, we shall consider Delzant polytopes up to transformations by $ SL(2,\mathbb{Z}) $.

On the other hand, Karshon's classification \cite{Kar} yields the correspondence between Hamiltonian circle actions and {\it decorated graphs}. A Hamiltonian circle action comes with a real-valued momentum map $ \Phi : M \rightarrow \mathbb{R} $, which is a Morse-Bott function with critical set corresponding to the fixed points. When the manifold is four-dimensional, the critical set consists of isolated points and two-dimensional submanifolds, and the latter can only occur at the extrema of $ \Phi $.

Karshon associates the following graph to $ (M, \omega, \Phi) $: for each isolated fixed point $ p $, there is a vertex $ \langle p \rangle $, labeled by the real number $ \Phi (p) $. For each two-dimensional invariant surface $ S $, there is a fat vertex $ \langle S \rangle $, labeled by two real numbers and one integer: the momentum map label $ \Phi(S) $, the area label $ \dfrac{1}{2 \pi} \int _{S} \omega $, and the genus $ g $ of the surface $ S $. A \textit{$ \mathbb{Z}_{k} $-sphere} is a gradient sphere in $ M $ on which $ S^{1} $ acts with isotropy $ \mathbb{Z}_{k} $. For each $ \mathbb{Z}_{k} $-sphere containing two fixed points $ p $ and $ q $, the graph has an edge connecting the vertices $ \langle p \rangle $ and $ \langle q \rangle $ labeled by the integer $ k $. Note that vertical translations of the graph correspond to equivariant symplectomorphisms, and flips correspond to automorphisms of the circle (see, for example, \cite[Section 3]{HolKes} for a good summary).

Karshon's classification tells us that there is a correspondence between these decorated graphs up to translation and flipping, and Hamiltonian circle actions up to equivariant symplectomorphisms that respect the moment maps. Namely, for two compact four-dimensional Hamiltonian $ S^{1} $ spaces $ (M,\omega, \Phi ) $ and $ (M',\omega', \Phi' ) $, an equivariant symplectomorphism $ F : M \rightarrow M' $ such that $ \Phi = F^{-1} \circ \Phi' \circ F $ induces an isomorphism on the corresponding graphs. Furthermore, this classification serves to keep track of symplectic blow-ups, as in Figure \ref{Karshon1} and \ref{Karshon2}.

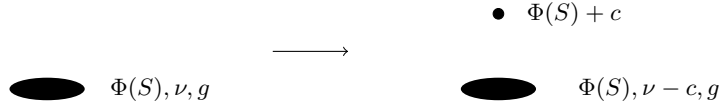
\begin{figure}[thp]
\begin{center}
\begin{tikzpicture}[font=\small]

	\fill[black] (2,-2) ellipse (0.5 and 0.15);
	\node at (3.5,-2) {$\Phi(S),\nu , g $};

	\draw[->] (5,-1.5) -- (6,-1.5);

	\filldraw [black] (8,-1) circle (2pt);
	\node at (9,-1) {$\Phi(S) + c$};
	\fill[black] (8,-2) ellipse (0.5 and 0.15);
	\node at (10,-2) {$\Phi(S),\nu - c, g$};

\end{tikzpicture}
\end{center}
\caption{Blowing-up at a point inside an invariant surface at the minimum value of $ \Phi $}\label{Karshon1}
\end{figure}

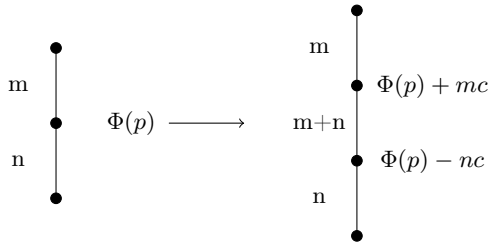
\begin{figure}[thp]
\begin{center}
\begin{tikzpicture}[font=\small]

	\filldraw [black] (2,1.5) circle (2pt);
	\filldraw [black] (2,0.5) circle (2pt);
	\node at (3,0.5) {$ \Phi(p) $};
	\filldraw [black] (2,-0.5) circle (2pt);

	\draw (2,1.5) -- (2,0.5);
	\node at (1.5,1) {m};
	\draw (2,0.5) -- (2,-0.5);
	\node at (1.5,0) {n};

	\draw[->] (3.5,0.5) -- (4.5,0.5);

	\filldraw[black] (6,2) circle (2pt);
	\filldraw [black] (6,1) circle (2pt);
	\node at (7,1) {$ \Phi(p) + m c $};
	\filldraw [black] (6,0) circle (2pt);
	\node at (7,0) {$ \Phi(p) - n c $};
	\filldraw [black] (6,-1) circle (2pt);

	\draw (6,2) -- (6,1);
	\node at (5.5,1.5) {m};
	\draw (6,1) -- (6,0);
	\node at (5.5,0.5) {m+n};
	\draw (6,0) -- (6,-1);
	\node at (5.5,-0.5) {n};
\end{tikzpicture}
\end{center}
\caption{Blowing-up at an interior fixed point}\label{Karshon2}
\end{figure}

Toric actions in 4--dimensional manifolds are generated by two Hamiltonian circle actions, which in the Delzant polytope can be seen as the projection to the axes. This relationship between polytopes and decorated graphs will be particularly useful in the following subsections.

\subsubsection{Toric actions}\label{toric actions}

In this section we will consider the Delzant polytopes of the symplectic manifold $ \widetilde{M}_{c_{1} , c_{2}, c_{3} }$ equipped with all possible toric actions. 

Let $ T^{4} \subset U(4) $ act in the standard way on $\mathbb{C}^{4}$. Given an integer $ n \geq 0 $, the action of the subtorus $ T^{2}_{0} := (s+t,t,s,s) $ is Hamiltonian with moment map
$$ (z_{1},...,z_{4}) \mapsto (|z_{1}|^{2} + |z_{3}|^{2} + |z_{4}|^{2} , |z_{1}|^{2} + |z_{2}|^{2}) $$
so that we identify $( S^{2} \times S^{2}, \sigma \oplus \sigma ) $ with the toric Hirzebruch surface $\mathbb{F}_{0}$ defined as the symplectic quotient $ \mathbb{C}^{4} // T^{2}_{0} $ at the regular value $ (1, 1) $ endowed with the residual action of the torus $ T(0) := (0,u,v,0) \subset T^{4} $. The image $ \Delta (0) $ of the moment map is the convex hull of
$$ \{ (0,0), (1,0), (1,1), (0, 1) \}.$$
We identify the symplectic blow-up $ \widetilde{M}_{c_{1}} $ at a ball of capacity $c_{1} $ with the equivariant blow-up of the Hirzebruch surface $ \mathbb{F}_{0} $.We define the even torus action $ \widetilde{T}(0) $ as the equivariant blow-up of the toric action of $T(0)$ on $ \mathbb{F}_{0} $ at the fixed point $ (0,0) $ with capacity $ c_{1} $. The image of the moment map then is the convex hull of
$$ \{ (1,1), (0, 1) , (0, c_{1} ), (c_{1} , 0), (1,0) \}.$$
The K\"{a}hler isometry group of  $ \mathbb{F}_{1} $ is $ N(T^{2} _{0}) / T^{2} _{0} $ where $ N(T^{2} _{0})$ is the normalizer of $T^{2} _{0} $ in $U(4)$. There is a natural isomorphism $N(T^{2} _{0}) / T^{2} _{0} \simeq SO(3) \times SO(3)  := K(0)$, and its restriction to the maximal torus is given in coordinates by
$$ (u,v) \mapsto (-u,v) \in T(0) := S^{1} \times S^{1} \subset K(0).$$
This identification implies that the moment polygons associated to the maximal tori $T(0)$ and $\widetilde{T}(0)$ are the images of $\Delta(0)$ and $ \widetilde{\Delta}(0) $, respectively, under the transformation
$$C_{0} =
\left( \begin{array}{cc}
-1 & 0 \\
0 & 1 \\
\end{array} \right).$$
Under the blow-down map, $ \widetilde{T}(0) $ is sent to the maximal torus of $ K(0)$. By \cite{LalPin} $ \text{Symp}(  \widetilde{M}_{c_{1}} ) $ is connected, hence the choices involved in the identification give the same maps up to homotopy.

We identify the symplectic blow-up $ \widetilde{M}_{c_{1},c_{2}} $ with the equivariant two blow-up of the Hirzebruch surface $ \mathbb{F}_{0} $ and obtain inequivalent toric structures. We define the torus actions $ \widetilde{T}_{i}(0) $, $i = 1, ... , 5$ as the equivariant blow-ups of the toric action of $T(0)$ on $ \widetilde{\mathbb{F}}_{0} $, with capacity $c_{2}$, at each one of the five fixed points, which correspond to the vertices of the moment polygon $ \widetilde{\Delta} (0) $.

Now, we blow-up each of the resulting toric actions in their six fixed points and obtain $ \widetilde{T}_{i,j}(0) $,  $j=1,...,6$. These toric pictures arise from the ones described in Section 4.2 of \cite{AnjPin} by applying one more symplectic blow-up, of capacity $c_{3}$, at each corner. In Figure \ref{Toric pictures}, we draw the toric pictures for $ \widetilde{T}_{i}(0) $, $ i=1,...,5 $, and plot the corners where the next blow-ups take place. Then, in each case we pick two Hamiltonian $S^{1}$-actions, $ x_{i,j} $ and $ y_{i,j} $, that generate the depicted toric action. More precisely, $ x_{i,j} $ and $ y_{i,j} $ are circle actions whose moment maps are, respectively, the first and second components of the moment map associated to the torus action $ \widetilde{T}_{i,j}(0) $. In fact, Y. Karshon explains in \cite{Kar} how to collect Hamiltonian $ S^{1} $-actions from these figures via graphs and how these graphs classify the circle actions. For a brief summary on decorated graphs, Delzant polytopes and the relationship between the two, see the beginning of this section.

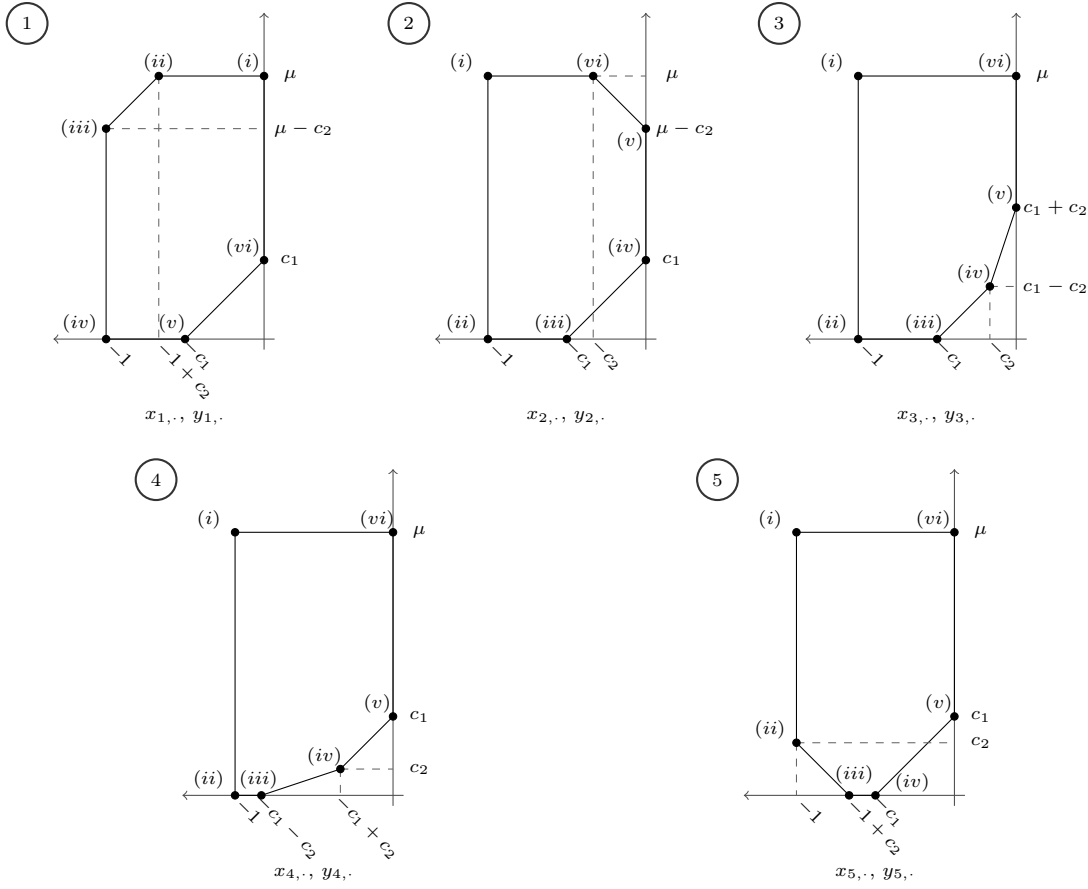
\begin{figure}[thp]
\begin{minipage}{.34\textwidth}
\begin{tikzpicture}[scale=0.70, roundnode/.style={circle, draw=black!80, thick, minimum size=3mm}, font=\scriptsize]

	\node[roundnode] at (-4.5,6) (maintopic) {1} ;
 
    \draw[<-][black!70] (-4,0) -- (0.2,0); 
    \draw[->][black!70] (0,-0.2) -- (0,6.2);

    \draw (-3,0) -- (-3,4);
    \draw (-3,4) -- (-2,5);
    \draw (-2,5) -- (0,5);
    \draw (0,5) -- (0,1.5);
    \draw (-1.5,0) -- (-3,0);
    \draw (-1.5,0) -- (0,1.5);
   
	\draw[dashed, black!60] (-3,4) -- (0,4);
	\draw[dashed, black!60] (-2,5) -- (-2,0);

    \node at (0.5,1.5) {$c_{1}$};
	\node at (0.5,5) {$\mu$};
	\node at (0.75,4) {$\mu-c_{2} $};
	\node[rotate=-45] at (-1.25,-0.3) {$ -c_{1} $};
	\node[rotate=-45] at (-1.5,-0.6) {$ -1 + c_{2} $};
	\node[rotate=-45] at (-2.75, -0.3) {$ -1 $};
	
	\filldraw [black] (-3,0) circle (2pt);
	\filldraw [black] (-3,4) circle (2pt);
	\filldraw [black] (-2,5) circle (2pt);
	\filldraw [black] (0,5) circle (2pt);
	\filldraw [black] (0,1.5) circle (2pt);
	\filldraw [black] (-1.5,0) circle (2pt);
	
    \node at (-0.3,5.25) {$(i)$};
    \node at (-2,5.25) {$(ii)$};
    \node at (-3.5,4) {$(iii)$};
    \node at (-3.5,0.3) {$(iv)$};    
    \node at (-1.75,0.3) {$(v)$};    
    \node at (-0.4,1.75) {$(vi)$};    
    
	\node at (-1.5,-1.5) {$x_{1,\cdot}$, $y_{1,\cdot}$ };

\end{tikzpicture}

\end{minipage}%
\begin{minipage}{.33\textwidth}

\begin{tikzpicture}[scale=0.70,roundnode/.style={circle, draw=black!80, thick, minimum size=3mm}, font=\scriptsize]

	\node[roundnode] at (-4.5,6) (maintopic) {2} ;
 
    \draw[<-][black!70] (-4,0) -- (0.2,0); 
    \draw[->][black!70] (0,-0.2) -- (0,6.2); 

    \draw (-3,0) -- (-3,5);
    \draw (-3,5) -- (-1,5);
    \draw (-1,5) -- (0,4);
    \draw (0,4) -- (0,1.5);
    \draw (0,1.5) -- (-1.5,0);
    \draw (-1.5,0) -- (-3,0);
   
	\draw[dashed, black!60] (-1,5) -- (0,5);
	\draw[dashed, black!60] (-1,5) -- (-1,0);

    \node at (0.5,1.5) {$c_{1}$};
	\node at (0.5,5) {$\mu$};
	\node at (0.75,4) {$\mu-c_{2} $};
	\node[rotate=-45] at (-1.25,-0.3) {$ -c_{1} $};
	\node[rotate=-45] at (-0.75,-0.3) {$ - c_{2} $};
	\node[rotate=-45] at (-2.75, -0.3) {$ -1 $};
	
	\filldraw [black] (-3,0) circle (2pt);
	\filldraw [black] (-3,5) circle (2pt);
	\filldraw [black] (-1,5) circle (2pt);
	\filldraw [black] (0,4) circle (2pt);
	\filldraw [black] (0,1.5) circle (2pt);
	\filldraw [black] (-1.5,0) circle (2pt);
	
    \node at (-3.5,5.25) {$(i)$};
    \node at (-3.5,0.3) {$(ii)$};    
    \node at (-1.75,0.3) {$(iii)$};    
    \node at (-0.4,1.75) {$(iv)$};   
    \node at (-0.3,3.75) {$(v)$};
    \node at (-1,5.25) {$(vi)$}; 
    
	\node at (-1.5,-1.5) {$x_{2,\cdot}$, $y_{2,\cdot}$ };

\end{tikzpicture}
\end{minipage}%
\begin{minipage}{.33\textwidth}
\begin{tikzpicture}[scale=0.70, roundnode/.style={circle, draw=black!80, thick, minimum size=3mm}, font=\scriptsize]

	\node[roundnode] at (-4.5,6) (maintopic) {3} ;
 
    \draw[<-][black!70] (-4,0) -- (0.2,0); 
    \draw[->][black!70] (0,-0.2) -- (0,6.2); 

    \draw (-3,0) -- (-3,5);
    \draw (-3,5) -- (0,5);
    \draw (0,5) -- (0,2.5);
    \draw (0,2.5) -- (-0.5,1);
    \draw (-0.5,1) -- (-1.5,0);
    \draw (-1.5,0) -- (-3,0);
   
	\draw[dashed, black!60] (-0.5,1) -- (0,1);
	\draw[dashed, black!60] (-0.5,1) -- (-0.5,0);

    \node at (0.75,2.5) {$c_{1} + c_{2}$};
	\node at (0.5,5) {$\mu$};
	\node at (0.75,1) {$ c_{1} - c_{2} $};
	\node[rotate=-45] at (-1.25,-0.3) {$ -c_{1} $};
	\node[rotate=-45] at (-0.25,-0.3) {$ - c_{2} $};
	\node[rotate=-45] at (-2.75, -0.3) {$ -1 $};
	
	\filldraw [black] (-3,0) circle (2pt);
	\filldraw [black] (-3,5) circle (2pt);
	\filldraw [black] (0,5) circle (2pt);
	\filldraw [black] (0,2.5) circle (2pt);
	\filldraw [black] (-0.5,1) circle (2pt);
	\filldraw [black] (-1.5,0) circle (2pt);
	
    \node at (-3.5,5.25) {$(i)$};
    \node at (-3.5,0.3) {$(ii)$};    
    \node at (-1.75,0.3) {$(iii)$};    
    \node at (-0.8,1.25) {$(iv)$};   
    \node at (-0.3,2.75) {$(v)$};
    \node at (-0.4,5.25) {$(vi)$}; 

	\node at (-1.5,-1.5) {$x_{3,\cdot}$, $y_{3,\cdot}$ };

\end{tikzpicture}
\end{minipage}

\bigskip

\begin{minipage}{.5\textwidth}
\begin{center}
\begin{tikzpicture}[scale=0.70, roundnode/.style={circle, draw=black!80, thick, minimum size=3mm}, font=\scriptsize]

	\node[roundnode] at (-4.5,6) (maintopic) {4} ;
 
    \draw[<-][black!70] (-4,0) -- (0.2,0); 
    \draw[->][black!70] (0,-0.2) -- (0,6.2); 

    \draw (-3,0) -- (-3,5);
    \draw (-3,5) -- (0,5);
    \draw (0,5) -- (0,1.5);
    \draw (0,1.5) -- (-1,0.5);
    \draw (-1,0.5) -- (-2.5,0);
    \draw (-2.5,0) -- (-3,0);
   
	\draw[dashed, black!60] (-1,0.5) -- (0,0.5);
	\draw[dashed, black!60] (-1,0.5) -- (-1,0);

    \node at (0.5,1.5) {$c_{1}$};
	\node at (0.5,5) {$\mu$};
	\node at (0.5,0.5) {$c_{2} $};
	\node[rotate=-45] at (-0.5,-0.6) {$ -c_{1} + c_{2} $};
	\node[rotate=-45] at (-2,-0.6) {$ -c_{1} - c_{2} $};
	\node[rotate=-45] at (-2.75, -0.3) {$ -1 $};
	
	\filldraw [black] (-3,0) circle (2pt);
	\filldraw [black] (-3,5) circle (2pt);
	\filldraw [black] (0,5) circle (2pt);
	\filldraw [black] (0,1.5) circle (2pt);
	\filldraw [black] (-1,0.5) circle (2pt);
	\filldraw [black] (-2.5,0) circle (2pt);
	
    \node at (-3.5,5.25) {$(i)$};
    \node at (-3.5,0.3) {$(ii)$};    
    \node at (-2.55,0.3) {$(iii)$};    
    \node at (-1.3,0.75) {$(iv)$};   
    \node at (-0.3,1.75) {$(v)$};
    \node at (-0.3,5.25) {$(vi)$}; 

	\node at (-1.5,-1.5) {$x_{4,\cdot}$, $y_{4,\cdot}$ };

\end{tikzpicture}
\end{center}
\end{minipage}%
\begin{minipage}{.5\textwidth}
\begin{center}
\begin{tikzpicture}[scale=0.70, roundnode/.style={circle, draw=black!80, thick, minimum size=3mm}, font=\scriptsize]

	\node[roundnode] at (-4.5,6) (maintopic) {5} ;
 
    \draw[<-][black!70] (-4,0) -- (0.2,0); 
    \draw[->][black!70] (0,-0.2) -- (0,6.2); 

    \draw (-3,5) -- (0,5);
    \draw (0,5) -- (0,1.5);
    \draw (0,1.5) -- (-1.5,0);
    \draw (-1.5,0) -- (-2,0);
    \draw (-2,0) -- (-3,1);
    \draw (-3,1) -- (-3,5);
   
	\draw[dashed, black!60] (-3,1) -- (0,1);
	\draw[dashed, black!60] (-3,1) -- (-3,0);

    \node at (0.5,1.5) {$c_{1}$};
	\node at (0.5,5) {$\mu$};
	\node at (0.5,1) {$c_{2} $};
	\node[rotate=-45] at (-1.25,-0.3) {$ -c_{1} $};
	\node[rotate=-45] at (-1.5,-0.6) {$ -1 + c_{2} $};
	\node[rotate=-45] at (-2.75, -0.3) {$ -1 $};
	
	\filldraw [black] (-3,1) circle (2pt);
	\filldraw [black] (-3,5) circle (2pt);
	\filldraw [black] (0,5) circle (2pt);
	\filldraw [black] (0,1.5) circle (2pt);
	\filldraw [black] (-1.5,0) circle (2pt);
	\filldraw [black] (-2,0) circle (2pt);
	
    \node at (-3.5,5.25) {$(i)$};
	\node at (-3.5,1.25) {$(ii)$};
    \node at (-1.9,0.4) {$(iii)$};
    \node at (-0.8,0.25) {$(iv)$};    
    \node at (-0.3,1.75) {$(v)$};    
    \node at (-0.4,5.25) {$(vi)$};
    
	\node at (-1.5,-1.5) {$x_{5,\cdot}$, $y_{5,\cdot}$ };

\end{tikzpicture}
\end{center}
\end{minipage}

\caption{Toric pictures for $ \widetilde{T}_{i}(0),$ $1 \leq i \leq 5$, with the next blow-ups plotted}
\label{Toric pictures}
\end{figure}

Some strata whose configurations are depicted in the figures in Section \ref{chp m1: Structure} correspond to a toric structure on $ \widetilde{M}_{c_{1} , c_{2}, c_{3} }$, unique up to equivariant symplectomorphism. These toric structures are given by the tori  $ \widetilde{T}_{i,j}(0) $. We will demonstrate this correspondence later in Appendix \ref{isoconfig m1} (see Propositon \ref{complextransitive}, Lemma \ref{isometrygroups} and Corollary \ref{homotopytype}). More accurately, the relationship between the toric pictures and the configurations is, conveniently,
$$ \widetilde{T}_{i,j}(0) \longleftrightarrow  \mbox{Configuration} (i,j)  \quad \mbox{for} \quad  i=1,...,5  \ \mbox{and} \ j=1,...,6.$$
Besides the circle actions coming from these tori $ \widetilde{T}_{i,j}(0) $, there are circle actions obtained by blowing up the toric actions $ \widetilde{T}_{i}(0) $ at an interior point of a curve, as for example in configurations 1.7 -- 1.12. Moreover, there is one configuration for which there is no correspondent $ S^{1} $--action, namely 1.13. We will now study each of these situations successively, starting with the toric pictures.

Using Karshon's classification, we use projections on the Delzant polytopes of the toric actions $ \widetilde{T}_{i,j}(0)$, onto the x-axis and onto the y-axis to obtain the graphs of the two selected Hamiltonian $S^{1}$--actions. This procedure gives the graphs in Figures \ref{S1-actions x0-x5}, \ref{S1-actions x6-x11}, \ref{S1-actions y0-y5} and \ref{S1-actions y6-y11}.

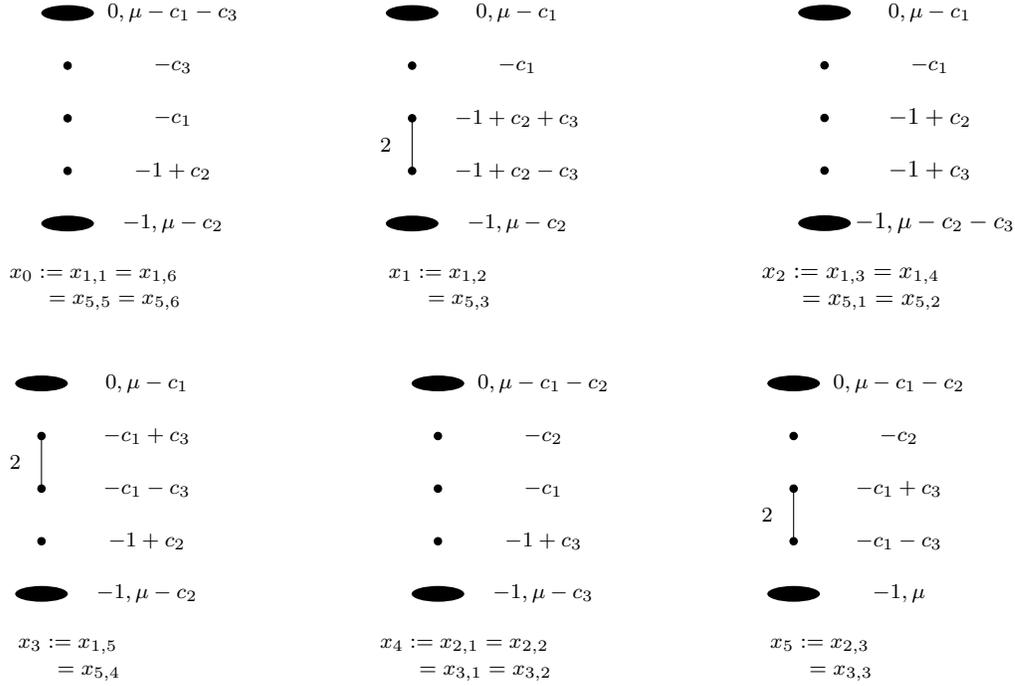
\begin{figure}[thp]
\begin{minipage}{.33\textwidth}
\begin{tikzpicture}[scale=0.7,font=\footnotesize]

		\fill[black] (2,2) ellipse (0.5 and 0.15);
	\node at (4,2) {$0,\mu - c_{1} - c_{3}$};
	\filldraw [black] (2,1) circle (2pt);
	\node at (4,1) {$ - c_{3} $};
	\filldraw [black] (2,0) circle (2pt);
	\node at (4,0) {$ - c_{1} $};
	\filldraw [black] (2,-1) circle (2pt);
	\node at (4,-1) {$-1+ c_{2} $};
	\fill[black] (2,-2) ellipse (0.5 and 0.15);
	\node at (4,-2) {$-1,\mu - c_{2}$};
       \node at (2.5,-3) {$x_{0} := x_{1,1}= x_{1,6}$};
	\node at (2.9, -3.5) {$  = x_{5,5} = x_{5,6}$};
\end{tikzpicture}
\end{minipage}%
\begin{minipage}{.34\textwidth}
\begin{tikzpicture}[scale=0.7,font=\footnotesize]

		\fill[black] (2,2) ellipse (0.5 and 0.15);
	\node at (4,2) {$0,\mu - c_{1}$};
	\filldraw [black] (2,1) circle (2pt);
	\node at (4,1) {$ - c_{1} $};
	\filldraw [black] (2,0) circle (2pt);
	\node at (4,0) {$ -1 + c_{2} + c_{3} $};
	\filldraw [black] (2,-1) circle (2pt);
	\node at (4,-1) {$-1 + c_{2} - c_{3} $};
	\fill[black] (2,-2) ellipse (0.5 and 0.15);
	\node at (4,-2) {$-1,\mu - c_{2}$};
	\draw (2,0) -- (2,-1);
	\node at (1.5,-0.5) {2};
	\node at (2.5,-3) {$x_{1} := x_{1,2}$};
	\node at (2.9,-3.5) {$  = x_{5,3} $};

\end{tikzpicture}
\end{minipage}%
\begin{minipage}{.33\textwidth}
\begin{tikzpicture}[scale=0.7,font=\footnotesize]
	
	\fill[black] (2,2) ellipse (0.5 and 0.15);
	\node at (4,2) {$0,\mu - c_{1}$};
	\filldraw [black] (2,1) circle (2pt);
	\node at (4,1) {$ - c_{1} $};
	\filldraw [black] (2,0) circle (2pt);
    \tikzstyle{every node}=[font=\small]
	\node at (4,0) {$ -1 + c_{2} $};
	\filldraw [black] (2,-1) circle (2pt);
	\node at (4,-1) {$-1 + c_{3} $};
	\fill[black] (2,-2) ellipse (0.5 and 0.15);
	\node at (4.1,-2) {$-1,\mu - c_{2} - c_{3}$};
	\node at (2.5,-3) {$x_{2} := x_{1,3}= x_{1,4} $};
	\node at (2.9,-3.5) {$  = x_{5,1}= x_{5,2}  $};
\end{tikzpicture}
\end{minipage}

\bigskip

\bigskip

\begin{minipage}{.33\textwidth}

\begin{tikzpicture}[scale=0.7,font=\footnotesize]

		\fill[black] (2,2) ellipse (0.5 and 0.15);
	\node at (4,2) {$0,\mu - c_{1}$};
	\filldraw [black] (2,1) circle (2pt);
	\node at (4,1) {$ - c_{1} + c_{3} $};
	\filldraw [black] (2,0) circle (2pt);
	\node at (4,0) {$ - c_{1} - c_{3} $};
    \draw (2,1) -- (2,0);
    \node at (1.5,0.5) {2};
	\filldraw [black] (2,-1) circle (2pt);
	\node at (4,-1) {$-1 + c_{2} $};
	\fill[black] (2,-2) ellipse (0.5 and 0.15);
	\node at (4,-2) {$-1,\mu - c_{2}$};
	\node at (2.5,-3) {$x_{3} := x_{1,5}$};
	\node at (2.9,-3.5) {$  = x_{5,4} $};

\end{tikzpicture}
\end{minipage}%
\begin{minipage}{.34\textwidth}
\begin{tikzpicture}[scale=0.7,font=\footnotesize]
	
	\fill[black] (2,2) ellipse (0.5 and 0.15);
	\node at (4,2) {$0,\mu - c_{1} -c_{2}$};
	\filldraw [black] (2,1) circle (2pt);
	\node at (4,1) {$ - c_{2} $};
	\filldraw [black] (2,0) circle (2pt);
	\node at (4,0) {$ - c_{1} $};
	\filldraw [black] (2,-1) circle (2pt);
	\node at (4,-1) {$-1 + c_{3} $};
	\fill[black] (2,-2) ellipse (0.5 and 0.15);
	\node at (4,-2) {$-1,\mu - c_{3}$};
	\node at (2.5,-3) {$x_{4} := x_{2,1}= x_{2,2}$};
	\node at (2.9,-3.5) {$  = x_{3,1}= x_{3,2}  $};
\end{tikzpicture}
\end{minipage}%
\begin{minipage}{.33\textwidth}

\begin{tikzpicture}[scale=0.7,font=\footnotesize]
  	\fill[black] (2,2) ellipse (0.5 and 0.15);
	\node at (4,2) {$0,\mu - c_{1} -c_{2}$};
	\filldraw [black] (2,1) circle (2pt);
	\node at (4,1) {$ - c_{2} $};
	\filldraw [black] (2,0) circle (2pt);
	\node at (4,0) {$ - c_{1} + c_{3} $};
	\filldraw [black] (2,-1) circle (2pt);
	\node at (4,-1) {$- c_{1} - c_{3} $};
    \draw (2,0) -- (2,-1);
    \node at (1.5,-0.5) {2};
	\fill[black] (2,-2) ellipse (0.5 and 0.15);
	\node at (4,-2) {$-1,\mu$};
	\node at (2.5,-3) {$x_{5} := x_{2,3}$};
	\node at (2.9,-3.5) {$  = x_{3,3} $};

\end{tikzpicture}
\end{minipage}

\caption{Graphs of the $ S^{1}$--actions $ x_{0},x_{1} x_2,x_3,x_4$ and $x_5$.}
\label{S1-actions x0-x5}
\end{figure}

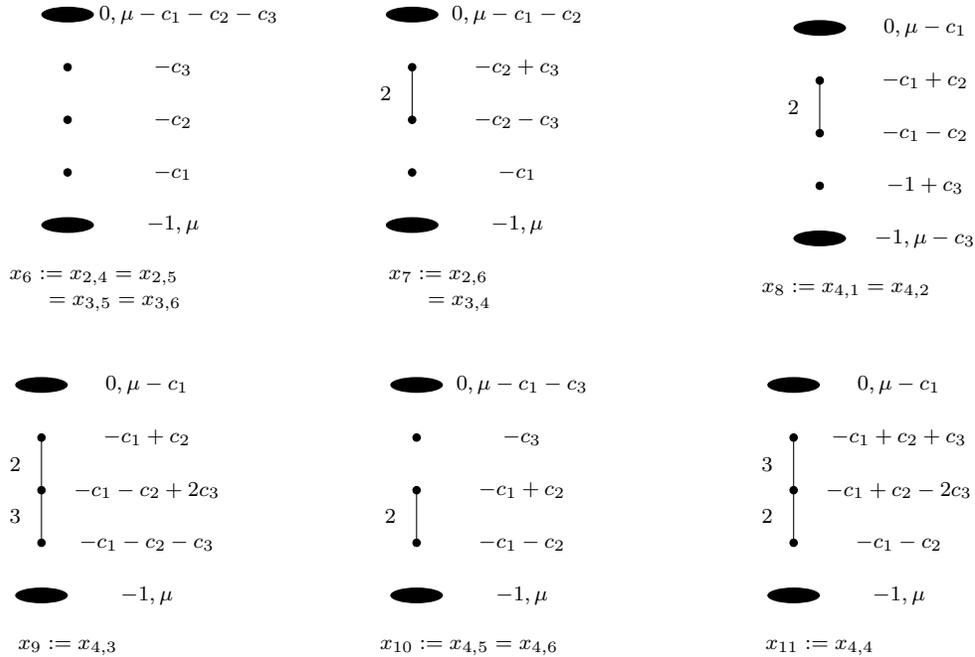
\begin{figure}[thp]
\vspace*{.7cm}
\begin{minipage}{.33\textwidth}
\begin{tikzpicture}[scale=0.7,font=\footnotesize]
	\fill[black] (2,2) ellipse (0.5 and 0.15);
	\node at (4.3,2) {$0,\mu - c_{1} - c_{2} - c_{3}$};
	\filldraw [black] (2,1) circle (2pt);
	\node at (4,1) {$ - c_{3} $};
	\filldraw [black] (2,0) circle (2pt);
	\node at (4,0) {$ - c_{2} $};
	\filldraw [black] (2,-1) circle (2pt);
	\node at (4,-1) {$ - c_{1} $};
	\fill[black] (2,-2) ellipse (0.5 and 0.15);
	\node at (4,-2) {$-1,\mu$};
	\node at (2.5,-3) {$x_{6} := x_{2,4}=x_{2,5}$};
	\node at (2.9,-3.5) {$  = x_{3,5} = x_{3,6} $};
\end{tikzpicture}
\end{minipage}%
\begin{minipage}{.34\textwidth}
\begin{tikzpicture}[scale=0.7,font=\footnotesize]
	\fill[black] (2,2) ellipse (0.5 and 0.15);
	\node at (4,2) {$0,\mu - c_{1} -c_{2}$};
	\filldraw [black] (2,1) circle (2pt);
	\node at (4,1) {$ - c_{2} + c_{3} $};
	\filldraw [black] (2,0) circle (2pt);
	\node at (4,0) {$ - c_{2} - c_{3} $};
    \draw (2,1) -- (2,0);
    \node at (1.5,0.5) {2};	
	\filldraw [black] (2,-1) circle (2pt);
	\node at (4,-1) {$- c_{1} $};
	\fill[black] (2,-2) ellipse (0.5 and 0.15);
	\node at (4,-2) {$-1,\mu$};
	\node at (2.5,-3) {$x_{7} := x_{2,6}$};
	\node at (2.9,-3.5) {$  = x_{3,4} $};
\end{tikzpicture}
\end{minipage}%
\begin{minipage}{.33\textwidth}
\begin{tikzpicture}[scale=0.7,font=\footnotesize]
	\fill[black] (2,2) ellipse (0.5 and 0.15);
	\node at (4,2) {$0,\mu - c_{1}$};
	\filldraw [black] (2,1) circle (2pt);
	\node at (4,1) {$ - c_{1} + c_{2} $};
	\filldraw [black] (2,0) circle (2pt);
	\node at (4,0) {$ - c_{1} - c_{2} $};
    \draw (2,1) -- (2,0);
    \node at (1.5,0.5) {2};	
	\filldraw [black] (2,-1) circle (2pt);
	\node at (4,-1) {$-1 + c_{3} $};
	\fill[black] (2,-2) ellipse (0.5 and 0.15);
	\node at (4,-2) {$-1,\mu - c_{3}$};
	\node at (2.5,-3) {$x_{8} := x_{4,1}= x_{4,2} $};
\end{tikzpicture}
\end{minipage}

\bigskip

\bigskip

\begin{minipage}{.33\textwidth}
\begin{tikzpicture}[scale=0.7,font=\footnotesize]
\fill[black] (2,2) ellipse (0.5 and 0.15);
	\node at (4,2) {$0,\mu - c_{1}$};
	\filldraw [black] (2,1) circle (2pt);
	\node at (4,1) {$ - c_{1} + c_{2} $};
	\filldraw [black] (2,0) circle (2pt);
	\node at (4,0) {$ - c_{1} - c_{2} + 2 c_{3} $};
    \draw (2,1) -- (2,0);
    \node at (1.5,0.5) {2};	
	\filldraw [black] (2,-1) circle (2pt);
	\node at (4,-1) {$ - c_{1} - c_{2} - c_{3} $};
    \draw (2,0) -- (2,-1);
    \node at (1.5,-0.5) {3};	
	\fill[black] (2,-2) ellipse (0.5 and 0.15);
	\node at (4,-2) {$-1,\mu$};
	\node at (2.5,-3) {$x_{9} := x_{4,3}$};
\end{tikzpicture}
\end{minipage}%
\begin{minipage}{.34\textwidth}
\begin{tikzpicture}[scale=0.7,font=\footnotesize]
		\fill[black] (2,2) ellipse (0.5 and 0.15);
	\node at (4,2) {$0,\mu - c_{1} - c_{3}$};
	\filldraw [black] (2,1) circle (2pt);
	\node at (4,1) {$ - c_{3} $};
	\filldraw [black] (2,0) circle (2pt);
	\node at (4,0) {$ - c_{1} + c_{2} $};
	\filldraw [black] (2,-1) circle (2pt);
	\node at (4,-1) {$ - c_{1} - c_{2} $};
    \draw (2,0) -- (2,-1);
    \node at (1.5,-0.5) {2};	
	\fill[black] (2,-2) ellipse (0.5 and 0.15);
	\node at (4,-2) {$-1,\mu $};
	\node at (3,-3) {$x_{10} := x_{4,5}= x_{4,6} $};
	\end{tikzpicture}
\end{minipage}%
\begin{minipage}{.33\textwidth}
\begin{tikzpicture}[scale=0.7,font=\footnotesize]
	\fill[black] (2,2) ellipse (0.5 and 0.15);
	\node at (4,2) {$0,\mu - c_{1}$};
	\filldraw [black] (2,1) circle (2pt);
	\node at (4,1) {$ - c_{1} + c_{2} + c_{3} $};
	\filldraw [black] (2,0) circle (2pt);
	\node at (4,0) {$ - c_{1} + c_{2} - 2 c_{3} $};
    \draw (2,1) -- (2,0);
    \node at (1.5,0.5) {3};	
	\filldraw [black] (2,-1) circle (2pt);
	\node at (4,-1) {$ - c_{1} - c_{2} $};
    \draw (2,0) -- (2,-1);
    \node at (1.5,-0.5) {2};	
	\fill[black] (2,-2) ellipse (0.5 and 0.15);
	\node at (4,-2) {$-1,\mu$};
	\node at (2.5,-3) {$x_{11} := x_{4,4}$};	
\end{tikzpicture}
\end{minipage}
\caption{Graphs of the $ S^{1}$-actions $ x_{6}, x_{7},x_8,x_9,x_{10} $ and $x_{11}$.}
\label{S1-actions x6-x11}
\end{figure}

\begin{figure}[thp]

\bigskip

\bigskip

\begin{minipage}{.33\textwidth}
\begin{tikzpicture}[scale=0.7,font=\footnotesize]
		\fill[black] (2,2) ellipse (0.5 and 0.15);
	\node at (4,2) {$\mu, 1 - c_{2}$};
	\filldraw [black] (2,1) circle (2pt);
	\node at (4,1) {$ \mu - c_{2} $};
	\filldraw [black] (2,0) circle (2pt);
	\node at (4,0) {$ c_{1} $};
	\filldraw [black] (2,-1) circle (2pt);
	\node at (4,-1) {$ c_{3} $};
	\fill[black] (2,-2) ellipse (0.5 and 0.15);
	\node at (4,-2) {$0,1 - c_{1} - c_{3}$};
	\node at (2.5,-3) {$y_0 := y_{1,4}= y_{1,5} $};
	\node at (2.9,-3.5) {$  = y_{2,2}= y_{2,3} $};
\end{tikzpicture}
\end{minipage}%
\begin{minipage}{.34\textwidth}
\begin{tikzpicture}[scale=0.7,font=\footnotesize]
	\fill[black] (2,2) ellipse (0.5 and 0.15);
	\node at (4,2) {$\mu ,1 - c_{2}$};
	\filldraw [black] (2,1) circle (2pt);
	\node at (4,1) {$ \mu - c_{2} + c_{3} $};
	\filldraw [black] (2,0) circle (2pt);
	\node at (4,0) {$ \mu - c_{2} - c_{3} $};
    \draw (2,1) -- (2,0);
    \node at (1.5,0.5) {2};
	\filldraw [black] (2,-1) circle (2pt);
	\node at (4,-1) {$c_{1}$};
	\fill[black] (2,-2) ellipse (0.5 and 0.15);
	\node at (4,-2) {$0,1 - c_{1}$};
	\node at (2.5,-3) {$y_{1} := y_{1,3}$};
	\node at (2.9,-3.5) {$  = y_{2,5} $};
\end{tikzpicture}
\end{minipage}%
\begin{minipage}{.33\textwidth}
\begin{tikzpicture}[scale=0.7,font=\footnotesize]
	\fill[black] (2,2) ellipse (0.5 and 0.15);
	\node at (4,2) {$\mu,1 - c_{2} - c_{3}$};
	\filldraw [black] (2,1) circle (2pt);
	\node at (4,1) {$ \mu - c_{3} $};
	\filldraw [black] (2,0) circle (2pt);
	\node at (4,0) {$ \mu - c_{2} $};
	\filldraw [black] (2,-1) circle (2pt);
	\node at (4,-1) {$ c_{1} $};
	\fill[black] (2,-2) ellipse (0.5 and 0.15);
	\node at (4,-2) {$0,1 - c_{1}$};
	\node at (2.5,-3) {$y_2 := y_{1,1}= y_{1,2} $};
	\node at (2.9,-3.5) {$  = y_{2,1}= y_{2,6}  $};
\end{tikzpicture}
\end{minipage}

\bigskip

\bigskip

\begin{minipage}{.33\textwidth}
\begin{tikzpicture}[scale=0.7,font=\footnotesize]
	\fill[black] (2,2) ellipse (0.5 and 0.15);
	\node at (4,2) {$\mu , 1 - c_{2}$};
	\filldraw [black] (2,1) circle (2pt);
	\node at (4,1) {$ \mu - c_{2}$};
	\filldraw [black] (2,0) circle (2pt);
	\node at (4,0) {$ c_{1} + c_{3} $};
	\filldraw [black] (2,-1) circle (2pt);
	\node at (4,-1) {$ c_{1} - c_{3} $};
    \draw (2,0) -- (2,-1);
    \node at (1.5,-0.5) {2};
	\fill[black] (2,-2) ellipse (0.5 and 0.15);
	\node at (4,-2) {$0,1 - c_{1}$};
	\node at (3,-3) {$y_{3} := y_{1,6}= y_{2,4} $};
\end{tikzpicture}
\end{minipage}%
\begin{minipage}{.34\textwidth}
\begin{tikzpicture}[scale=0.7,font=\footnotesize]
	\fill[black] (2,2) ellipse (0.5 and 0.15);
	\node at (4,2) {$\mu, 1 -c_{3}$};
	\filldraw [black] (2,1) circle (2pt);
	\node at (4,1) {$ \mu - c_{3} $};
	\filldraw [black] (2,0) circle (2pt);
	\node at (4,0) {$ c_{1} $};
	\filldraw [black] (2,-1) circle (2pt);
	\node at (4,-1) {$ c_{2} $};
	\fill[black] (2,-2) ellipse (0.5 and 0.15);
	\node at (4,-2) {$0,1- c_{1} - c_{2}$};
	\node at (2.5,-3) {$y_4 := y_{4,1}= y_{4,6}$};
	\node at (2.9,-3.5) {$  = y_{5,1} = y_{5,6} $};
\end{tikzpicture}
\end{minipage}%
\begin{minipage}{.33\textwidth}
\begin{tikzpicture}[scale=0.7,font=\footnotesize]
	\fill[black] (2,2) ellipse (0.5 and 0.15);
	\node at (4,2) {$\mu, 1$};
	\filldraw [black] (2,1) circle (2pt);
	\node at (4,1) {$ c_{1} + c_{3} $};
	\filldraw [black] (2,0) circle (2pt);
	\node at (4,0) {$ c_{1} - c_{3} $};
    \draw (2,1) -- (2,0);
    \node at (1.5,0.5) {2};
	\filldraw [black] (2,-1) circle (2pt);
	\node at (4,-1) {$ c_{2} $};
	\fill[black] (2,-2) ellipse (0.5 and 0.15);
	\node at (4,-2) {$0,1 - c_{1} - c_{2} $};
	\node at (3,-3) {$y_5 := y_{4,5}= y_{5,5}$};
\end{tikzpicture}
\end{minipage}
\caption{Graphs of the  $ S^{1}$-actions $y_0$, $y_1$, $ y_2 $, $ y_{3} $, $ y_4 $,  and $ y_5  $}
\label{S1-actions y0-y5}
\end{figure}
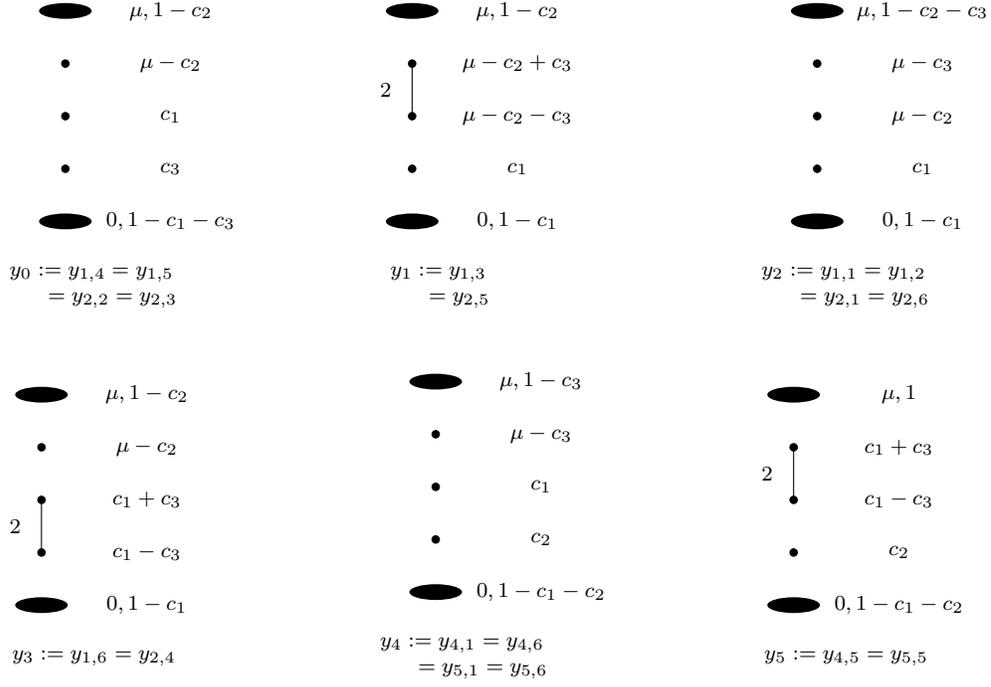

\begin{figure}[thp]

\bigskip
\begin{minipage}{.33\textwidth}
\begin{tikzpicture}[scale=0.7,font=\footnotesize]
	\fill[black] (2,2) ellipse (0.5 and 0.15);
	\node at (4,2) {$\mu,1$};
	\filldraw [black] (2,1) circle (2pt);
	\node at (4,1) {$ c_{1} $};
	\filldraw [black] (2,0) circle (2pt);
	\node at (4,0) {$ c_{2} $};
	\filldraw [black] (2,-1) circle (2pt);
	\node at (4,-1) {$ c_{3} $};
	\fill[black] (2,-2) ellipse (0.5 and 0.15);
	\node at (4.3,-2) {$0,1 - c_{1} - c_{2} - c_{3}$};
	\node at (2.5,-3) {$y_6 := y_{4,2}= y_{4,3}$};
	\node at (2.9,-3.5) {$  = y_{5,3} = y_{5,4} $};
\end{tikzpicture}
\end{minipage}%
\begin{minipage}{.34\textwidth}
\begin{tikzpicture}[scale=0.7,font=\footnotesize]
	\fill[black] (2,2) ellipse (0.5 and 0.15);
	\node at (4,2) {$\mu, 1$};
	\filldraw [black] (2,1) circle (2pt);
	\node at (4,1) {$ c_{1} $};
	\filldraw [black] (2,0) circle (2pt);
	\node at (4,0) {$ c_{2} + c_{3} $};
	\filldraw [black] (2,-1) circle (2pt);
	\node at (4,-1) {$ c_{2} - c_{3} $};
    \draw (2,0) -- (2,-1);
    \node at (1.5,-0.5) {2};
	\fill[black] (2,-2) ellipse (0.5 and 0.15);
	\node at (4,-2) {$0,1-c_{1}-c_{2}$};
	\node at (3,-3) {$y_7 := y_{4,4}= y_{5,2} $};
\end{tikzpicture}
\end{minipage}%
\begin{minipage}{.33\textwidth}
\begin{tikzpicture}[scale=0.7,font=\footnotesize]
	\fill[black] (2,2) ellipse (0.5 and 0.15);
	\node at (4,2) {$\mu, 1 - c_{3}$};
	\filldraw [black] (2,1) circle (2pt);
	\node at (4,1) {$ \mu - c_{3} $};
	\filldraw [black] (2,0) circle (2pt);
	\node at (4,0) {$ c_{1} + c_{2} $};
	\filldraw [black] (2,-1) circle (2pt);
	\node at (4,-1) {$ c_{1} - c_{2} $};
    \draw (2,0) -- (2,-1);
    \node at (1.5,-0.5) {2};
	\fill[black] (2,-2) ellipse (0.5 and 0.15);
	\node at (4,-2) {$0,1 - c_{1}$};
	\node at (3,-3) {$y_8 := y_{3,1}= y_{3,6}$};
\end{tikzpicture}
\end{minipage}

\bigskip

\bigskip

\begin{minipage}{.33\textwidth}
\begin{tikzpicture}[scale=0.7,font=\footnotesize]
	\fill[black] (2,2) ellipse (0.5 and 0.15);
	\node at (4,2) {$\mu,1$};
	\filldraw [black] (2,1) circle (2pt);
	\node at (4,1) {$ c_{1} + c_{2} + c_{3} $};
	\filldraw [black] (2,0) circle (2pt);
	\node at (4,0) {$ c_{1} + c_{2} - 2c_{3} $};
    \draw (2,1) -- (2,0);
    \node at (1.5,0.5) {3};
	\filldraw [black] (2,-1) circle (2pt);
	\node at (4,-1) {$ c_{1} - c_{2} $};
    \draw (2,0) -- (2,-1);
    \node at (1.5,-0.5) {2};
	\fill[black] (2,-2) ellipse (0.5 and 0.15);
	\node at (4,-2) {$0,1-c_{1}$};
	\node at (2.5,-3) {$y_9 := y_{3,5}$};
	\node at (2.9,-3.5) {$\ \ $};
\end{tikzpicture}
\end{minipage}%
\begin{minipage}{.34\textwidth}
\begin{tikzpicture}[scale=0.7,font=\footnotesize]
	\fill[black] (2,2) ellipse (0.5 and 0.15);
	\node at (4,2) {$\mu,1$};
	\filldraw [black] (2,1) circle (2pt);
	\node at (4,1) {$ c_{1} + c_{2} $};
	\filldraw [black] (2,0) circle (2pt);
	\node at (4,0) {$ c_{1} - c_{2} $};
    \draw (2,1) -- (2,0);
    \node at (1.5,0.5) {2};
	\filldraw [black] (2,-1) circle (2pt);
	\node at (4,-1) {$c_{3} $};
	\fill[black] (2,-2) ellipse (0.5 and 0.15);
	\node at (4,-2) {$0,1-c_{1}-c_{3}$};
	\node at (3,-3) {$y_{10} := y_{3,2}= y_{3,3}$};
\end{tikzpicture}
\end{minipage}%
\begin{minipage}{.33\textwidth}
\begin{tikzpicture}[scale=0.7,font=\footnotesize]
	\fill[black] (2,2) ellipse (0.5 and 0.15);
	\node at (4,2) {$\mu, 1$};
	\filldraw [black] (2,1) circle (2pt);
	\node at (4,1) {$ c_{1} + c_{2}$};
	\filldraw [black] (2,0) circle (2pt);
	\node at (4,0) {$ c_{1} - c_{2} + 2 c_{3} $};
    \draw (2,1) -- (2,0);
    \node at (1.5,0.5) {2};	
	\filldraw [black] (2,-1) circle (2pt);
	\node at (4,-1) {$ c_{1} - c_{2} - c_{3} $};
    \draw (2,0) -- (2,-1);
    \node at (1.5,-0.5) {3};	
	\fill[black] (2,-2) ellipse (0.5 and 0.15);
	\node at (4,-2) {$0,1 - c_{1}$};
	\node at (3,-3) {$y_{11} := y_{3,4}$};	
\end{tikzpicture}
\end{minipage}
\caption{Graphs of the  $ S^{1}$-actions $y_6$, $y_7$, $ y_8 $, $ y_9 $, $ y_{10} $ and $ y_{11} $}
\label{S1-actions y6-y11}
\end{figure}
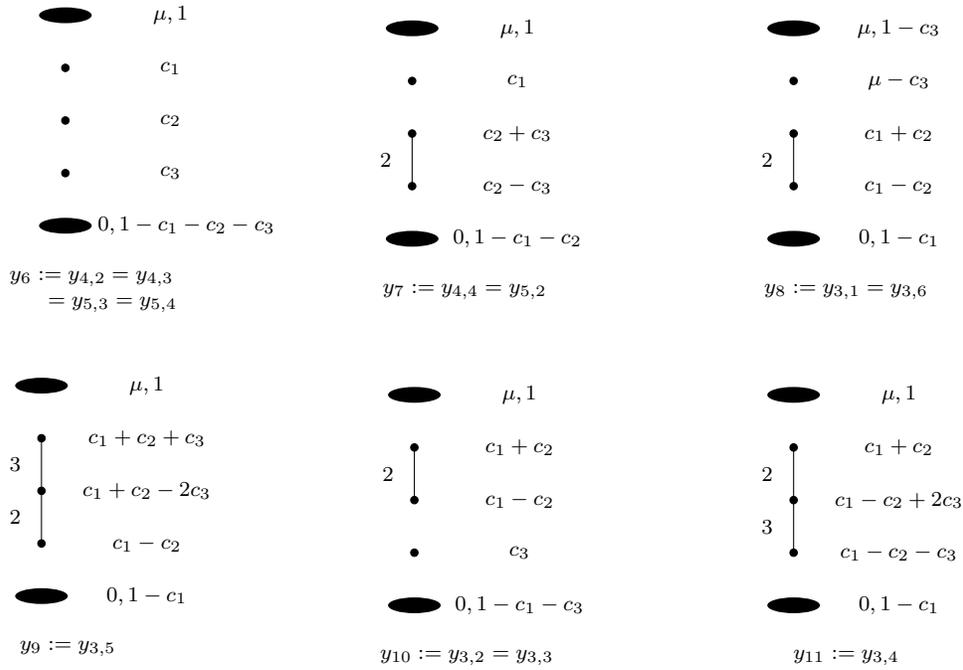

\begin{remark}\label{injectivehomotopy}
 As explained by D. McDuff and S. Tolman in \cite{McDTol}, the maps $ \widetilde{T}_{i,j}(0) \rightarrow G_{c_{1}, c_{2}, c_{3}} $ induce injective maps of fundamental groups. More precisely, this follows from combining
Theorem 1.3 and Theorem 1.25 (i) in their paper, as explained in Remark 1.26, since those theorems imply that if $(M, \omega, T, \phi)$ is a symplectic toric 4--manifold with moment polytope $\Delta$ which has five or more facets then the induced map on the fundamental groups is injective. Therefore we can see the actions coming from the toric structures as elements  of the fundamental group of $ G_{c_{1}, c_{2}, c_{3}} $. Using Karshon's classification, we can now find relations between the corresponding elements in $ \pi _{1} $. Note that since we work only with topological groups up to rational equivalences and in order to simplify notation, we will also denote by $ (x_{i,j},y_{i,j})$ the elements in $\pi_1(\widetilde{T}_{i,j}(0))$ and in $\pi_1(\Gccc)$. 
\end{remark}

Transforming the polytopes by $ SL(2, \mathbb{Z}) $ transformations, we get several relations that we will denote by type--0 relations. For instance, consider the polytopes obtained from the toric picture for $\widetilde{T}_1(0)$ in Figure 
\ref{Toric pictures} blowing-up at the points $(ii)$ and $(iii)$. These contain the toric actions represented by $(x_{1,2},y_{1,2})=(x_1,y_2)$ and $(x_{1,3},y_{1,3})=(x_2,y_1)$, respectively. Performing the $ SL(2, \mathbb{Z}) $ transformation represented by the matrix 
$$ \left( \begin{array}{cc} 1 & 0 \\ -1 & 1 \end{array} \right)$$
to both polytopes and then projecting onto the $y$-axis we obtain the same graph, so we conclude, using Karshon's classification, that as circle actions we have $ y_{2} - x_{1} = y_{1} - x_{2} $. 

Below, we list all the type--0 relations together with the reference to the polytopes that give each relation. 

Type--0 relations:
\begin{align*}
& (1) \    y_{2} - x_{1} = y_{1} - x_{2},& &  1(ii)  \leftrightarrow  1(iii); & 
& (2) \   y_0 - x_{3} = y_{3} - x_{0},& & 1(v)  \leftrightarrow  1(vi);\\
 & (3) \   y_0 - x_{5} = y_{3} - x_{6},& & 2(iii)  \leftrightarrow  2(iv);&
& (4) \  y_{1} + x_{6} = y_{2} + x_{7},& &2(v)  \leftrightarrow  2(vi);\\
& (5) \  y_{8} - x_{4} = y_{4} - x_{8},& &3(i)  \leftrightarrow  4(i); & 
& (6) \   y_{10} - x_{4} = y_{6} - x_{8},& & 3(ii)  \leftrightarrow  4(ii);\\
& (7) \   y_{6} - x_{9} = y_{9} - x_{6},& & 3(v)  \leftrightarrow  4(iii);& 
& (8) \   y_{8} - x_{6} = y_{4} - x_{10},& & 3(vi)  \leftrightarrow  4(vi);\\
& (9) \   y_{7} + x_{2} = y_{6} + x_{1},& & 5(ii)  \leftrightarrow  5(iii);& 
& (10) \  y_{6} - x_{3} = y_{5} - x_{0},& & 5(iv) \leftrightarrow  5(v);\\
& (11) \   y_{10} - x_{5} = y_{11} - x_{7},& & 3(iii)  \leftrightarrow  3(iv);&
& (12) \  y_{11} - 2 x_{7} = y_{9} - 2 x_{6},& & 3(iv)  \leftrightarrow  3(v);\\
& (13) \   x_{9}-2 y_{6}   = x_{11}-2 y_{7}  ,& & 4(iii)  \leftrightarrow  4(iv);&
 & (14) \  y_{7} - x_{11} = y_{5} - x_{10},& & 4(iv)  \leftrightarrow  4(v);\\
& (15) \   y_{11} - x_{7} = y_{5} - x_{10},& & 3(iv)  \leftrightarrow  4(v);&
& (16) \   y_{10} - x_{5} = y_{7} - x_{11},& & 3(iii)  \leftrightarrow  4(iv)
\end{align*}

\subsubsection{New relations}\label{section: new relation m1}

There are two more relations of a somewhat different nature:

\begin{proposition} Consider the circle actions $x_0,x_2,x_4,x_6$  and $y_0,y_2,y_{4}, y_{6}$  defined  in Section \ref{toric actions} as elements of the fundamental group $\pi_1(\Gccc)$.  We have the following relations:
\begin{equation}\label{the x-relation}
x_{0} + x_{4} = x_{2} + x_{6}.
\end{equation}
\begin{equation}\label{the y-relation}
y_0 + y_{4} = y_{2} + y_{6}
\end{equation}

\begin{proof}
Note that for any $ \mu > 1 $, we can also draw the auxiliary Delzant polytopes in Figure \ref{New relation toric pictures}. We name the $S^{1}$-actions corresponding to these polytopes in accordance with the more general case explained later in Appendix \ref{app: relations m>1}.

\begin{figure}[thp]
\begin{center}
\begin{tikzpicture}[roundnode/.style={circle, draw=black!80, thick, minimum size=7mm}, font=\small]

	\node[roundnode] at (-1,6) (maintopic) {9$(vi)$} ;
 
    \draw[->][black!70] (-0.2,0) -- (1.2,0); 
    \draw[->][black!70] (0,-1.1) -- (0,5.2); 

    \draw[brown] (1,5) -- (1,-0.9);
    \draw[brown] (1,-0.9) -- (0.9,-0.9);
    \draw[brown] (0.9,-0.9) -- (0.5,-0.5);
    \draw[brown] (0.5,-0.5) -- (0,0.5);
    \draw[brown] (0,0.5) -- (0,3.7);
    \draw[brown] (0,3.7) -- (0.3,4.3);
    \draw[brown] (0.3,4.3) -- (1,5);
   
	\draw[dashed, black!60] (1,5) -- (0,5);
	\draw[dashed, black!60] (0.9,-0.9) -- (0,-0.9);
	\draw[dashed, black!60] (0.9,-0.9) -- (0.9,0);
	\draw[dashed, black!60] (0.5,-0.5) -- (0.5,0);
	\draw[dashed, black!60] (0.5,-0.5) -- (0,-0.5);
	\draw[dashed, black!60] (0.3,4.3) -- (0.3,0);
	\draw[dashed, black!60] (0.3,4.3) -- (0,4.3);

	\node at (-0.3,5) {$\mu$};
	\node at (-0.9,4.3) {$\mu -1 + c_{2}$};
	\node at (-0.9,3.7) {$\mu -1 - c_{2} $};
	\node at (-0.5,0.5) {$c_{1}$};
	\node at (-0.5,-0.5) {$-c_{1}$};
	\node at (-0.8,-0.9) {$-1+c_{3}$};
	\node[rotate=45] at (0.3,0.3) {$ c_{2} $};
	\node[rotate=45] at (0.6,0.3) {$ c_{1} $};
	\node[rotate=45] at (1.3,0.5) {$ 1 - c_{3}$};
    
	\node at (0,-2) {($x_{2,2,6}$, $y_{2,2,6}$)};

	\node[roundnode] at (7,6) (maintopic) {11$(vi)$} ;
 
    \draw[->][black!70] (7.8,0) -- (9.2,0); 
    \draw[->][black!70] (8,-1.1) -- (8,5.2); 

    \draw[brown] (9,5) -- (9,-0.9);
    \draw[brown] (9,-0.9) -- (8.9,-0.9);
    \draw[brown] (8.9,-0.9) -- (8.5,-0.5);
    \draw[brown] (8.5,-0.5) -- (8.3,-0.1);
    \draw[brown] (8.3,-0.1) -- (8,0.8);
    \draw[brown] (8,0.8) -- (8,4);
    \draw[brown] (8,4) -- (9,5);
   
	\draw[dashed, black!60] (9,5) -- (8,5);
	\draw[dashed, black!60] (8.9,-0.9) -- (8,-0.9);
	\draw[dashed, black!60] (8.9,-0.9) -- (8.9,0);
	\draw[dashed, black!60] (8.5,-0.5) -- (8.5,0);
	\draw[dashed, black!60] (8.5,-0.5) -- (8,-0.5);
	\draw[dashed, black!60] (8.3,-0.1) -- (8.3,0);
	\draw[dashed, black!60] (8.3,-0.1) -- (8,-0.1);

	\node at (7.7,5) {$\mu$};
	\node at (7.5,4) {$\mu -1 $};
	\node at (7.2,0.8) {$c_{1} + c_{2} $};
	\node at (7.2,-0.1) {$c_{1} - 2 c_{2}$};
	\node at (7.5,-0.5) {$-c_{1}$};
	\node at (7.2,-0.9) {$-1+c_{3}$};
	\node[rotate=45] at (8.3,0.3) {$ c_{2} $};
	\node[rotate=45] at (8.7,0.3) {$ c_{1} $};
	\node[rotate=45] at (9.3,0.5) {$ 1 - c_{3}$};
    
	\node at (8,-2) {($x_{2,3,6}$, $y_{2,3,6}$)};

\end{tikzpicture}

\bigskip

\bigskip

\begin{tikzpicture}[roundnode/.style={circle, draw=black!80, thick, minimum size=7mm}, font=\small]

	\node[roundnode] at (-1,6) (maintopic) {9$(v)$} ;
 
    \draw[->][black!70] (-0.2,0) -- (1.2,0); 
    \draw[->][black!70] (0,-1.1) -- (0,5.2);

    \draw[brown] (1,5) -- (1,-1);
    \draw[brown] (1,-1) -- (0.6,-0.6);
    \draw[brown] (0.6,-0.6) -- (0.4,-0.3);
    \draw[brown] (0.4,-0.3) -- (0,0.5);
    \draw[brown] (0,0.5) -- (0,3.7);
    \draw[brown] (0,3.7) -- (0.3,4.3);
    \draw[brown] (0.3,4.3) -- (1,5);
   
	\draw[dashed, black!60] (1,5) -- (0,5);
	\draw[dashed, black!60] (0.6,-0.6) -- (0,-0.6);
	\draw[dashed, black!60] (0.6,-0.6) -- (0.6,0);
	\draw[dashed, black!60] (0.4,-0.3) -- (0.4,0);
	\draw[dashed, black!60] (0.4,-0.3) -- (0,-0.3);
	\draw[dashed, black!60] (0.3,4.3) -- (0.3,0);
	\draw[dashed, black!60] (0.3,4.3) -- (0,4.3);

	\node at (-0.5,5) {$\mu$};
	\node at (-1,4.3) {$\mu -1 + c_{2}$};
	\node at (-1,3.7) {$\mu -1 - c_{2} $};
	\node at (-0.5,0.5) {$c_{1}$};
	\node at (-0.8,-0.3) {$-c_{1} + 2 c_{3}$};
	\node at (-0.8,-0.6) {$-c_{1} - c_{3}$};
	\node[rotate=45] at (0.3,0.5) {$ c_{2} $};
	\node[rotate=45] at (0.7,0.5) {$ c_{1} - c_{3} $};
	\node[rotate=45] at (1.1,0.5) {$ c_{1} + c_{3}$};
    
	\node at (0,-2) {($x_{2,2,5}$, $y_{2,2,5}$)};

	\node[roundnode] at (7,6) (maintopic) {11$(v)$} ;
 
    \draw[->][black!70] (7.8,0) -- (9.2,0); 
    \draw[->][black!70] (8,-1.1) -- (8,5.2); 

    \draw[brown] (9,5) -- (9,-1);
    \draw[brown] (9,-1) -- (8.6,-0.6);
    \draw[brown] (8.6,-0.6) -- (8.4,-0.3);
    \draw[brown] (8.4,-0.3) -- (8.3,-0.1);
    \draw[brown] (8.3,-0.1) -- (8,0.8);
    \draw[brown] (8,0.8) -- (8,4);
    \draw[brown] (8,4) -- (9,5);
   
	\draw[dashed, black!60] (9,5) -- (8,5);
	\draw[dashed, black!60] (8.6,-0.6) -- (8,-0.6);
	\draw[dashed, black!60] (8.6,-0.6) -- (8.6,0);
	\draw[dashed, black!60] (8.4,-0.3) -- (8.4,0);
	\draw[dashed, black!60] (8.4,-0.3) -- (8,-0.3);
	\draw[dashed, black!60] (8.3,-0.1) -- (8.3,0);
	\draw[dashed, black!60] (8.3,-0.1) -- (8,-0.1);

	\node at (7.5,5) {$\mu$};
	\node at (7.2,4) {$\mu -1 $};
	\node at (7.5,0.8) {$c_{1} + c_{2} $};
	\node at (7.2,-0.1) {$c_{1} - 2c_{2}$};
	\node at (7.2,-0.5) {$-c_{1} + 2c_{3}$};
	\node at (7.2,-0.9) {$-c_{1}-c_{3}$};
	\node[rotate=45] at (8.3,0.5) {$ c_{2} $};
	\node[rotate=45] at (8.7,0.5) {$ c_{1} - c_{3} $};
	\node[rotate=45] at (9.1,0.5) {$ c_{1} + c_{3}$};
    
	\node at (8,-2) {($x_{2,3,5}$, $y_{2,3,5}$)};

\end{tikzpicture}
\end{center}
\caption{The auxiliary Delzant polytopes}\label{New relation toric pictures}
\end{figure}
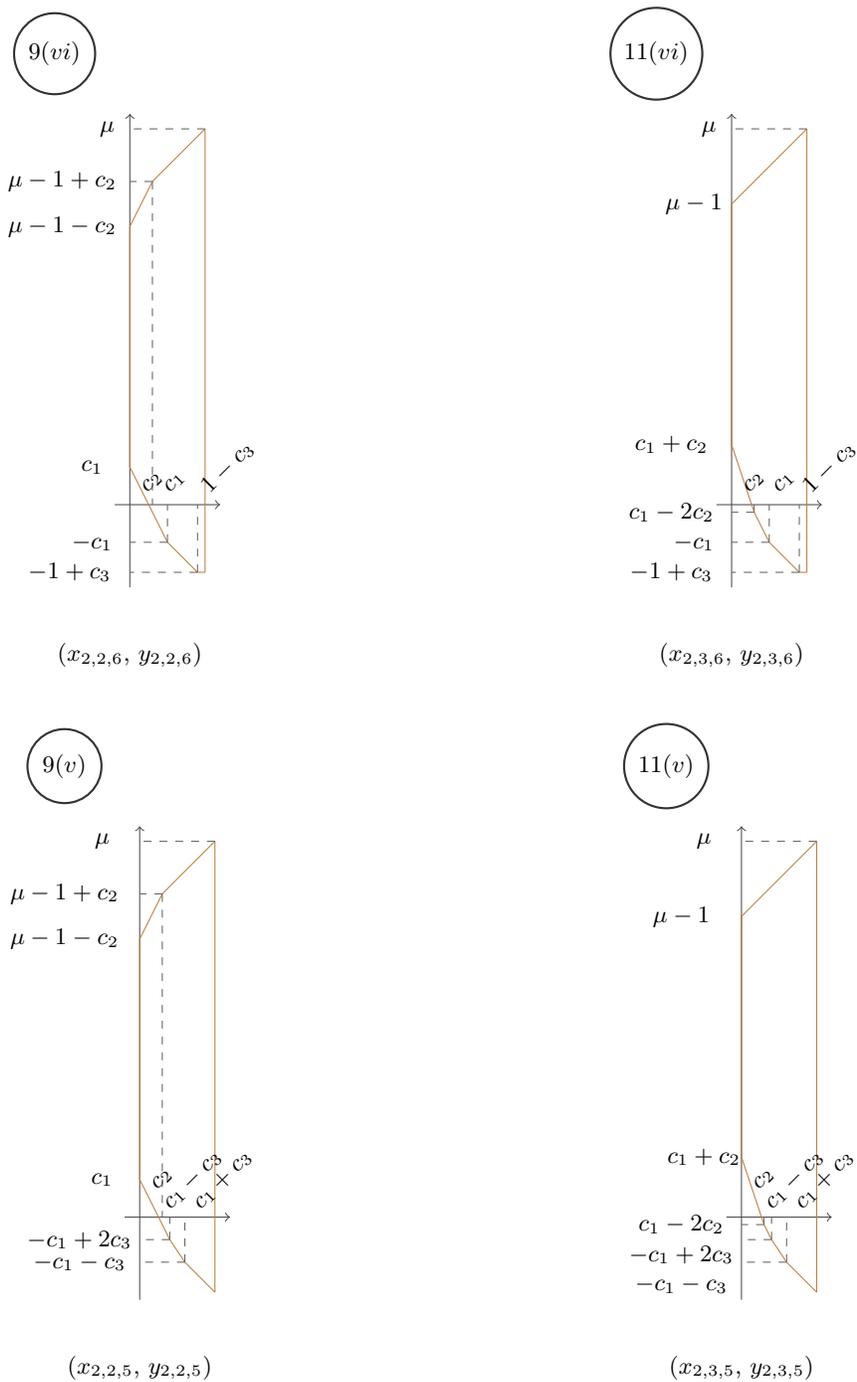

\bigskip

Using Karshon's classification one more time, we obtain the  relations below. As before, for each relation we also list the two polytopes  from which we deduce the relation through a $SL(2, \Z)$ transformation and/or projection onto one of the axis. 

\begin{align*}
& [(S1)] \ x_{2,2,5} = x_{2,3,5}, &  & 9(v)  \leftrightarrow 11(v);\\
& [(S2)] \ x_{2,2,6} = x_{2,3,6}, &  & 9(vi)  \leftrightarrow  11(vi);\\
& [(S3)] \  y_{2,2,5} = x_{1,5} + y_{1,5}, &  & 9(v)  \leftrightarrow  1(v);\\
& [(S4)] \ y_{2,2,6} = x_{1,4} + y_{1,4}, &  &  9(vi)  \leftrightarrow  1(iv);\\
& [(S5)] \ y_{2,3,5} = x_{3,3} + y_{3,3}, &  & 11(v)  \leftrightarrow  3(iii);\\
& [(S6)] \ y_{2,3,6} = x_{3,2} + y_{3,2}, &  & 11(vi)  \leftrightarrow  3(ii);\\
& [(S7)] \ x_{2,2,5} + y_{2,2,5} = x_{2,2,6} + y_{2,2,6}, &  & 9(v)  \leftrightarrow  9(vi);\\
& [(S8)] \ x_{2,3,5} + y_{2,3,5} = x_{2,3,6} + y_{2,3,6}, &  &  11(v)  \leftrightarrow  11(vi).
\end{align*}

Using (S8) together with (S1) and (S2), we get $ x_{2,2,5} + y_{2,3,5} = x_{2,2,6} + y_{2,3,6} $. Then (S7) yields $ y_{2,2,6} + y_{2,3,5} = y_{2,2,5} + y_{2,3,6} $. Combining this with (S3), (S4), (S5) and (S6) gives
$$ x_{1,4} + y_{1,4} + x_{3,3} + y_{3,3}   = x_{1,5} + y_{1,5} + x_{3,2} + y_{3,2}.$$
By the definition of the $x_i's$ and $y_i's$ given in Figures \ref{S1-actions x0-x5} and \ref{S1-actions y0-y5}, this is equivalent to
\begin{gather*}
 x_{5} + y_{10} + x_{2} + y_0 = x_{4} + y_{10} + x_{3} + y_0, \\
 x_{5} + x_{2} = x_{4} + x_{3}.
\end{gather*}
Finally using the type--0 relations (0.2) and (0.3) to replace $ x_{3} $ and $ x_{5} $ we obtain relation \eqref{the x-relation}:
\begin{gather*}
 y_0 - y_{3} + x_{6} + x_{2} = x_{4} + y_0 - y_{3} + x_{0}, \\
 x_{2} + x_{6} = x_{0} + x_{4}.
\end{gather*}

There is a more geometrical interpretation of this relation.  We trace the blow-ups of the symplectic manifold $ \widetilde{M}_{c_1} $ at the invariant surfaces of the action whose moment map corresponds to the first component of the torus action described in the blow-up of capacity $c_1$ of the Hirzebruch surface $ \mathbb{F}_{0} $. Tracking each one of the blow-ups via the arrows in Figure \ref{New relation graphs compared}, we can  see how the relation emerges. More precisely, it shows up by pairing the elements, that is,  $ x_{6} - x_{4} = x_{10} - x_{8} = x_{0} - x_{2} $. Using type--0 relations we note that indeed these are equivalent to \eqref{the x-relation}. Since the relation between the graphs clearly does not depend on the value $\mu$ and the map  $ \widetilde{T}_{i,j}(0) \rightarrow G_{\mu, c_{1}, c_{2}, c_{3}} $ induces an injective map in $\pi_1$, the relation must hold for $ \mu = 1 $ as well.

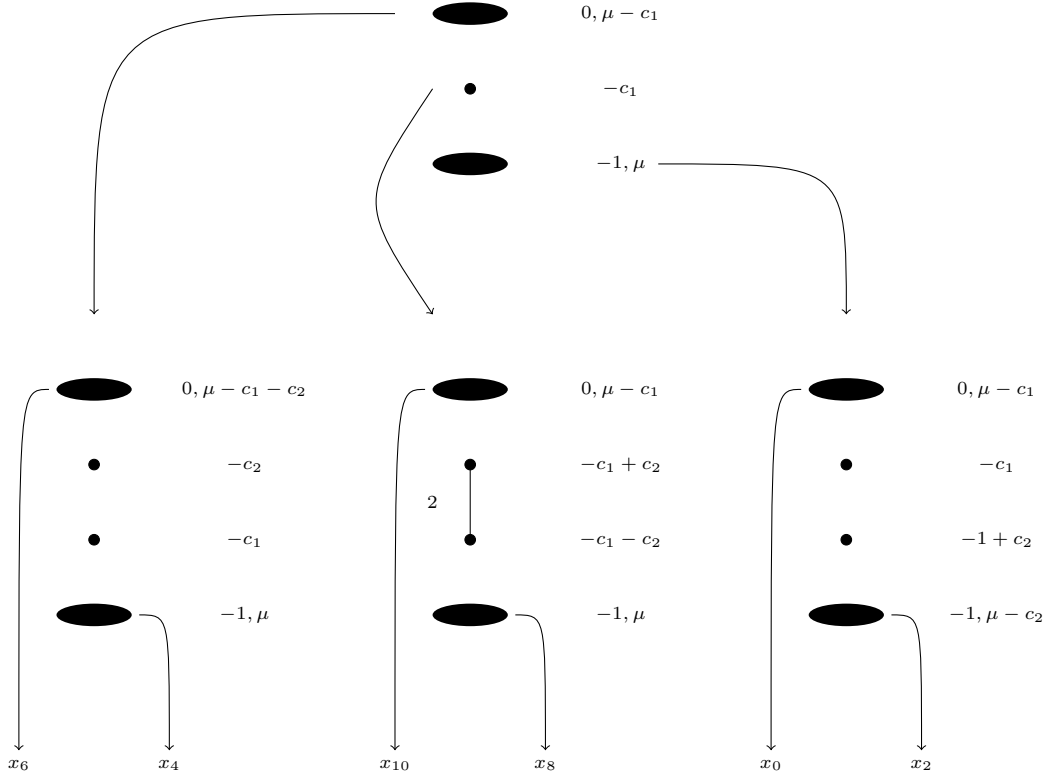
\begin{figure}[thp]
\begin{tikzpicture}[font=\scriptsize]

	\fill[black] (5,1) ellipse (0.5 and 0.15);
	\node(f1) at (7,1) {$0,\mu - c_{1}$};
	\filldraw [black] (5,0) circle (2pt);
	\node(second) at (7,0) {$ - c_{1} $};
	\fill[black] (5,-1) ellipse (0.5 and 0.15);
	\node(third) at (7,-1) {$-1,\mu$};        

	\draw[->] (4,1) .. controls (0,1) .. (0,-3);

	\fill[black] (0,-4) ellipse (0.5 and 0.15);
	\node(domates) at (2,-4) {$0,\mu - c_{1} - c_{2}$};
	\filldraw [black] (0,-5) circle (2pt);
	\node at (2,-5) {$ - c_{2} $};
	\filldraw [black] (0,-6) circle (2pt);
	\node at (2,-6) {$ - c_{1} $};
	\fill[black] (0,-7) ellipse (0.5 and 0.15);
	\node at (2,-7) {$-1,\mu$};

	\draw[->] (4.5,0) .. controls (3.5,-1.5) .. (4.5,-3);

	\fill[black] (5,-4) ellipse (0.5 and 0.15);
	\node at (7,-4) {$0,\mu - c_{1}$};
	\filldraw [black] (5,-5) circle (2pt);
	\node at (7,-5) {$ - c_{1} + c_{2} $};
	\filldraw [black] (5,-6) circle (2pt);
	\node at (7,-6) {$ - c_{1} - c_{2}$};
	\draw (5,-5) -- (5,-6);
	\node at (4.5,-5.5) {2};
	\fill[black] (5,-7) ellipse (0.5 and 0.15);
	\node at (7,-7) {$-1,\mu$};

	\draw[->] (7.5,-1) .. controls (10,-1) .. (10,-3);

	\fill[black] (10,-4) ellipse (0.5 and 0.15);
	\node at (12,-4) {$0,\mu - c_{1}$};
	\filldraw [black] (10,-5) circle (2pt);
	\node at (12,-5) {$ - c_{1} $};
	\filldraw [black] (10,-6) circle (2pt);
	\node at (12,-6) {$ -1 + c_{2} $};
	\fill[black] (10,-7) ellipse (0.5 and 0.15);
	\node at (12,-7) {$-1,\mu - c_{2}$};

	\node at (-1,-9) {$ x_{6} $};
	\node at (1,-9) {$ x_{4} $};
	\node at (4,-9) {$ x_{10} $};
	\node at (6,-9) {$ x_{8} $};
	\node at (9,-9) {$ x_{0} $};
	\node at (11,-9) {$ x_{2} $};
	
	\draw[->] (-0.6,-4) .. controls (-1,-4) .. (-1,-8.8);
	\draw[->] (0.6,-7) .. controls (1,-7) .. (1,-8.8);
	\draw[->] (4.4,-4) .. controls (4,-4) .. (4,-8.8);
	\draw[->] (5.6,-7) .. controls (6,-7) .. (6,-8.8);
	\draw[->] (9.4,-4) .. controls (9,-4) .. (9,-8.8);
	\draw[->] (10.6,-7) .. controls (11,-7) .. (11,-8.8);
\end{tikzpicture}
\caption{Comparison of the graphs of the $ S^{1} $-actions in the new relation \eqref{the x-relation}}\label{New relation graphs compared}
\end{figure}

By symmetry, we obtain relation \eqref{the y-relation}. More precisely, we have a similar picture for the y's, where we begin with the graph of Figure \ref{y-relation basic} and follow the blow-ups at analogous points, which yields $ y_{6} - y_{4} = y_{10} - y_{8} = y_0 - y_{2} $.
\begin{figure}[thp]
\begin{center}
\begin{tikzpicture}[font=\scriptsize]

	\fill[black] (5,1) ellipse (0.5 and 0.15);
	\node(f1) at (7,1) {$\mu,1$};
	\filldraw [black] (5,0) circle (2pt);
	\node(second) at (7,0) {$ c_{1} $};
	\fill[black] (5,-1) ellipse (0.5 and 0.15);
	\node(third) at (7,-1) {$0,1 - c_{1}$};        

\end{tikzpicture}
\end{center}
\caption{Graph of the $ S^{1} $-action which gives  relation \eqref{the y-relation} after blowing-up.}
\label{y-relation basic}
\end{figure}
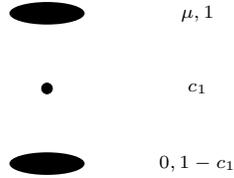
\end{proof}
\end{proposition}

Now we pick nine generators for the fundamental group $\pi_1(\Gccc) \otimes \Q$. We choose the actions   $ x_{0}$, $x_{1}$, $x_{2}$, $x_{4}$, $y_0$, $y_2$, $y_3$, $y_{4}$, $z=y_{8}-x_{4}$ and we show that the remaining elements can indeed be written  in terms of these generators. 
\begin{lemma}
\label{linear combinations}
Let $1 > c_1 +c_2 > c_1 +c_3 > c_1 >c_2 > c_3 >0$. Then the following identifications hold in $\pi_1 (\Gccc) \otimes \Q$. 
\begin{align*}
& x_3 = x_0 +y_0 -y_3 , & & x_{5} = x_0 -x_{2} + x_{4} +y_0 -y_3,  \\
&  x_{6} = x_{0} - x_{2} + x_{4},  & & x_{7} = x_{0} - x_{1} + x_{4},\\  
& x_{8} = y_{4} - z, &  & x_{9} = x_0- x_{1} +y_0 -y_3 + y_{4} - z, \\ 
& x_{10} = x_{0} - x_{2} + y_{4} -z,  & & x_{11} = x_0 +x_{1}-2x_{2} + y_0 -y_3 + y_{4} - z, \\ 
 & y_{1} = -x_{1} + x_{2} + y_{2},  & & y_{5}= -y_2 +y_3 +y_{4},  \\ 
 & y_{6} =   y_0 -y_{2}+ y_{4}, &  & y_{7} = x_{1} - x_{2}+y_0- y_2 + y_{4}, \\ 
 & y_{8} = x_{4} + z,  &  & y_{9}=  x_{1} - x_{2} + x_{4} - y_{2} + y_3 +z, \\
 & y_{10}= x_{4} + y_0 - y_{2} +z,  &  & y_{11}=-x_{1} +x_{2} +x_{4} -y_2 +y_3 +z.
 \end{align*}
\end{lemma}
\begin{proof}
The proof just follows from the type--0 relations together with relations  \eqref{the x-relation} and  \eqref{the y-relation}.
\end{proof}

\subsubsection{Other circle actions}\label{Other circle actions}

Besides the circle actions coming from the Delzant polytopes outlined above, there are other circle actions obtained by blowing up the toric actions at interior points of $ J $-holomorphic curves. Again, using Karshon's classification theorem, we can show that they can be expressed in terms of the generators chosen in the previous section. This is an easy  (yet cumbersome) task so we will not write the calculations up. Instead, we will show one elucidatory example on how the calculations go through, and leave the remaining ones for the interested reader. 

We recall Configuration 5.8 \footnote{This is a simplified version of the actual configuration 5.8, which further has a $ J $-holomorphic representative of the class $ B-E_{3} $, a curve that intersects $ E_{3} $  exactly once. We use the simplified version as it makes it clearer to follow the symplectic blow-downs.} in Figure \ref{Config5.8}. Although this configuration does not correspond to any 2-torus action, we observe that blowing down the exceptional curve $E_{3}$ gives Configuration 5 in Figure \ref{Config42 of AP}, which in turn corresponds to the toric action whose moment polytope is polyope \circled{5} in Figure \ref{Toric pictures}, with the corresponding $ S^{1} $-actions $ x_{0,5} $ and $ y_{0,5} $. Because the blow-up takes place at an interior point of the curve representing the class $F-E_{1}-E_{2}$, it can be traced on the circle action corresponding to the projection to y-axis, as can be seen in Figure \ref{An example of how S1-action is created}.

\begin{figure}[thp]
\begin{center}
\begin{tikzpicture}[scale=0.7,squarednode/.style={rectangle, draw=blue!60, fill=blue!5, very thick, minimum size=5mm}, font=\scriptsize]
	\node[squarednode] at (-3,3) (maintopic) {5.8} ; 
    \draw (-2.1,2) -- (0.1,2.2);
	\draw (-1.9,2.1) -- (-3.2,-0.1);
	\draw (-3.2,0.1) -- (-1.9,-2.1);
	\draw (-2.1,-2) -- (0.1,-2);
	\draw (-0.2,2.3) -- (1.3,0.7);
	\draw (-1.1,0.3) -- (-2.6,1.4);
	\node at (-1,2.3) {$E_{1}$};
	\node at (-4.4,1) {$F-E_{1}-E_{2}-E_{3} $};
	\node at (-3.1,-1) {$E_{2} $};
	\node at (-1,-2.3) {$B-E_{2}$};
	\node at (1.2,1.7) {$B-E_{1}$};
	\node at (-1, 0) {$E_{3} $};
\end{tikzpicture}
\end{center}
\caption{Configuration 5.8}\label{Config5.8}
\end{figure}
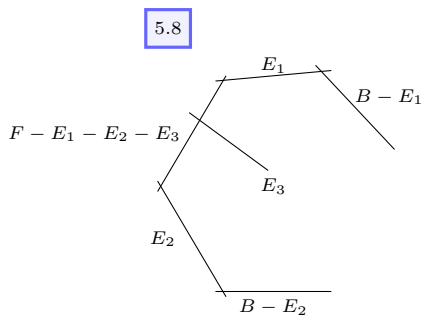

\bigskip

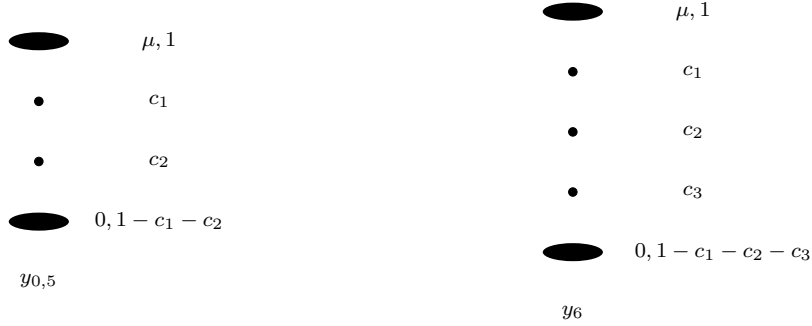
\begin{figure}[thp]
\begin{minipage}{.5\textwidth}
\begin{center}
\begin{tikzpicture}[scale=0.8,font=\footnotesize]
	\fill[black] (2,2) ellipse (0.5 and 0.15);
	\node at (4,2) {$\mu,1$};
	\filldraw [black] (2,1) circle (2pt);
	\node at (4,1) {$ c_{1} $};
	\filldraw [black] (2,0) circle (2pt);
	\node at (4,0) {$ c_{2} $};
	\fill[black] (2,-1) ellipse (0.5 and 0.15);
	\node at (4,-1) {$0,1 - c_{1} - c_{2}$};
	\node at (2,-2) {$ y_{0,5} $};
\end{tikzpicture}
\end{center}
\end{minipage}%
\begin{minipage}{.5\textwidth}
\begin{center}
\begin{tikzpicture}[scale=0.8,font=\footnotesize]
	\fill[black] (2,2) ellipse (0.5 and 0.15);
	\node at (4,2) {$\mu,1$};
	\filldraw [black] (2,1) circle (2pt);
	\node at (4,1) {$ c_{1} $};
	\filldraw [black] (2,0) circle (2pt);
	\node at (4,0) {$ c_{2} $};
	\filldraw [black] (2,-1) circle (2pt);
	\node at (4,-1) {$ c_{3} $};
	\fill[black] (2,-2) ellipse (0.5 and 0.15);
	\node at (4.5,-2) {$0,1 - c_{1} - c_{2} - c_{3}$};
	\node at (2,-3) {$ y_{6} $};
\end{tikzpicture}
\end{center}
\end{minipage}
\caption{Graph of the circle action corresponding to Configuration 5.8.}
\label{An example of how S1-action is created}
\end{figure}

Hence the circle action corresponding to Configuration 5.8, seen in $\pi _{1} (G_{c_{1}, c_{2}, c_{3}} )$, is equivalent to $ y_{6} $.

Finally, we note that not all configurations correspond to $S^{1}$-actions, and thereby cannot be read from Delzant polytopes. As mentioned before, there is exactly one configuration that does not correspond to any $S^{1}$-action, and it is when $B - E_{i} $ and  $F - E_{i} $ with $i=1,2,3$ are represented by $J$-holomorphic spheres. This is Configuration 1.13. In fact, J. Evans showed in \cite{Eva} that in the monotone case (i.e. when $ \mu = 1 $ and $ c_{1} = c_{2} = c_{3} = 1/2$), the symplectomorphism group is contractible. In this case, the automorphism group of the generic structure $ J_{0} $ is trivial as it is isomorphic to the stabilizer of four generic points in $ \mathbb{CP}^{2} $. The only possible configuration is 1.13, that is, the space of almost complex structures contains a single stratum, which is therefore contractible. In the non-monotone case, more $ J $-holomorphic curves are allowed to exist and that stratum corresponds to the open stratum of $ \mathcal{J} _{\mu=1 , c_{1} , c_{2}, c_{3} } $. So, the open stratum corresponding to this configuration does not yield new generators in $ \pi _{1} $.

\subsection{Samelson products} \label{sec m1:Samelsonproducts}

In the previous sections, we studied the circle actions $x_{i,j}, y_{i,j}$ for $ i=1,..,5 $ and $ j=1,...6 $, which are embedded into $ \pi _{1} (G_{c_{1} , c_{2}, c_{3} }) \otimes \Q $ and we established the type--0 relations, as well as the relations \eqref{the x-relation} and \eqref{the y-relation}. Note that we also have $ [ x_{i,j}, y_{i,j}] = 0$, where $[ \cdot, \cdot ]$ denotes the Samelson product in $ \pi_* (\Gccc) \otimes \Q$, since the actions $ x_{i,j}$ and $ y_{i,j}$ have commuting representatives in the torus action $ \widetilde{T}_{i,j}(0)$. Picking generators as in Lemma \ref{linear combinations} and writing down the afore-mentioned relations yield the linearly independent relations in the next Lemma. 
\begin{lemma}
\label{Samelson products}
The Samelson products of the generators $ x_{0}$, $x_{1}$, $x_{2}$, $x_{4}$, $y_0$, $y_2$, $y_3$, $y_{4}$, $z=y_{8}-x_{4} \in \pi_1(\Gccc)\otimes \Q$ satisfy the following relations:
\begin{gather}
  [x_{0}, y_0]  = [y_{0}, y_{3}],  \qquad   [x_{2}, y_2]  = [x_{1},x_{2}], \qquad [x_{2}, y_3]  = [x_{4},y_{3}],   \label{relations138} \\ 
  [x_{1}, y_0]    + [x_{1},y_{4}]=0, \label{relation2} \\ 
  [x_{2}, z] + [x_{2},x_{4}]  = [x_{0},x_{4}] + [x_{0},z] , \label{relation5} \\ 
   [y_0, z] +[y_2, y_{4}] =   [y_0, y_{4}] + [y_{2}, z] , \label{relation7}\\
    [x_{1}, x_{4}] +[x_{0},x_{1}] =  [x_{0},x_{2}] + [x_{2},x_{4}] ,  \label{relation4} \\ 
  [ y_{3},y_4] +  [y_{0},y_2]=  [y_{0},y_{4}] + [y_{2},y_{3}],  \label{relation6} \\
  [x_{2}, y_{3}]  + [y_{2}, y_4]+[y_{3}, z] =  [x_{2},y_{2}] + [y_{3},y_{4}] + [y_{2},z],  \label{relation9}\\
  [x_{1},x_{2}] + [x_{1},x_{4}] + [x_{1},y_{3}] +  [x_{1}, z]  = [x_{0},x_{4}] + [x_{4},y_{3}] + [x_{0},z]  \label{relation10}
   \end{gather}
   and all the following Samelson products vanish:
\begin{equation} \label{relation11-20}
[x_{0}, y_{2}], [x_{0}, y_{3}], [x_{0}, y_{4}],  [x_{1}, y_{2}], [x_{2}, y_0], [x_{2}, y_{4}], [x_{4},y_0],  [x_{4}, y_{2}], [x_{4},z] , [y_{4},z].
\end{equation}
 \end{lemma}
\begin{proof}
We show the proof for $[ x_{1} , y_{2}] = 0$ and relation \eqref{relation2}. The proof for the remaining relations is similar. 

Since $ x_{1}=x_{1,2} $ and $ y_{2}=y_{1,2} $ have commuting representatives in the torus action $ \widetilde{T}_{1,2}(0)$, we obtain immediately $ [ x_{1} , y_{2}] = 0 $. Similarly, using $ \widetilde{T}_{5,3}(0)$ we get 
\begin{align*}
0 &=  [x_{5,3},y_{5,3}]= [x_{1}, y_{6}]\\
     &= [x_{1},-y_{2}+y_0+y_{4}],  \quad \mbox{ by  Lemma \ref{linear combinations},}\\
   &= [x_{1},y_0+y_{4}], \quad  \mbox{ using $ [ x_{1} , y_{2}] = 0 $,}
\end{align*}
so we get relation \eqref{relation2}.  Writing all the relations $ [ x_{i,j}, y_{i,j}] = 0$ and simplifying them, we obtain the resulting list of 20 equalities.
\end{proof}

In addition to these relations coming from the toric pictures, there are two extra relations, of a different nature. 
\begin{proposition} \label{specialSamelson}
The Samelson products of the generators $x_0,x_2, x_4, y_0,y_2,y_{4} \in\pi_1(\Gccc) \otimes \Q$ satisfy the following relations in $\pi_2(\Gccc) \otimes \Q$: 
\begin{equation}\label{x-Samelson}
[x_{0}, x_{4}] = [x_{0}, x_{2}] + [x_{2}, x_{4}],
\end{equation}
\begin{equation}\label{y-Samelson}
[y_0, y_{4}] = [y_{0}, y_2] + [y_{2}, y_{4}].
\end{equation}
\end{proposition}
\begin{proof}
In order to prove these relations, we will need, as in Section \ref{section: new relation m1}, auxiliary polytopes and relations which arise more naturally in the setup for $ \mu > 1 $. However, we will need to introduce several new polytopes, given in Appendix \ref{app: relations m>1}, and deduce a total of 22 new relations. More precisely, consider the polytopes obtained by blowing-up the polytopes in Appendix \ref{app: relations m>1} at particular points, namely: 
\begin{itemize}
\item \circled{6} at  $(i)$  and $ (iii)$;
\item \circled{7} at  $(i)$  and $ (iii)$;
\item \circled{8} at  $(iii), (iv)$ and  $(vi)$;
\item \circled{9} at $(i), (iv)$ and  $(vi)$. 
\item \circled{10} at $(i)$ and $(ii)$.    
\end{itemize}

Then, using appropriate $SL(2,\Z)$ transformations and Karshon's classification we obtain the following relations: 

\begin{align*}
(r1) \  \ & x_{1,1,1}  = y_{1,5} - x_{1,5}, &  (r13)  \   \ y_{2,1,3} &  = x_{2,1} +y_{2,1}, \\
(r2) \  \ & x_{1,1,3} = y_{1,4} - x_{1,4}, &  (r14)  \   \ y_{2,1,4} & = x_{2,4} +y_{2,4},\\
(r3) \   \ & x_{1,2,1} = y_{2,3} - x_{2,3}, & (r15)  \  \  y_{2,1,6} & = x_{2,2} +y_{2,2}, \\
 (r4) \ \ & x_{1,2,3} = y_{2,2} - x_{2,2}, &  (r16)  \  \ x_{1,2,3} &  = y_{2,2} - x_{2,2}, \\
(r5) \   \ & x_{2,1,4} + y_{2,1,4} = 2 x_{1,1,1} + y_{1,1,1}, & (r17)  \   \ y_{2,2,1}  & = x_{1,1} +y_{1,1}, \\
 (r6) \  \ & x_{2,1,6} + y_{2,1,6} = 2 x_{1,1,3} + y_{1,1,3}, & (r18)  \  \  y_{2,2,4} & = x_{1,6} +y_{1,6}, \\
 (r7) \  \ & x_{2,2,4} + y_{2,2,4} = 2 x_{1,2,1} + y_{1,2,1},& (r19)  \   \ y_{2,2,6} & = x_{1,4} +y_{1,4}, \\
  (r8)  \  \ & x_{2,2,6} + y_{2,2,6} = 2 x_{1,2,3} + y_{1,2,3}, & (r20)  \  \ y_{2,5,1} & = x_{5,6} +y_{5,6},\\
 (r9) \  \ & x_{2,2,6} + y_{2,2,6} = 2 x_{1,2,3} + y_{1,2,3}, & (r21)  \  \  y_{2,5,2} & = x_{5,1} +y_{5,1}, \\
 (r10) \  \ & x_{2,1,3} - y_{2,1,3} = x_{2,2,1} - y_{2,2,1}, & (r22)  \   \ x_{2,1,6} & = x_{2,5,1}, \\
 (r11) \  \ & x_{2,1,4} - y_{2,1,4} = x_{2,2,4} - y_{2,2,4}, & (r23)  \  \ x_{2,1,3} & = x_{2,1,4} = x_{2,5,2} ,\\
 (r12) \  \ & x_{2,5,1} - y_{2,5,1} = x_{2,5,2} - y_{2,5,2}, &  (r24)  \   \ x_{2,2,1} & = x_{2,2,6},
  \end{align*}
where $(x_{i,j}, y_{i,j})$ are the circle actions in $\widetilde{T}_{i,j}(0)$ defined in Section \ref{sec m1: generators}. 
For example, in order to obtain the relation $(r1)$ we first perform the $SL(2,\Z)$ transformation 
$$ \left( \begin{array}{cc} 1 & 0 \\ -1 & 1 \end{array} \right)$$ 
to the polytope obtained from \circled{1} in Figure \ref{Toric pictures} blowing-up at the vertex $(v)$. Then projecting onto the $y$-axis we obtain the graph representing a circle action that we can identify with $ y_{1,5} - x_{1,5}$. Finally, we just observe that this graph is equivalent to the one of the circle action $x_{1,1,1}$, which is obtained by blowing-up polytope \circled{6} at vertex $(i)$ (see Figure \ref{Config10 of AP} with $k=1$) and then projecting onto the $x$-axis. 

The relations above together with the identifications in Figures \ref{S1-actions x0-x5} -- \ref{S1-actions y6-y11} and Lemma \ref{linear combinations} yield:
\begin{align*}
 & x_{1,1,1} = y_0-x_{3}, &
&  x_{1,2,1} = y_0-x_{5} = y_0 + x_{2} - x_{3} - x_{4},\\
& x_{1,1,3} = y_0-x_{2}, &
& x_{1,2,3} = y_0-x_{4},\\
& y_{2,1,3} = x_{4} + y_{2}, &
& y_{2,1,4} = x_{6} + y_{3} = 2x_{0} - x_{2} -x_{3} +x_{4} +y_0,\\
& y_{2,1,6} = x_{4} + y_0, &
& y_{2,2,4} = x_{0} + y_{3} = 2x_{0}-x_{3}+y_0,  \\
&   y_{2,2,1} = x_{0} + y_{2},&
&  y_{2,2,6} = x_{2} + y_0, \\
&  y_{2,5,1} = x_{0} + y_{4}, &
& y_{2,5,2} = x_{2} + y_{4}. 
\end{align*}
Define $ y := y_{1,1,1} $. 
Next, we express the elements $y_{1,2,1}, y_{1,2,3}$ and  $ y_{1,1,3}$ also in terms of the chosen generators for $\pi_1 (\Gccc) \otimes \Q$. This will be useful in the last step of the proof.  

First note that by relation $(r5)$ we have
\begin{equation}\label{x_214}
x_{2,1,4} = 2x_{1,1,1} + y_{1,1,1} - y_{2,1,4} = -2x_{0} + x_{2} - x_{3} -x_{4} +y_0 +y.
\end{equation}
Then by $(r11)$, $ x_{2,2,4} = x_{2,1,4} - y_{2,1,4} + y_{2,2,4} = -2x_{0} + 2x_{2} - x_{3} -2x_{4} +y_0 +y $. Hence, using $(r7)$, we obtain 
\begin{equation}\label{1st}
 y_{1,2,1} = x_{2,2,4} + y_{2,2,4} - 2x_{1,2,1} = y.
\end{equation}
For the second one combine relation \eqref{x_214} with relation $(r23)$, so we get $ x_{2,1,3} = x_{2,1,4} = -2x_{0} + x_{2} - x_{3} -x_{4} +y_0 +y $. Then it follows from $(r10)$ that  $ x_{2,2,1} = x_{2,1,3} - y_{2,1,3} + y_{2,2,1} = -x_{0} + x_{2} - x_{3} -2x_{4} +y_0 +y $, so that $(r24)$ implies $ x_{2,2,6} = x_{2,2,1} = -x_{0} + x_{2} - x_{3} -2x_{4} +y_0 +y $. Finally, by $(r8)$, we obtain
\begin{equation}\label{2nd}
 y_{1,2,3} = x_{2,2,6} + y_{2,2,6} -2 x_{1,2,3} = -x_{0} + 2x_{2} - x_{3} +y.
 \end{equation}
In order to obtain an expression for $y_{1,1,3}$ note that relation \eqref{x_214} together  with relation $(r23)$ yields $ x_{2,5,2} = x_{2,1,4} = -2x_{0} + x_{2} - x_{3} -x_{4} +y_0 +y $. Then, by $(r12)$, $ x_{2,5,1} = x_{2,5,2} - y_{2,5,2} + y_{2,5,1} = -x_{0} - x_{3} - x_{4} +y_0 +y $, therefore, it follows from $(r22)$ that  $ x_{2,1,6} = x_{2,5,1} = -x_{0} - x_{3} - x_{4} +y_0 +y $. Finally, by $(r6)$, we get 
\begin{equation}\label{3rd}
y_{1,1,3} = x_{2,1,6} + y_{2,1,6} -2 x_{1,1,3} = -x_{0} + 2x_{2} - x_{3} +y.
\end{equation}

We now have all the ingredients to complete the proof of the proposition.
Since $x_{1,1,1}$ and $y_{1,1,1}$ have commuting representatives in the torus action which corresponds to the polytope obtained from  \circled{6} blowing-up at the vertex $(i)$ (see Figure \ref{Config10 of AP} in Appendix \ref{app: relations m>1}), it follows that 
 $ 0=[x_{1,1,1},y_{1,1,1}]=[y_0-x_{3},y] $, so that
\begin{equation}\label{Samelson flower}
[x_{3},y] = [y_0,y].
\end{equation}
Similarly, blowing-up \circled{7} at $(i)$ it follows from \eqref{1st}  that $ 0=[x_{1,2,1},y_{1,2,1}]=[x_{2}-x_{3}-x_{4}+y_0,y] $, which, using relation \eqref{Samelson flower}, gives
\begin{equation}\label{Samelson daisy}
[x_{2},y] = [x_{4},y].
\end{equation}
Then, using the blow-up of \circled{6} at $(iii)$ and \eqref{3rd}, we obtain
\begin{align*}
  0 &= [x_{1,1,3},y_{1,1,3}]\\
 &= [-x_{2}+y_0, -x_{0}+2x_{2}-x_{3}+y]\\
  &= [x_{0},x_{2}] + [x_{2},x_{3}] - [x_{2},y] - [x_{0},y_0] + [y_0,y]\\
   &= [x_{0},x_{2}] + [x_{2},x_{3}] - [x_{4},y] - [x_{0},x_{3}] + [y_0,y],
\end{align*}
where we also used $ [x_{0}, y_0] = [x_{0}, x_{3}] $ from Lemma \ref{Samelson products} and relation \eqref{Samelson daisy}. This yields
\begin{equation}\label{Samelson tree}
[x_{4},y] = [x_{0},x_{2}] - [x_{0},x_{3}] + [x_{2},x_{3}] + [y_0,y].
\end{equation}
Finally,  the polytope obtained from blowing-up polytope \circled{8}  at $(iii)$ together with \eqref{2nd} gives
\begin{align*}
  0 &= [x_{1,2,3},y_{1,2,3}]\\
     &= [-x_{4}+y_0, -x_{0}+2x_{2}-x_{3}+y]\\
   &= [x_{0},x_{4}] - 2 [x_{2},x_{4}] + [x_{3},x_{4}] - [x_{4},y] - [x_{0},y_0] + [y_0,y].
\end{align*}
Using relations $ [x_{0}, y_0] = [x_{0}, x_{3}] $ and $ [x_{3}, x_{4}] = -[x_{0},x_{2}] + [x_{0},x_{4}] + [x_{2},x_{3}] $ from Lemma \ref{Samelson products} and relation \eqref{Samelson tree}, this becomes 
\begin{align*}
0 &= [x_{0},x_{4}] - 2 [x_{2},x_{4}]  - [x_{0},x_{2}] + [x_{0},x_{4}] + [x_{2},x_{3}] \\
    &\qquad \quad \quad  \  \ - ([x_{0},x_{2}] - [x_{0},x_{3}] + [x_{2},x_{3}] + [y_0,y]) - [x_{0},x_{3}] + [y_0,y]  \\
    &= 2[x_{0},x_{4}] - 2 [x_{2},x_{4}] -2[x_{0},x_{2}].
\end{align*}
Therefore we obtain relation \eqref{x-Samelson}. 

 In order to understand this new relation geometrically, we recall the graphs of these circle actions (see Figure \ref{S1-actions x0-x5}).

\begin{figure}[thp]

\begin{tikzpicture}[scale=0.7,font=\footnotesize]

	\node at (3,-3) {$x_{0}$};
	\fill[black] (2,2) ellipse (0.5 and 0.15);
	\node at (4,2) {$0,\mu - c_{1} - c_{3}$};
	\filldraw [black] (2,1) circle (2pt);
	\node at (4,1) {$ - c_{3} $};
	\filldraw [black] (2,0) circle (2pt);
	\node at (4,0) {$ - c_{1} $};
	\filldraw [black] (2,-1) circle (2pt);
	\node at (4,-1) {$-1+ c_{2} $};
	\fill[black] (2,-2) ellipse (0.5 and 0.15);
	\node at (4,-2) {$-1,\mu - c_{2}$};

	\node at (8.5,-3) {$x_{2}$};
	\fill[black] (7.5,2) ellipse (0.5 and 0.15);
	\node at (9.5,2) {$0,\mu - c_{1}$};
	\filldraw [black] (7.5,1) circle (2pt);
	\node at (9.5,1) {$ - c_{1} $};
	\filldraw [black] (7.5,0) circle (2pt);
    \tikzstyle{every node}=[font=\small]
	\node at (9.5,0) {$ -1 + c_{2} $};
	\filldraw [black] (7.5,-1) circle (2pt);
	\node at (9.5,-1) {$-1 + c_{3} $};
	\fill[black] (7.5,-2) ellipse (0.5 and 0.15);
	\node at (9.7,-2) {$-1,\mu - c_{2} - c_{3}$};

	\node at (14,-3) {$x_{4}$};
	\fill[black] (13,2) ellipse (0.5 and 0.15);
	\node at (15,2) {$0,\mu - c_{1} -c_{2}$};
	\filldraw [black] (13,1) circle (2pt);
	\node at (15,1) {$ - c_{2} $};
	\filldraw [black] (13,0) circle (2pt);
	\node at (15,0) {$ - c_{1} $};
	\filldraw [black] (13,-1) circle (2pt);
	\node at (15,-1) {$-1 + c_{3} $};
	\fill[black] (13,-2) ellipse (0.5 and 0.15);
	\node at (15,-2) {$-1,\mu - c_{3}$};
\end{tikzpicture}
\caption{The $ S^{1}$-actions $ x_{0} $, $ x_{2} $ and $ x_{4} $}\label{x0x2x4}
\end{figure}
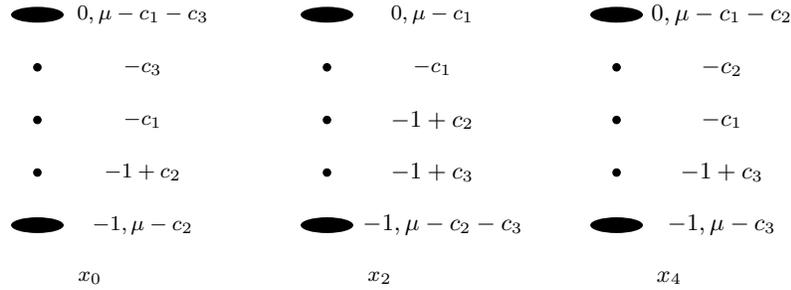

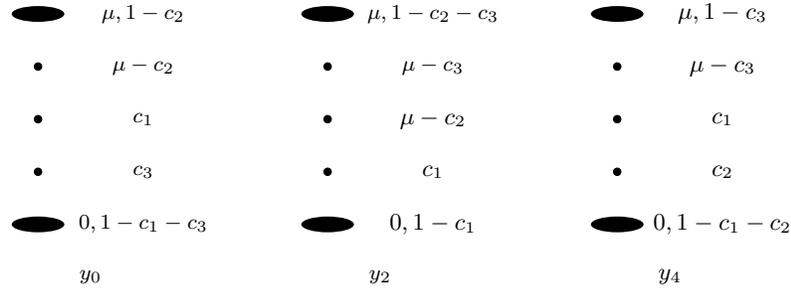
\begin{figure}[thp]
\begin{tikzpicture}[scale=0.7,font=\footnotesize]

	\node at (3,-3) {$y_0$};
	\fill[black] (2,2) ellipse (0.5 and 0.15);
	\node at (4,2) {$\mu, 1 - c_{2}$};
	\filldraw [black] (2,1) circle (2pt);
	\node at (4,1) {$ \mu - c_{2} $};
	\filldraw [black] (2,0) circle (2pt);
	\node at (4,0) {$ c_{1} $};
	\filldraw [black] (2,-1) circle (2pt);
	\node at (4,-1) {$c_{3} $};
	\fill[black] (2,-2) ellipse (0.5 and 0.15);
	\node at (4,-2) {$0,1 - c_{1} - c_{3}$};

	\node at (8.5,-3) {$y_{2}$};
	\fill[black] (7.5,2) ellipse (0.5 and 0.15);
	\node at (9.5,2) {$\mu,1 - c_{2} - c_{3}$};
	\filldraw [black] (7.5,1) circle (2pt);
	\node at (9.5,1) {$ \mu - c_{3} $};
	\filldraw [black] (7.5,0) circle (2pt);
    \tikzstyle{every node}=[font=\small]
	\node at (9.5,0) {$ \mu - c_{2} $};
	\filldraw [black] (7.5,-1) circle (2pt);
	\node at (9.5,-1) {$c_{1} $};
	\fill[black] (7.5,-2) ellipse (0.5 and 0.15);
	\node at (9.5,-2) {$0,1- c_{1}$};

	\node at (14,-3) {$y_{4}$};
	\fill[black] (13,2) ellipse (0.5 and 0.15);
	\node at (15,2) {$\mu,1  - c_{3}$};
	\filldraw [black] (13,1) circle (2pt);
	\node at (15,1) {$ \mu - c_{3} $};
	\filldraw [black] (13,0) circle (2pt);
	\node at (15,0) {$ c_{1} $};
	\filldraw [black] (13,-1) circle (2pt);
	\node at (15,-1) {$ c_{2} $};
	\fill[black] (13,-2) ellipse (0.5 and 0.15);
	\node at (15,-2) {$0,1 - c_{1} -c_{2}$};

\end{tikzpicture}
\caption{The $ S^{1}$-actions $ y_0$, $ y_{2} $ and $ y_{4} $}\label{y0y2y7}
\end{figure}

Looking at  Figure \ref{x0x2x4}, we see that the roles of $ c_{2} $ and $ c_{3} $ are interchanged in the actions $  x_{0}$ and $ x_{4} $, and relation \eqref{x-Samelson} tells us that these  "commute up" to $ x_{2}$.  As in Section \ref{section: new relation m1}, since the relation between  the graphs remains the same when the value of $\mu$ is changed, we conclude that this relation does not depend on its value. Since the map  $ T_{i,j} \rightarrow G_{\mu, c_{1}, c_{2}, c_{3}} $ allowing us to see these actions as elements of $ \pi_{1}(G_{\mu, c_{1}, c_{2}, c_{3}}) $ is injective in homotopy, the relation must also hold for $ \mu = 1 $.

Interchanging the roles of $ B $ and $ F $ (and flipping the graphs, which gives equivalent graphs) we see that we obtain the graphs of $ y_0$, $ y_{2} $ and $ y_{4} $, respectively (see Figure \ref{y0y2y7}). 

Hence, by symmetry, a similar result has to hold for these actions, that is, equation \eqref{y-Samelson} holds. This completes the proof of the proposition.
\end{proof}

\subsection{Proof of the main theorem}\label{proof m1}

We first give the complete statement of the main theorem.  

\medskip

\noindent \textbf{Theorem \ref{mainthm m1}.} \textit{Let $ \widetilde{M} _{c_{1} , c_{2}, c_{3} } $ denote the symplectic manifold obtained from $ (S ^{2} \times S ^{2}, \sigma \oplus \sigma )$, by performing three successive blow-ups of capacities of $ c_{1}$, $ c_{2}$ and $c_{3}$, with $ 1 > c_{1} + c_{2} > c_{1} + c_{3} > c_{1} > c_{2} > c_{3}. $ Let $ G_{c_{1} , c_{2}, c_{3} }$ denote the group of symplectomorphisms of $ \widetilde{M}_{c_{1} , c_{2}, c_{3} } $ acting trivially on homology. Define $ \widetilde{ \Lambda} $ as the Lie graded algebra over $\mathbb{Q}$ generated by elements of degree 1 denoted by $ x_{0}$, $ x_{1} $, $x_{2}$, $x_{4}$, $y_0$, $y_2$, $y_{3}$, $y_{4}$, $z$, where the Lie bracket satisfies the following relations
\begin{gather*}
 [x_{0}, y_0]  = [y_{0}, y_{3}],  \qquad   [x_{2}, y_2]  = [x_{1},x_{2}], \qquad [x_{2}, y_3]  = [x_{4},y_{3}],   \\ 
  [x_{1}, y_0]    + [x_{1},y_{4}]=0, \\ 
  [x_{2}, z]  = [x_{0},x_{2}] + [x_{0},z] ,   \qquad [y_0, z]  =   [y_0, y_{2}] + [y_{2}, z] ,\\
  [x_{0}, x_{4}] = [x_{0}, x_{2}] + [x_{2}, x_{4}], \qquad [y_0, y_{4}] = [y_{0}, y_2] + [y_{2}, y_{4}],\\
  [x_{0},x_{4}] = [x_{0},x_{1}]+  [x_{1}, x_{4}] ,  \qquad [ y_{3},y_4] =  [y_{2},y_{3}] + [y_{2},y_{4}], \\
  [x_{2}, y_{3}] +[y_{3}, z] =  [x_{2},y_{2}] + [y_{2},y_{3}] + [y_{2},z], \\
  [x_{1},x_{2}]  + [x_{1},y_{3}] +  [x_{1}, z]  = [x_{0},x_{1}] + [x_{4},y_{3}] + [x_{0},z] 
  \end{gather*} 
and the following Lie brackets vanish
\begin{equation*}
[x_{0}, y_{2}], [x_{0}, y_{3}], [x_{0}, y_{4}],  [x_{1}, y_{2}], [x_{2}, y_0], [x_{2}, y_{4}], [x_{4},y_0],  [x_{4}, y_{2}], [x_{4},z] , [y_{4},z].
\end{equation*}
Then there is an isomorphism of Lie graded algebras between $ \widetilde{ \Lambda} $ and 
$\pi _{*} (G_{\mu=1, c_{1} , c_{2}, c_{3} }) \otimes \mathbb{Q}. $}

\begin{remark}
Note that  the Lie algebra $ \pi _{*} (G_{c_{1} , c_{2}, c_{3} }) $ contains a subalgebra, generated by $ x_{2} $, $ x_{4} $, $ y_0 $, $ y_{4} $ and $ z $, that is isomorphic to the Lie graded algebra of $G_{\mu =1, c_{1} , c_{2} }$, described in \cite{AnjPin}. The underlying reason for this is that the $ S^{1} $-actions on the manifold $ \widetilde{M}_{\mu =1, c_{1} , c_{2}, c_{3} } $ are lifts of $ S^{1} $-actions on $ \widetilde{M}_{\mu =1, c_{1} , c_{2}} $. One easy way to see this is to consider the configurations 1.4, 2.2, 3.1, 4.1 and 5.1 in Section \ref{sec m1: configs}. These configurations correspond to toric pictures $ \widetilde{T}_{i,j}(0) $ with the same labelling and thus yield, respectively,
\begin{center}
$ [x_{2},y_0] = [x_{4},y_0] = [x_{4},z] = [y_{4},z] = [x_{2},y_{4}] = 0$.
\end{center}
Moreover, if we blow-down the exceptional curve $ E_{3} $ in those configurations, we observe that the ones obtained downstairs, in $ \widetilde{M}_{\mu =1, c_{1} , c_{2}} $, give the circle actions $ \bar{x_{0}}$, $\bar{y_0}$, $\bar{x_{1}}$, $\bar{z}$, $\bar{y_{1}}$ (in accordance with the notation in \cite{AnjPin}) which generate the algebra $\pi _{*} (G_{\mu =1, c_{1} , c_{2}})$. Therefore, we can conclude that the homotopy Lie algebra of $G_{\mu =1, c_{1} , c_{2} }$ lifts as a subalgebra of $\pi_{*}(G_{c_{1} , c_{2}, c_{3}})$.
\end{remark}

\begin{proof}

First, using the results from the previous sections, we show there is a well-defined  homomorphism of Lie algebras from $ \widetilde{\Lambda} $ into $ \pi_{*} (G_{c_{1},c_{2},c_{3}}) \otimes \Q$. Next, we will show that these two algebras have the same rank in each dimension so that this homomorphism is in fact an isomorphism.

The definition of the algebra $\widetilde{\Lambda}$ makes the definition of the homomorphism 
$$ \widetilde{\Lambda}  \longrightarrow \pi_{*} (G_{c_{1},c_{2},c_{3}}) \otimes \Q$$ obvious. Therefore  
each generator of $ \widetilde{\Lambda} $ corresponds to a generator of  $\pi_{1}(G_{c_{1},c_{2},c_{3}}) \otimes \Q$, which is induced by a circle action on the manifold. Note that the Lie bracket of two generators of $ \widetilde{\Lambda} $ is mapped, under the homomorphism, to the Samelson product of the correspondent generators in the fundamental group of $\Gccc$. Moreover, combining Lemma \ref{Samelson products} together with Proposition \ref{specialSamelson} we obtain a set of relations for the Samelson products that coincide with the relations in the statement of the theorem, so the homomorphism is a well-defined homomorphism of Lie algebras. 

Let $ \widetilde{ \lambda} _{n} $ be the rank of $ \widetilde{ \Lambda} $ in dimension $n$.  We now show that the two algebras have the same rank in each dimension.

We need to recall a theorem due to Milnor and Moore. Let $L$ be a graded Lie algebra and $TL$ the tensor algebra on the graded vector space $L$. Take the ideal
$$I := \{ x \otimes y \ - (-1)^{\deg x \deg y} y \otimes x - [x,y] \ : \ x,y \in L \} ,$$
where $ [\cdot , \cdot ] $ denotes the Lie bracket on $ L $. We define the \textit{universal enveloping algebra} $ \mathcal{U} L $ of $L$ as $TL / I $.

\begin{theorem}
(Milnor-Moore, \cite{MilMoo}) \label{Milnor-Moore-thm} If X is a simply-connected topological space then
\begin{enumerate}
\item[(1)] $ \pi _{*} ( \Omega X) \otimes \mathbb{Q} $ is a graded Lie algebra, $ L_{X} $;
\item[(2)] The Hurewicz homomorphism extends to an isomorphism of graded algebras
\begin{center}
$\mathcal{U} L_{X} \overset{\cong}{\longrightarrow} H_{*} ( \Omega X ; \mathbb{Q} )$
\end{center}
\end{enumerate}
\end{theorem}
Then, combining this result with the Poincar\'{e}-Birkoff-Witt Theorem (see \cite[Section 33]{FelHalTho}, as applied to topological spaces) it follows that $ \widetilde{ \lambda} _{n} $ satisfies
\begin{equation}\label{Poin}
\sum_{n=0}^{+ \infty}  \widetilde{h'_n} z^{n} = \frac{\prod_{n=0}^{\infty} (1+z^{2n+1})^{\widetilde{\lambda}_{2n+1}}}{ \prod_{n=1}^{\infty} (1-z^{2n})^{\widetilde{\lambda}_{2n}} }
\end{equation}
where the left hand side of the equation is the Poincar\'{e} series of the universal enveloping algebra of $ \widetilde{\Lambda} $, as described in the Milnor-Moore theorem.

In order to calculate $ \widetilde{ \lambda} _{n} $, we first establish some notation. Let $ \widetilde{h}_{n} = \dim H_{n} ( \Omega (\widetilde{M}_{c_{1} , c_{2} }) , \mathbb{Q}) $ and 
$ \widetilde{r} _{n} = \dim \pi _{n} ( \Omega (\widetilde{M}_{c_{1} , c_{2} }) \otimes \mathbb{Q} = \dim \pi _{n+1} (\widetilde{M}_{c_{1} , c_{2} }) \otimes \mathbb{Q} $.
The former, as explained earlier in Section \ref{sec m1: homotopy}, satisfies $ \widetilde{h}_{0} = 1 $, $ \widetilde{h}_{1} = 4 $ and $ \widetilde{ h}_{n} = 4 \widetilde{h}_{n-1} - \widetilde{h}_{n-2} $ for $ n \geq 2 $.

 We also need to recall some notation and results from \cite{AnjPin}. Let $ h_{n} = \dim H_{n} ( \Omega (\widetilde{M}_{c_{1}}) , \mathbb{Q})$.
Recall from \cite[Section 4]{AnjPin} that $h_{0}=1 $, $ h_{1}=3 $, and $h_{n}= 3 h_{n-1} - h_{n-2} $ for $ n \geq 2 $. Let $ \Lambda $ denote the Lie graded algebra $ \pi _{*} ( G_{c_{1},c_{2}}) \otimes \mathbb{Q} $, where $ G_{c_{1},c_{2}} $, for $ \mu \geq 1 $, is the group of symplectomorphisms, acting trivially on homology, of the sympletic manifold $ M_{\mu,c_{1},c_{2}} $ obtained from $ (S^{2} \times S^{2}, \sigma \oplus \sigma) $ by performing two sympletic blow-ups of capacities $ c_{1}$ and $c_{2} $. Let $ h'_{n}$ denote the coefficient of $ z^{n} $ of the Poincar\'{e} series of the universal enveloping algebra of $ \Lambda $. As also explained in \cite[Section 4]{AnjPin}, these satisfy $ h'_{0} =1 $ and $ h'_{n} = 5 h_{n-1} $ for $ n \geq 1 $. Finally, let $\lambda_{n}$ denote the rank of the algebra $ \Lambda $ in dimension $ n $.

\begin{proposition}
\label{rank proposition} We have $ \dim \pi_{1} (G_{c_{1},c_{2},c_{3}}) \otimes \Q = 9 $, and
$$ \dim \pi_{n} (G_{c_{1},c_{2},c_{3}})  \otimes \Q= \dim \pi_{n} (G_{c_{1},c_{2}})  \otimes \Q + \dim \pi_{n+1}(\widetilde{M}_{c_{1},c_{2}})  \otimes \Q,  \quad \mbox{for} \quad  n>1.$$
\end{proposition}

\begin{proof}

We recall the long exact sequence in the proof of Proposition \ref{rank upper limit}:
   $$ . . . \xrightarrow{ev_{*+1}} \pi_{*+1} (  \widetilde{M}_{c_{1} , c_{2} } ) \xrightarrow{\gamma_{*+1}} \pi _{*} (G_{c_{1} , c_{2}, c_{3} } ) \xrightarrow{\alpha_{*}}  \pi _{*} (G_{c_{1} , c_{2}} ) \xrightarrow{ev_{*}} \pi_{*} (  \widetilde{M}_{c_{1} , c_{2} } )  \xrightarrow{\gamma_{*}} . . .$$
   First note that $\pi_1(\Mcc)=0$ implies that the map 
   $$\pi_1 (\Symp_p(\Mcc)) \longrightarrow \pi_1 (\Symp(\Mcc))$$
   is surjective. This means that we can choose representatives of the generators which fix the point $p$. Moreover, if $\xi, \eta: S^1 \rightarrow \Symp(\Mcc)$, with $\xi (s)(p) = p$ and $\eta (t) (p)= p$, are such representatives then their Samelson product 
   $$ (s,t)  \mapsto [\xi,\eta](s,t) = \xi (s) \eta (t) \xi (s) ^{-1} \eta(t)^{-1}$$
 clearly  also fixes the point $p$. Therefore, since we know that the homotopy Lie graded algebra $\pi_*(\Symp(\Mcc))$ is generated by elements of degree 1 (induced by circle actions on $\Mcc$), it follows that the map $\alpha_*$ is surjective and hence the map 
 $$ \pi _{*} (G_{c_{1} , c_{2}} ) \xrightarrow{ev_{*}} \pi_{*} (  \widetilde{M}_{c_{1} , c_{2} } ) $$
 is trivial, which makes the map $ \gamma_{*+1} $ injective. This concludes the proof of the proposition.  
  \end{proof}

\begin{lemma}\label{rank lemma} With the above notation, we have $ \widetilde{\lambda}_{n} = \widetilde{r}_{n} + \lambda _{n}$.
\begin{proof}
First we prove that $  \widetilde{h'_n} = h'_{n} + 4 \widetilde{h'_{n-1}} - \widetilde{h'_{n-2}} $ for $ n \geq 2 $: to compute $  \widetilde{h'_n} $ for $ n \geq 2 $, we begin with $ h'_{n} $ elements coming from the enveloping algebra of $ \Lambda $ plus $ 4 \widetilde{h'_{n-1}} $ elements: the number of classes in dimension $ n-1 $ times the four remaining generators (we should add $ 9 \widetilde{h'_{n-1}} $ elements because, by Proposition \ref{rank proposition}, we have $ 9 $ generators in total, but $ 5 \widetilde{h'_{n-1}} $ are already counted in $ h'_{n} $). Then, due to the relations, we can check that we have to remove $\widetilde{h'_{n-2}} $ elements.

Next, we show that $  \widetilde{h'_n} = \sum_{i+j=n} \widetilde{h}_{i} h'_{j} $ by induction. It clearly holds for $ n=0 $. For $ n=1 $, the afore-mentioned results yield $ \widetilde{h}_{0} h'_{1} + \widetilde{h}_{1} h'_{0} = 1 \cdot 5 + 4 \cdot 1 = 9$, which is indeed the number $ \widetilde{h}_{1}^{'} $ of generators of $ \pi _{1} $. For the inductive step, we use $  \widetilde{h'_n} = h'_{n} + 4 \widetilde{h'_{n-1}} - \widetilde{h'_{n-2}} $, as follows:
\begin{align*}
  \widetilde{h'_n} &= h'_{n} + 4 \widetilde{h'_{n-1}} - \widetilde{h'_{n-2}}\\
     &= h'_{n} + 4(\widetilde{h}_{n-1} h'_{0} + \widetilde{h}_{n-2} h'_{1} + \widetilde{h}_{n-3} h'_{2} + ... + \widetilde{h}_{1} h'_{n-2} + \widetilde{h}_{0} h'_{n-1})\\
      &\qquad \ - (\widetilde{h}_{n-2} h'_{0} + \widetilde{h}_{n-3} h'_{1} + \widetilde{h}_{n-4} h'_{2} + ... + \widetilde{h}_{0} h'_{n-2}) \\
   &= \widetilde{h}_{0} h'_{n} + 4 h'_{n-1} + (4 \widetilde{h}_{1}-\widetilde{h}_{0})h'_{n-2} + ... \\
	&\qquad + (4\widetilde{h}_{n-3}-\widetilde{h}_{n-4}) h'_{2} + (4\widetilde{h}_{n-2}-\widetilde{h}_{n-3}) h'_{1} + (4\widetilde{h}_{n-1}-\widetilde{h}_{n-2}) h'_{0} \\
   &= \widetilde{h}_{0} h'_{n} + 4 h'_{n-1} + \widetilde{h}_{2} h'_{n-2} + ... + \widetilde{h}_{n-2} h'_{2} + \widetilde{h}_{n-1} h'_{1} + \widetilde{h}_{n} h'_{0} \\
   &= \widetilde{h}_{0} h'_{n} + \widetilde{h}_{1} h'_{n-1} + \widetilde{h}_{2} h'_{n-2} + ... + \widetilde{h}_{n} h'_{0} \\
  &= \sum _{i+j=n} \widetilde{h}_{i} h'_{j}
\end{align*}
so that 
\begin{align*}
 \frac{\prod_{n=0}^{\infty} (1+z^{2n+1})^{\widetilde{\lambda}_{2n+1}}}{ \prod_{n=1}^{\infty} (1-z^{2n})^{\widetilde{\lambda}_{2n}} } &  = \sum_{n=0}^{+ \infty}  \widetilde{h'_n} z^{n} = \sum_{n=0}^{+ \infty} ( \sum_{i+j=n} \widetilde{h}_{i} h'_{j}) z^{n}\\
  &= \left( \sum_{n=0}^{\infty} \widetilde{h}_{n} z^{n} \right) \left( \sum_{n=0}^{\infty} h'_{n} z^{n} \right)   \\
  &= \left( \frac{\prod_{n=0}^{\infty} (1+z^{2n+1})^{\widetilde{r}_{2n+1}}}{ \prod_{n=1}^{\infty} (1-z^{2n})^{\widetilde{r}_{2n}} } \right) \left( \frac{\prod_{n=0}^{\infty} (1+z^{2n+1})^{\lambda_{2n+1}}}{ \prod_{n=1}^{\infty} (1-z^{2n})^{\lambda_{2n}} } \right) \\
 &= \left( \frac{\prod_{n=0}^{\infty} (1+z^{2n+1})^{\widetilde{r}_{2n+1}+\lambda_{2n+1}}}{ \prod_{n=1}^{\infty} (1-z^{2n})^{\widetilde{r}_{2n}+\lambda_{2n} }} \right) .
\end{align*}

This gives the desired result.
\end{proof}
\end{lemma}

Combining Proposition \ref{rank proposition} and Lemma \ref{rank lemma}, we obtain
$$ \dim \pi_{n} (G_{c_{1},c_{2},c_{3}})  \otimes \Q= \dim \pi_{n} (G_{c_{1},c_{2}})  \otimes \Q+ \dim \pi_{n+1}(\widetilde{M}_{c_{1},c_{2}}) \otimes \Q = \lambda _{n} + \widetilde{r}_{n} = \widetilde{\lambda}_{n},$$
 so the algebras $ \pi_{n} (G_{c_{1},c_{2},c_{3}})  \otimes \Q$ and $ \widetilde{\Lambda} $ have the same rank in each dimension. It follows that  the homomorphism $ \widetilde{\Lambda} \rightarrow \pi_{*} (G_{c_{1},c_{2},c_{3}}) $ is in fact an isomorphism.
 \end{proof}

\begin{remark}
We note that the first statement of Proposition \ref{rank proposition} agrees with the result obtained by J. Li, T. J. Li and W. Wu in \cite[Table 5]{LiLiWu2} regarding the rank of the fundamental group, which was studied using different techniques from ours. 

\end{remark}

\begin{remark}\label{No need for c1c2c3<1}
The generators and the relations stated in Theorem \ref{mainthm m1} do not depend on assumptions as $ 1 > c_{1} + c_{2} + c_{3} $ and $ c_{1} > c_{2} + c_{3}  $, although there are some $ S^{1} $-actions and therefore some toric pictures that do not exist in those cases. They hold also when $ 1 \leq c_{1} + c_{2} + c_{3} $ and $ c_{1} \leq c_{2} + c_{3}  $, so the fundamental group and the rational homotopy Lie algebra of the group $ G_{c_{1},c_{2},c_{3}} $ do not depend on these conditions.
\end{remark}

\begin{remark} The relations in the statement of Theorem \ref{mainthm m1} imply that the following elements form a basis for $ \pi _{2} (G_{c_{1} , c_{2}, c_{3} }) \otimes \Q \simeq \Q^{14}$ as a vector space: $ [x_{0},x_{1}] $, $ [x_{0},x_{2}] $, $ [x_{0},y_0] $, $ [x_{0},z] $,  $ [x_{1},y_{4}] $, $ [x_{1},z] $, $ [x_{2}, x_{4}] $, $ [x_{2},y_{2}] $, $ [x_{2},y_{3}] $, $ [x_{4},y_{4}] $, $ [y_0, y_2] $, $ [y_2, y_3] $, $ [y_{2}, y_{4}] $, $ [y_{2}, z] $, which agrees with the rank predicted by Proposition \ref{rank proposition}  because $\pi_2(\Gcc) \otimes \Q \simeq \Q^5$ and $\pi_3 (\Mcc) \otimes \Q \simeq \Q^9$.
\end{remark}

\section{Applications}\label{applications}

Our main application of Theorem \ref{mainthm m1} is the computation of the rank of the homotopy groups of the group of symplectomorphisms, acting trivially in homology, $\Symp_h(\X_5, \omega)$ where $\X_5$ is $\bbcp^2\#\,5\overline{ \bbcp}\,\!^2$ and $\omega$ is a particular symplectic form as we explain below. The monotone case was first studied by P. Seidel in \cite{Seidel} where he proved  that $\pi_0$ is an infinite discrete group generated by square Lagrangian Dehn twists while J. Evans proved later, in \cite{Eva}, that this group is weakly homotopy equivalent to $\Diff^+(S^2,5)$, the group of orientation-preserving diffeomorphisms of $S^2$ preserving 5 points. On the other hand, D. McDuff in \cite{McD4} noticed that for a certain symplectic form such that the blow-up size is small and there is no Lagrangian spheres, the group 
$\Symp_h(\bbcp^2\#\,5\overline{ \bbcp}\,\!^2)$ is connected. Recently, J. Li, T-J. Li and W. Wu in \cite[Section 5.5]{LiLiWu2} gave a lower bound for the rank of the fundamental group of $\Symp_h(\bbcp^2\#\,5\overline{ \bbcp}\,\!^2)$ with an arbitrary symplectic form, which together with an upper bound obtained from D. McDuff in \cite{McD4} gives the rank of the fundamental group. Note that $\X_5$ equipped with the standard symplectic structure is symplectomorphic to   the symplectic manifold obtained from $ (S ^{2} \times S ^{2}, \sigma \oplus \sigma )$, by performing four successive blow-ups of capacities $ c_{1}$, $ c_{2}$, $c_{3}$, and $c_4$, $\Mucccc$, where the equivalence is given in a similar way to the case of $\X_4$. We will consider the particular case of a small blow-up when $\mu=1$. 

\begin{corollary}\label{mainapplication}
Let  $ 1 > c_{1} + c_{2} > c_{1} + c_{3} > c_1+c_4> c_{1} > c_{2} > c_{3}>c_4$ and $ G_{c_{1} , c_{2}, c_{3},c_4}$ denote the group of symplectomorphisms of $ \widetilde{M}_{c_{1} , c_{2}, c_{3},c_4 } $ acting trivially on homology. Then 
$$\dim \pi_n (\Gcccc) \otimes \Q = \dim \pi_n (\Gccc) \otimes \Q + \dim \pi_{n+1}(\Mccc) \otimes \Q.$$
In particular $\pi_1(\Gcccc) \otimes \Q \simeq \Q^{14}$. 
\end{corollary}
\begin{proof}
The conditions on the capacities of the blow-ups ensure, as in Lemma \ref{reduced m1}, that we have a reduced symplectic form. 
There are two main steps in the proof. First we need to prove that the group $ G_{c_{1} , c_{2}, c_{3},c_4}$ has the homotopy type of a stabilizer, that is, it is homotopy equivalent to $\Symp_p(\Mccc)$. This is similar to Lemma \ref{small c3 m1} and requires also a previous analysis of the stability of symplectomorphism groups, as in Section \ref{sec m1: stability}, using the inflation technique.  Therefore, the evaluation fibration
$$ \Symp_p(\Mccc) \longrightarrow \Symp(\Mccc) \xrightarrow{ev} \Mccc$$
induces a long exact homotopy sequence 
$$ . . . \xrightarrow{ev_{*+1}} \pi_{*+1} ( \Mccc ) \xrightarrow{\gamma_{*+1}} \pi _{*} (\Gcccc ) \xrightarrow{\alpha_{*}}  \pi _{*} (\Gccc ) \xrightarrow{ev_{*}} \pi_{*} ( \Mccc)  \xrightarrow{\gamma_{*}} . . .$$
Then, since the homotopy Lie algebra $\pi_*(\Gccc) \otimes \Q$ is  $\pi_1$ generated it follows, as in the proof of Proposition \ref{rank proposition}, that the map $ev_*$ is trivial, so the sequence splits and we obtain the desired result. In particular, we know that the fundamental group of $\Gccc$ has 9 generators and $\pi_2(\Mccc) \otimes \Q= H_2 (\Mccc; \Q)=\Q^5$ so the rank of the fundamental group of $\Gcccc$ is indeed 14.  
\end{proof}

\begin{remark}
Note that this result agrees with the computation of the rank of the fundamental group of $\Symp_h(\bbcp^2\#\,5\overline{ \bbcp}\,\!^2)$ by J. Li, T. -J. Li and W. Wu in \cite{LiLiWu}. More precisely, they claim that 
the rank of $ \pi_1 (\Symp_h(\bbcp^2\#\,5\overline{ \bbcp}\,\!^2)$ is equal to $N-5$ where $ N$ is the number of -2 symplectic spheres classes, which they prove is 19 in this particular case (see \cite[Table 6]{LiLiWu}).
\end{remark}

The other applications concern the Pontryagin ring and the rational cohomology algebra of the topological group $G_{c_{1},c_{2},c_{3}} $. We recall that for a topological group $ G $, the Pontryagin product in $ H_{*}(G;\mathbb{Z}) $ is related to the Samelson product in $ \pi_{*}(G) $ by the formula
\begin{center}
$ [x,y] = xy - (-1)^{\deg x \deg y} yx $, for $ x,y \in \pi_{*}(G) $,
\end{center}
where we simplify the notation by suppressing the Hurewicz homomorphism $ \rho : \pi_{*}(G) \rightarrow H_{*}(G;\mathbb{Z}) $. We denote by $ \mathbb{Q} \langle x_{1},...,x_{n} \rangle $ the free non-commutative algebra over $ \mathbb{Q} $ with generators $ x_{i} $. Then, applying the Milnor-Moore Theorem (Theorem \ref{Milnor-Moore-thm}) to the classifying space $ B G_{c_{1},c_{2},c_{3}}  $ of the space of symplectomorphisms of $ \widetilde{M}_{c_{1} , c_{2}, c_{3}} $, we obtain the following consequence of Theorem \ref{mainthm m1}.

\begin{corollary} Let $ 1 > c_{1} + c_{2} > c_{1} + c_{3} > c_{1} > c_{2} > c_{3}  $. Let $ G_{c_{1} , c_{2}, c_{3} }$ denote the group of symplectomorphisms of $ \widetilde{M}_{c_{1} , c_{2}, c_{3} } $ that act trivially on homology. Then the Pontryagin ring of $ G_{c_{1} , c_{2}, c_{3} }$ is given by
$$ H_{*}( G_{c_{1} , c_{2}, c_{3}} ;\mathbb{Q}) = \mathbb{Q} \langle x_{0}, x_{1}, x_{2}, x_{4}, y_0, y_2, y_3, y_{4}, z \rangle / R $$
where all generators have degree 1, and $ R $ consists of the relations in the statement of Theorem \ref{mainthm m1} together with $ x_{0}^{2}= x_{1}^{2}= x_{2}^{2} = x_{4}^{2}= y_0^{2}= y_2^{2}= y_{3}^{2}= y_{4}^{2}= z^{2} =0  $.
\end{corollary}

\begin{corollary}
Let $ 1 > c_{1} + c_{2} > c_{1} + c_{3} > c_{1} > c_{2} > c_{3}  $. Let $ G_{c_{1} , c_{2}, c_{3} }$ denote the group of symplectomorphisms of $ \widetilde{M}_{c_{1} , c_{2}, c_{3} } $ that act trivially on homology. Then the rational cohomology algebra of $ G_{c_{1} , c_{2}, c_{3} }$ is infinitely generated.
\end{corollary}
\begin{proof}
By the Cartan-Serre theorem (see \cite[Theorem 1.1 and Theorem 1.2]{Nov}), if the rational homology of $ G_{c_{1} , c_{2}, c_{3} }$ is finitely generated in each dimension, then its rational cohomology is a Hopf algebra for the cup product and the coproduct induced by the product in $ G_{c_{1} , c_{2}, c_{3} }$ and it is generated as an algebra by elements that are dual to the spherical classes in homology. Furhermore, the number of generators of odd dimension $ d $ appearing in the anti-symmetric part of the rational cohomology algebra is equal to the dimension of $ \pi_{d}( G_{c_{1} , c_{2}, c_{3} }) \otimes \mathbb{Q} $, and the number of generators of even dimension $ d $ appearing in the symmetric part of the rational cohomology algebra is equal to the dimension of $ \pi_{d}( G_{c_{1} , c_{2}, c_{3} }) \otimes \mathbb{Q} $. Therefore, we conclude that the rational cohomology algebra of $ G_{c_{1} , c_{2}, c_{3} }$ is infinitely generated.
\end{proof}

\section{Special Cases} \label{SpecialCases}
Up to this point we just studied the generic case, that is, we studied the topology of the group of symplectomorphisms $\Gccc$ when $ 1 > c_{1} + c_{2} > c_{1} + c_{3} > c_{1} > c_{2} > c_{3} $. In this section  we give some remarks about some of the remaining cases, which means that we allow equalities in these relations between the capacities. More precisely, we will consider particular cases when we can apply the same argument we used in the generic case, that is, when the symplectomorphism group $\Gccc$ has the homotopy type of a stabilizer, and therefore the capacity $c_3$ of the third  blow-up is small. Next we illustrate how to study these cases with two examples.

The first example is when $c_1=c_2=1/2 > c_3$. In this case there are several classes, namely  $E_1 -E_2$, $E_1 -E_2 -E_3$, $ B-E_1 -E_2$, $ B-E_1 -E_2 -E_3$,  $F-E_1 -E_2$ and $F-E_1 -E_2-E_3$, which cannot be represented by a $J$-holomorphic sphere for all $J$ in the space of almost complex structures, that is, not all configurations of Section \ref{sec m1: configs} can be realized by $J$-holomorphic spheres. The isometries corresponding to the configurations that still exist in this case contain the following $S^1$-actions: $ x_{0}, x_{1}, x_{2}, x_{3}, y_0, y_1,y_2,y_3$ with the relations $y_3-x_0=y_0-x_3$ and $y_1-x_2=y_2-x_1$. 
Moreover, the evaluation fibration 
$$ \Symp_p (\widetilde{M}_{1/2 , 1/2 }) \rightarrow \Symp(\widetilde{M}_{1/2 , 1/2 }) \xrightarrow{ev} \widetilde{M}_{1/2,1/2 }  $$
can be still used as the space $G_{1/2,1/2,c_{3}} $ is homotopy equivalent to $ \Symp_{p} (\widetilde{M}_{1/2 , 1/2 })$. The argument for the proof of this fact is similar to the one of  Lemma \ref{small c3 m1}. 
Since $ G_{1/2 , 1/2 } $ is a torus (see \cite{Eva} or \cite{AnjPin})  and $\pi_2 (\widetilde{M}_{1/2 , 1/2 }) \otimes \Q \simeq \Q^4$, it follows that the long exact sequence 
\begin{equation*}
0    \rightarrow \pi_{2} (  \widetilde{M}_{1/2 , 1/2 } ) \rightarrow \pi _{1} (G_{1/2 , 1/2, c_3 } ) \rightarrow \pi _{1} (G_{1/2 , 1/2 }) \rightarrow \pi_{1} (  \widetilde{M}_{1/2 , 1/2 } ) \rightarrow \hdots
\end{equation*}
yields that the rank of the fundamental group of $G_{1/2,1/2,c_{3}} $ is 6 in this case. Furthemore, as generators of $\pi_1 ( G_{1/2,1/2,c_{3}}) \otimes \Q$,  and therefore, as generators of the homotopy Lie algebra $\pi_* ( G_{1/2,1/2,c_{3}}) \otimes \Q$, we can pick a subset of the generators of the generic case, namely $ \{ x_{0}, x_{1}, x_{2}, y_0, y_2, y_3 \} $.  It follows that the algebra $\pi_* ( G_{1/2,1/2,c_{3}}) \otimes \Q$ is a Lie subalgebra of the one in the generic case. 
Note that, on one hand, the rank of the fundamental group agrees with the result obtained by Li, Li and Wu in \cite{LiLiWu2} (this case  corresponds to the case $\lambda=1; c_1 =c_2=c_3 >c_4$, in their notation, in Table 5). On the other hand, it follows from the long exact sequence that 
\begin{equation*}
0    \rightarrow \pi_{3} (  \widetilde{M}_{1/2 , 1/2 } ) \rightarrow \pi _{2} (G_{1/2 , 1/2, c_3 } ) \rightarrow 0
\end{equation*}
since $\pi _{i} (G_{1/2 , 1/2 }) \simeq \pi _{i} (T^2) \simeq 0$ if $i \geq 2$, so $\pi _{2} (G_{1/2 , 1/2, c_3 } ) \simeq \pi_{3} (  \widetilde{M}_{1/2 , 1/2 } ) \simeq \Q^9 $ and it is simple to check that the generators of $\pi_2(G_{1/2,1/2,c_{3}})$ as a vector space can be given by the following nine Samelson products: $[x_0,x_1], [x_0,x_2],[x_0,y_0], [x_1,y_0], [x_1,y_3], [x_2,y_2], [x_2,y_3], [y_0,y_2]$ and $ [y_2,y_3]$. 

Similarly, we can study the case   $1 = c_{1} + c_{2} > c_{1} + c_{3} > c_{1} > c_{2} > c_{3}$ using the evaluation fibration since we can see the symplectomorphism group again as a stabilizer and thereby deduce information using the results for the case of two blow-ups. More precisely, $ G_{c_{1}+c_{2}=1, c_{3}} \simeq \Symp _{p} (\widetilde{M}_{c_{1}+c_{2}=1}) $, so,  using  $ \pi _{1}(G_{c_{1}+c_{2}=1}) \otimes \Q \simeq \mathbb{Q}^{3} $ (by \cite{AnjPin}) and  $ \pi _{2}(\widetilde{M}_{c_{1}+c_{2}=1}) \otimes \Q \simeq \mathbb{Q}^{4} $ it follows that the long exact sequence of the fibration
$$ \Symp  _{p} (\widetilde{M}_{c_{1} + c_{2}=1 }) \longrightarrow \Symp(\widetilde{M}_{c_{1} + c_{2}=1 }) \xrightarrow{ev} \widetilde{M}_{c_{1} + c_{2}=1}, $$
yields  rank 7 for the fundamental group, which coincides with Li, Li and Wu computation ($\lambda <1; c_1=c_2=c_3 > c_4$ in \cite[Table 5]{LiLiWu}).  In this case the $S^1$-actions contained in the isometries are represented by the elements $ x_{0}, x_{1}, x_{2}, x_{3}, y_0, y_1,y_2,y_3, z$, so we can choose as a set of generators of the homotopy Lie algebra  $\pi_*(G_{c_{1}+c_{2}=1, c_{3}}) \otimes \Q$ the following subset: $ \{ x_{0}, x_{1}, x_{2}, y_0, y_2, y_3, z \} $.

We remark that in all these cases we are using the fact that the map $ev_*$  in the long exact sequence, induced in homotopy by the evaluation map is trivial. This follows, as in the proof of Proposition \ref{rank proposition}, from the manifold being simply connected together with the fact that the homotopy Lie algebra $\pi_* (\Symp (\Mcc))$ is generated by elements of degree 1 in all  cases. Therefore, as is the generic case, the homotopy Lie algebras of these special cases contain subalgebras isomorphic to homotopy Lie algebras of the special cases of the two blow-up. 

In the cases the group $\Gccc$ does not have the homotopy type of  a stabilizer we cannot compute its homotopy groups as before. For example, when $c_1 + c_2 = c_1 +c_3=1$ and $ c_1 > c_2$ or the monotone case, that is, when $ c_1 =c_2 = c_3 = 1/2$.  The former case was studied by J. Evans in \cite{Eva}, as explained before: since there is only one stratum, the symplectomorphism group $G_{1/2,1/2,1/2}$ is homotopy equivalent to the group of isometries of the complex structure in that stratum, which is trivial, so the group is contractible. For the remaining cases we need to use different techniques to study the group of symplectomorphisms, in particular we need to study the limit of $\Guccc$ as $\mu \to \infty$, so we will leave this study out of this paper. Nevertheless, we conjecture that 
in all special cases the homotopy Lie algebra of the symplectomorphism group $\Gccc$ is a subalgebra of the one for the generic case. Moreover, we should remove from the set of generators the ones which correspond to configurations that are not allowed to exist, because some particular homology classes are not represented by $J$-homolorphic spheres.

\begin{appendices}

\section{Proof of Theorem \ref{slight changes m1}}\label{pf inflation}

This section is dedicated to complement the proof of Theorem \ref{slight changes m1}. As the calculations are similar in all cases, we will confine ourselves with listing the curves used in inflation. Note that the existence of such curves follows from Lemmas \ref{existenceJ-curves_generic} and \ref{existenceJ-curves}. Moreover, we need to be careful so that our choices do not lay in the list of exceptions of Lemma \ref{existenceJ-curves}.

\textit{Step 1:} $ \mathcal{A}_{c'_{1} , c_{2} , c_{3}} = \mathcal{A}_{c_{1} , c_{2} , c_{3}} $ for $ c_{1} \leq c'_{1} $.

The inclusion $ \mathcal{A}_{c'_{1} , c_{2} , c_{3}} \subset \mathcal{A}_{ c_{1} , c_{2} , c_{3}} $ was proved previously. To show $ \mathcal{A}_{ c_{1} , c_{2} , c_{3}} \subset \mathcal{A}_{c'_{1} , c_{2} , c_{3}} $, we use Table \ref{inflation c1 to c1'} to pick the curves for each configuration.

\begin{table}[thp]
\begin{center}
\begin{tabular}{ |m{11em}|m{25em}| } 
 \hline
 Configurations & Curves  \\ \hline
 1: 1-13; 2: 1,2,3,6,7,9,10 & $B$, $ B+F-E_{1} $, $ B+F-E_{2} $, $ B+2F-E_{1}-E_{2}-E_{3} $  \\ \hline
 2: 4,5,8 & $B$, $ B+F-E_{1} $,   $ B+F-E_{2} $, $ B+3F-E_{1}-E_{2}-E_{3} $  \\ \hline
3: 1,2,6; 6: 1,2,7 &  $B$, $ B+F-E_{1} $, $ B+F-E_{3} $, $ B+2F-E_{1}-E_{2}-E_{3} $ \\ \hline
 3: 3; 6: 3 & $B$,  $ B+F-E_{1} $, $ B+2F-E_{1}-E_{2} $, $ F-E_{1}-E_{3} $ \\ \hline
 3: 4,8; 6: 4,6 & $B$, $ B+F-E_{1} $, $ B+2F-E_{1}-E_{2} $, $ E_{1}-E_{2}-E_{3} $ \\ \hline
 3: 5,7; 6: 5 & $B$, $ B+F-E_{1} $, $ B+2F-E_{1}-E_{2} $, $ B+3F-E_{1}-E_{2}-E_{3} $ \\ \hline
\end{tabular}
\end{center}
\bigskip 
\caption{Curves used in the inflation process to show that $ \mathcal{A}_{ c_{1} , c_{2} , c_{3}} \subset \mathcal{A}_{c'_{1} , c_{2} , c_{3}} $}\label{inflation c1 to c1'}
\end{table}
\begin{remark} Note that in table \ref{inflation c1 to c1'} the first two rows look quite similar. However, we cannot simply interchange their roles, as the configurations in the second row have $ B-E_{1}-E_{2}-E_{3} $ represented by a $ J $-holomorphic curve and indeed the proof in this case uses the inequality $ c_{1} + c_{2} + c_{3} <1 $. In the first row, on the other hand, there is no a priori reason to assume this inequality.
\end{remark}

Since interchanging the roles os $ B $ and $ F $ in the configurations  of types  2 and  3 yield the configurations  of types 5 and  4, respectively, an obvious adaptation of Table \ref{inflation c1 to c1'} implies all the cases.

\textit{Step 2:} $ \mathcal{A}_{c_{1} , c'_{2} , c_{3}} = \mathcal{A}_{c_{1} , c_{2} , c_{3}} $ for $ c_{2} \leq c'_{2} $.

Starting with the inclusion $ \mathcal{A}_{c_{1} , c'_{2} , c_{3}} \subset \mathcal{A}_{c_{1} , c_{2} , c_{3}} $, using negative inflation along the curve $ E_{2} $ suffices for the following configurations:  1: 1,4-8,10-13;  2: 1-4,7,8,10;  3: 1,2,3,6,8;  6: 1,2,3,6,7. For the remaining cases, we use $ B+F+E_{1} $, $ B+F $ and $ E_{2} - E_{3} $.

\begin{table}[thp]
\begin{center}
\begin{tabular}{ |m{10em}|c|m{15em}| } 
\hline
Configurations &  & Curves \\ \hline
 1: 1-13;  2: 1-3,6,7,9,10 &  & $B$, $ B+F-E_{1} $, $ B+F-E_{2} $, $ B+2F-E_{1}-E_{2}-E_{3} $ \\ \hline
 2: 4,5,8 &  & $B$, $ B+F-E_{1} $, $ B+F-E_{2} $, $ B+3F-E_{1}-E_{2}-E_{3} $ \\ \hline
\multirow{2}{10em}{3: 1,2,7;  6: 1,2,5,7 } & $ 2c_{1}<1 $ & $B$, $ B+F $, $ B+2F-E_{1}-E_{2} $, $ B+2F-E_{1}-E_{2}-E_{3} $ \\ 
& $ 2c_{1} \geq 1 $ & $B$, $ B+F-E_{1} $, $ B+2F-E_{1}-E_{2} $, $ B+2F-E_{1}-E_{2}-E_{3} $ \\ \hline
\multirow{2}{10em}{3: 3,4,8;  6: 3,4,6 } & $ 2c_{1}+2c_{3}<1 $ & $B$, $ B+F $, $ B+2F-E_{1}-E_{2} $, $ E_{1}-E_{2}-E_{3} $ \\ 
& $ 2c_{1} + 2c_{3} \geq 1 $ & $B$, $ B+F-E_{1} $, $ B+2F-E_{1}-E_{2} $, $ E_{1}-E_{2}-E_{3} $ \\ \hline
 3: 5,6 &  & $B$, $ B+F-E_{1} $, $ B+3F-E_{1}-E_{2}-E_{3} $, $ E_{1}-E_{2} $ \\ \hline
\end{tabular}
\end{center}
\bigskip
\caption{Curves used in the inflation process to show that $ \mathcal{A}_{ c_{1} , c_{2} , c_{3}} \subset \mathcal{A}_{c_{1} , c'_{2} , c_{3}} $}\label{inflation c2 to c2'}
\end{table}

For the reverse inclusion, we use Table \ref{inflation c2 to c2'}. Here too, since interchanging the roles of $ B $ and $ F $ yield  configurations of types 4 and  5 too, again an adaptation of this table covers all the possible configurations.

\textit{Step 3:} $ \mathcal{A}_{c_{1} , c_{2} , c'_{3}} = \mathcal{A}_{c_{1} , c_{2} , c_{3}} $ for $ c_{3} \leq c'_{3} $.

To show the inclusion $ \mathcal{A}_{c_{1} , c_{2} , c'_{3}} \subset \mathcal{A}_{c_{1} , c_{2} , c_{3}} $, it suffices to use negative inflation over $ E_{3} $, for, by Lemma \ref{E3 m1}, $ E_{3} $ is represented by an embedded $ J $-curve for all $ J $.

For the reverse inclusion, we pick the curves in Table \ref{inflation c3 to c3'}.

\begin{table}[thp]
\begin{center}
\begin{tabular}{ |m{10em}|c|m{15em}| } 
\hline
Configurations &  & Curves \\ \hline
 1: 1,4,7,8,10,11,13;  2: 1,2,10;  3: 1,2;  6: 1,2,7 &  & $B$, $ B+F-E_{1} $, $ B+F-E_{3} $, $ B+2F-E_{1}-E_{2}-E_{3} $ \\ \hline
\multirow{2}{10em}{1: 2,3,9;  2: 6,9 } & $ 2c_{1}<1 $ & $B$, $ B+F-E_{2} $, $ B+2F-E_{1}-E_{2}-E_{3} $, $ E_{2}-E_{3} $ \\ 
& $ 2c_{1} \geq 1 $ & $B$, $ B+F-E_{1} $, $ B+2F-E_{1}-E_{2}-E_{3} $, $ E_{2}-E_{3} $ \\ \hline
 1: 5,6,12;  2: 3,7 &  & $B$, $ B+F-E_{1} $, $ B+2F-E_{1}-E_{2}-E_{3} $, $ E_{1}-E_{3} $ \\ \hline
\multirow{2}{10em}{2: 4,8} & $ c_{1}+2c_{2}<1 $ & $B$, $ B+F-E_{1} $, $ B+3F-E_{1}-E_{2}-E_{3} $, $ E_{1}-E_{3} $ \\ 
& $ c_{1}+2c_{2}\geq 1 $ & $B$, $ B+F-E_{1} $, $ B+3F-E_{1}-E_{2}-E_{3} $, $B+2F-E_{1}-E_{2} $ \\ \hline
 2: 5 &  & $ B+F-E_{1} $, $ B+F-E_{2} $, $ E_{2}-E_{3} $ \\ \hline
\multirow{2}{10em}{3: 3,8;  6: 3,6 } & $ 2c_{1}+2c_{2}<1 $ & $B$, $ B+F $, $B+2F-E_{1}-E_{3} $, $ E_{1}-E_{2}-E_{3} $ \\ 
& $ 2c_{1} + 2c_{2} \geq 1 $ & $B$, $ B+F-E_{1} $, $ B+2F-E_{1}-E_{3} $, $ E_{1}-E_{2}-E_{3} $ \\ \hline
\multirow{2}{10em}{3: 4,5,7;  6: 4,5 } & $ 2c_{1}<1 $ & $B$, $ B+F $, $B+2F-E_{1}-E_{2} $, $ E_{2}-E_{3} $ \\ 
& $ 2c_{1} \geq 1 $ & $B$, $ B+F-E_{1} $, $B+2F-E_{1}-E_{2} $, $ E_{2}-E_{3} $ \\ \hline
\multirow{2}{10em}{3: 6 } & $ c_{1}+2c_{2}<1 $ & $B$, $ B+F-E_{1} $, $B+3F-E_{1}-E_{2}-E_{3} $, $ B+F-E_{3} $ \\ 
& $ c_{1}+2c_{2} \geq 1 $ & $B$, $ B+F-E_{1} $, $B+3F-E_{1}-E_{2}-E_{3} $, $ B+2F-E_{1}-E_{2} $ \\ \hline
\end{tabular}
\end{center}
\bigskip
\caption{Curves used in inflation process to show that $ \mathcal{A}_{ c_{1} , c_{2} , c_{3}} \subset \mathcal{A}_{c_{1} , c_{2} , c'_{3}} $}\label{inflation c3 to c3'}
\end{table}

Finally, as noted previously, an adaptation of Table \ref{inflation c3 to c3'} covers the configurations  of types 4 and  5 as well. This finishes Step 3 and the proof of the theorem.

\eject

\section{The correspondence between isometry groups and configurations}\label{isoconfig m1}

In this section, we will state some of the basic results on the differential and topological aspects of the space $ \mathcal{J}_{\omega} = \tilde{\mathcal{J}}_{c_{1},c_{2},c_{3}} $ of compatible almost complext structures on $ \widetilde{M}_{c_{1} , c_{2}, c_{3} }$.  Our main goal is to illustrate the correspondence between isometry groups and configurations, as was used in Section \ref{sec m1: generators}.

By Section \ref{chp m1: Structure}, we know that the space $ \mathcal{J}_{\omega} $ is a disjoint union of finitely many strata, each of which is characterized by the existence of a unique chain of holomorphic spheres (configurations). Our first remark is that these sets indeed give a stratification of $ \mathcal{J}_{\omega} $ into Fr\'{e}chet manifolds.

\begin{proposition} Let $ U_{\mathcal{A}} \subset \mathcal{J}_{\omega} $ be a stratum characterized by the existence of a configuration of $ J $-holomorphic embedded spheres $ C_{1} \cup C_{2} \cup ... \cup C_{N} $ representing a given set of distinct homology classes, $ \mathcal{A} = \{ A_{1} , ... , A_{N} \} $ of negative self-intersection. Then $ U_{\mathcal{A}} $ is a cooriented Fr\'{e}chet submanifold of $ \mathcal{J}_{\omega} $ of real codimension $ 2N - 2 c_{1}(A_{1}+...+A_{N}) $.
\end{proposition}
\begin{proof}
The proof of this proposition is identical to the proof of Proposition 7.1 in \cite{AnjPin}.
\end{proof}

Consider the subspace $ \mathcal{J}^{int}_{\omega} \subset \mathcal{J}_{\omega} $ of compatible, integrable, complex structures. Given a stratum $ U_{i} $, set $ V_{i} = U_{i} \cap \mathcal{J}^{int}_{\omega} $.

\begin{lemma}
For any $ J \in \mathcal{J}^{int}_{\omega} $, the complex surface $ (\mathbb{X}_{4},J) $ is the 3-fold blow-up of a Hirzebruch surface.
\end{lemma}
\begin{proof}
Given $ J \in \mathcal{J}^{int}_{\omega} $, the class $ E_{3} $ is always represented by an exceptional curve. Blowing it down yields a surface $ (\mathbb{X}_{3},J) $, which is a twofold blow-up of a Hirzebruch surface by Lemma 7.2 in \cite{AnjPin}. See Section \ref{chp m1: Structure} for an explanation of how these configurations appear.
\end{proof}

\begin{remark} As detailed in \cite[Theorem 2.3]{AbGrKi} and \cite[Lemma 7.5]{AnjPin}, the space $ \mathcal{J}^{int}_{\omega} $ is in fact a Fr\'{e}chet submanifold of $ \mathcal{J}_{\omega} $.
\end{remark}

We denote the group of diffeomorphisms of $ \mathbb{X}_{4} $ acting trivially on homology by Diff$ _{h} $, let $ \text{Aut}_{h}(J) \subset \text{Diff}_{h} $ be the subgroup of complex automorphisms of $ (M,J) $ and $ \text{Iso}_{h}(\omega,J) \subset \text{Aut}_{h}(J) $ be the K\"{a}hler isometry group of $ (M,\omega, J) $.

\begin{proposition}\label{complextransitive}
Let $ J_{1},J_{2} \in V_{i} $ be two integrable compatible structures in the same stratum. Then there exists $ \phi \in \text{Diff}_{h} $ such that $ J_{2} = \phi_{*}J_{1} $. Hence we have
\begin{center}
$ V_{i} = U_{i} \cap \mathcal{J}^{int}_{\omega} = (\text{Diff}_{h} \cdot J_{i}) \cap \mathcal{J}_{\omega} $, for any $ J_{i} \in V_{i} $.
\end{center}
\end{proposition}
\begin{proof}
The proof is similar to Proposition 7.3 in \cite{AnjPin} with the only difference being that in our case the construction of compatible complex structures on $ \mathbb{X}_{4} $ can be made in more numerous ways, resulting in a total of 56 types that yield the configurations listed at the end of Section \ref{chp m1: Structure}.

Briefly, for $\mathbb{X}_{3}$, the one-point blow-down of our space, we have exactly 6 types of compatible complex structures:

\begin{enumerate}
\item[(1)] Twofold blow-up of $ \mathbb{F}_{0} $ at two generic points (not lying on the same fiber $ F $ nor on the same section $ B $).
\item[(2)] Twofold blow-up of $ \mathbb{F}_{0} $ at two distinct points on the same fiber $ F $.
\item[(3)] Twofold blow-up of $ \mathbb{F}_{0} $ at two distinct points on the same section $ B $.
\item[(4)] Twofold blow-up of $ \mathbb{F}_{0} $ at two ``infinitely near" points on a fiber, that is, the blow-up of $ \mathbb{F}_{0} $ at $ p $ followed by the blow-up of $ \widetilde{\mathbb{F}}_{0} $ at the line $ l_{p} = T_{p}F \subset T_{p}\mathbb{F}_{0} $ on the exceptional divisor.
\item[(5)] Twofold blow-up of $ \mathbb{F}_{0} $ at two ``infinitely near" points on a flat section $ B $, that is, at $ (p,l_{p}=T_{p}B) $.
\item[(6)] Twofold blow-up of $ \mathbb{F}_{0} $ at two ``infinitely near" points $ (p,l_{p}) $ with the direction $ l_{p} $ transverse to $ T_{p}F $ and $ T_{p}B $.
\end{enumerate}

To get $ \mathbb{X}_{4} $, we consider a third blow-up on each of these types. This produces the 56 types mentioned above, and we refer the reader to the Hirzebruch surfaces in Section \ref{sec m1: generators} to observe where the last blow-up may take place.

For instance, consider the toric picture obtained from blowing-up polytope \circled{1} in Figure \ref{Toric pictures} at  vertex $(iv)$.
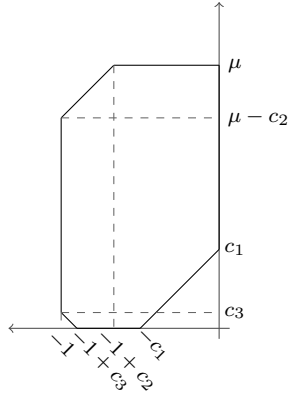
\begin{figure}[thp]
\begin{center}
\begin{tikzpicture}[scale=0.7, roundnode/.style={circle, draw=black!80, thick, minimum size=7mm}, font=\footnotesize]
 
    \draw[<-][black!70] (-4,0) -- (0.2,0); 
    \draw[->][black!70] (0,-0.2) -- (0,6.2); 

	\draw (-3,0.3) -- (-2.7,0);
	\draw (-3,0.3) -- (-3,4);
    \draw (-3,4) -- (-2,5);
    \draw (-2,5) -- (0,5);
    \draw (0,5) -- (0,1.5);
    \draw (-1.5,0) -- (-2.7,0);
    \draw (-1.5,0) -- (0,1.5);
   
	\draw[dashed, black!60] (-3,4) -- (0,4);
	\draw[dashed, black!60] (-2,5) -- (-2,0);
	\draw[dashed, black!60] (-3,0.3) -- (-3,0);
	\draw[dashed, black!60] (-3,0.3) -- (0,0.3);	

    \node at (0.3,0.3) {$c_{3}$};
    \node at (0.3,1.5) {$c_{1}$};
	\node at (0.3,5) {$\mu$};
	\node at (0.75,4) {$\mu-c_{2} $};
	\node[rotate=-45] at (-1.25,-0.3) {$ -c_{1} $};
	\node[rotate=-45] at (-1.75,-0.6) {$ -1 + c_{2} $};
	\node[rotate=-45] at (-2.25, -0.6) {$ -1 + c_{3} $};
	\node[rotate=-45] at (-3, -0.3) {$ -1 $};

\end{tikzpicture}
\end{center}
\caption{Toric action obtained from blowing-up polytope 1 at vertex $(iv)$}
 \label{Toric picture 1(iv)}
 \end{figure}

This figure is obtained by blowing up $ \mathbb{F}_{0} $ at three distinct points $ p_{1}$, $p_{2} $ and $p_{3} $ such that $p_{1} $ and $p_{2} $ lie on the same fiber $ F $, and $p_{2} $ and $p_{3} $ lie on the same section $ B $. Since the complex automorphism group of $ \mathbb{F}_{0} $ is isomorphic to
\begin{center}
$ \text{Aut} (\mathbb{F}_{0}) \simeq (\text{PSL}(2,\mathbb{C}) \times \text{PSL}(2,\mathbb{C})) \ltimes \mathbb{Z}_{2} $,
\end{center}
it follows that the automorphism group of the 3-point blow-up must act transitively on the triple $ (p_{1}, p_{2}, p_{3}) $ that defines the almost complex structure.
\end{proof}

Now, to prove the correspondence between isometry groups and configurations, we start by recalling the following fact due to Abreu-Granja-Kitchloo:

\begin{theorem} (\cite[Corollary 2.6]{AbGrKi}) If $ J \in \mathcal{J}^{int}_{\omega} $ is such that the inclusion
\begin{center}
$ \text{Iso}_{h}(\omega,J) \hookrightarrow \text{Aut}_{h}(J) $ 
\end{center}
is a weak homotopy equivalence, then the inclusion of the $ \text{Symp}_{h}(M,\omega) $-orbit of $ J $ in $ ( \text{Diff}_{h} \cdot J ) \cap \mathcal{J}^{int}_{\omega} $
\begin{center}
$ \text{Symp}_{h}(M,\omega) / \text{Iso}_{h}(\omega,J) \hookrightarrow ( \text{Diff}_{h} \cdot J ) \cap \mathcal{J}^{int}_{\omega} $
\end{center}
is also a weak homotopy equivalence.
\end{theorem}

\begin{lemma} \label{isometrygroups}
For any $ J \in \mathcal{J}^{int}( \widetilde{M}_{c_{1} , c_{2}, c_{3}}  ) $, the inclusion $ \text{Iso}_{h}(\omega_{c_{1}, c_{2}, c_{3}}, J) \hookrightarrow \text{Aut}_{h}(J) $ is a weak homotopy equivalence.
\end{lemma}\label{B06}
\begin{proof}
We start by recalling that if $ (\widetilde{M},\widetilde{J}) $ is the blow-up of $ (M,J) $ at a point $ p $, then the complex automorphism group of $ \widetilde{J} $ is isomorphic to the stabilizer subgroup of $ p $ in the automorphism group of $ J $.

Going through all possible types of compatible almost complext structures, we can see that the space $ \text{Aut}_{h}(\mathbb{X}_{4}, J) $ is homotopy equivalent to either a point, $ S^{1} $ or $ T^{2} $. More specifically, looking at the configurations in  Section \ref{chp m1: Structure}, we get:
\[
  \text{Aut}_{h}(\mathbb{X}_{4}, J) \simeq \left\{\def\arraystretch{1.2}%
  \begin{array}{@{}c@{\quad}l@{}}
	{*} & \text{if $J$ corresponds to  1:13}\\    
    T^{2} & \text{if $J$ corresponds to  1:1-6,  2:1-6,  3:1-6,  4:1-6,  5:1-6}\\
    S^{1} & \text{if $J$ corresponds to one of the remaining configurations}\\
     \end{array}\right.
\]
On the other hand, the isometry groups of the Hirzebruch surfaces are the maximal compact Lie subgroups of their complex automorphism groups:
\[
  \text{Iso}(\mathbb{F}_{i}, \omega) \simeq \left\{\def\arraystretch{1.2}%
  \begin{array}{@{}c@{\quad}l@{}}
    (SO(3) \times SO(3)) \ltimes \mathbb{Z}_{2} & \text{for $i=0$}\\
    U(2) & \text{for $i=1$}\\
  \end{array}\right.
\]
In particular, they are deformation retracts of $ \text{Aut}(\mathbb{F}_{m}) $, and after blow-up they induce isometry groups $ \text{Iso}_{h}(\mathbb{X}_{4},\omega,J) $ that are homotopy equivalent to the cases designated above.
For example, recalling Figure \ref{Toric picture 1(iv)} in the proof of Proposition \ref{complextransitive}, we see that its automorphism group is homotopy equivalent to the isometry group of the almost complex structure in the stratum characterized by configuration 1.4 (see Figure  \ref{Config1 of AP blown up 1-6}).
\end{proof}

Combining the last three results, we get:

\begin{corollary}\label{homotopytype}
Given $ J \in V_{i} \subset \mathcal{J}^{int}_{\omega} $, there is a weak homotopy equivalence 
\begin{center}
$ \text{Symp}_{h}(\widetilde{M}_{c_{1} , c_{2}, c_{3}}, \omega) / \text{Iso}_{h}(\omega,J) \simeq V_{i} $.
\end{center}
\end{corollary}\label{B07}

\section{Some toric actions for \texorpdfstring{$ \mu > 1 $}{Lg}}
\label{app: relations m>1}

Following a similar line of thought to that of Section \ref{sec m1: generators}, we can draw the possible Delzant polytopes for $ \widetilde{M}_{\mu , c_{1} , c_{2}, c_{3} }$. In this section, we will briefly summarize this construction and show some of these polytopes, more precisely, the ones which are necessary for the proof of Proposition \ref{specialSamelson}. 

Let $ T^{4} \subset U(4) $ act in the standard way on $\mathbb{C}^{4}$. Given an integer $ n \geq 0 $, the action of the subtorus $ T^{2} _{n} := (ns+t,t,s,s) $ is Hamiltonian with moment map

\begin{center}
$ (z_{1},...,z_{4}) \mapsto (n |z_{1}|^{2} + |z_{3}|^{2} + |z_{4}|^{2} , |z_{1}|^{2} + |z_{2}|^{2}) $.
\end{center}

We identify $( S^{2} \times S^{2}, \mu \sigma \oplus \sigma ) $ with each of the toric Hirzebruch surfaces $ \mathbb{F}_{\mu}^{2k} $, $ 0 \leq k \leq l $, defined as the symplectic quotient $ \mathbb{C}^{4} // T^{2} _{2k} $ at the regular value $ (\mu + k, 1) $ endowed with the residual action of the torus $ T(2k) := (0,u,v,0) \subset T^{4} $. The image $ \Delta (2k) $ of the moment map is the convex hull of
$$ \{ (0,0), (1,0), (1,\mu + k), (0, \mu - k) \}.$$

Similarly, we identify $( S^{2} \tilde{\times} S^{2}, \omega_{\mu} ) $ with the toric Hirzebruch surface $ \mathbb{F}^{\mu}_{2k-1} $, $ 1 \leq k \leq l $, defined as the symplectic quotient $ \mathbb{C}^{4} // T^{2} _{2k-1} $ at the regular value $ (\mu + k, 1) $. The image $ \Delta (2k-1) $ of the moment map of the residual action of the torus $ (0,u,v,0)$ is the convex hull of
$$ \{ (0,0), (1,0), (1,\mu + k), (0, \mu - k + 1) \}.$$
Since the group $ \text{Symp}_{h}(M_{\mu}) $ of symplectomorphisms acting trivially on homology is connected, any two identifications of $ \mathbb{F}^{\mu}_{n} $ with the spaces $ S^{2} \times S^{2} $ and $ S^{2} \tilde{\times} S^{2} $ are isotopic and lead to isotopic identifications of $ \text{Symp}_{h}(\mathbb{F}^{\mu}_{n}) $ with the respective symplectomorphism groups.

We identify the symplectic blow-up $ \widetilde{M}_{\mu,c_{1}} $ at a ball of capacity $c_{1} $ with the equivariant blow-up of the Hirzebruch surfaces $ \mathbb{F}^{\mu}_{n} $. We define the even torus action $ \widetilde{T}(2k) $ as the equivariant blow-up of the toric action of $T(2k)$ on $ \mathbb{F}^{\mu}_{2k} $ at the fixed point $ (0,0) $ with capacity $ c_{1} $. The image $\widetilde\Delta(2k)$ of the moment map then is the convex hull of
$$ \{ (1,\mu + k), (0, \mu - k) , (0, c_{1} ), (c_{1} , 0), (1,0) \}.$$
Similarly, we define the odd torus action $ \widetilde{T}(2k-1) $ as the equivariant blow-up of the toric action of $T(2k-1)$ on $ \mathbb{F}^{\mu}_{2k-1} $ at the fixed point $ (0,0) $ with capacity $1- c_{1} $. The image $\widetilde\Delta(2k-1)$ of the moment map then is the convex hull of
$$ \{ (1,\mu - c_{1} + k), (0, \mu - c_{1} - k +1) , (0, 1- c_{1} ), (1- c_{1} , 0), (1,0) \}.$$
Note that when $ c_{1} < c_{crit} := \mu - l $, $ \widetilde{M}_{\mu,c_{1}} $ admits exactly $ 2l+1 $ inequivalent toric structures $ \widetilde{T}(0),...,\widetilde{T}(2l) $, while when  $ c_{1} \geq c_{crit}$, it admits $ 2l $ of those, namely $ \widetilde{T}(0),...,\widetilde{T}(2l-1) $.

The K\"{a}hler isometry group of  $ \mathbb{F}^{\mu}_{n} $ is $ N(T^{2} _{n}) / T^{2} _{n} $ where $ N(T^{2} _{n})$ is the normalizer of $T^{2} _{n} $ in $U(4)$. There is a natural isomorphism $N(T^{2} _{0}) / T^{2} _{0} \simeq SO(3) \times SO(3)  := K(0)$, while for $ k \geq 1 $, we have $ N(T^{2} _{2k}) / T^{2} _{2k} \simeq S^{1} \times SO(3) := K(2k) $ and $ N(T^{2} _{2k-1}) / T^{2} _{2-1} \simeq U(2) := K(2k-1) $ The
restrictions of these isomorphisms to the maximal tori are given in coordinates by
\begin{align*}
 (u,v) & \mapsto (-u,v) \in T(0) := S^{1} \times S^{1} \subset K(0) \\
(u,v)  & \mapsto (u,ku+v) \in T(2k) := S^{1} \times S^{1} \subset K(2k) \\
(u,v)  & \mapsto (u+v ,ku + (k-1) v) \in T(2k-1) := S^{1} \times S^{1} \subset K(2k-1)
\end{align*}
These identifications imply that the moment polygon associated to the maximal tori $T(n) = S^{1} \times S^{1} \subset K(n)$ and $\widetilde{T}(n)$ are the images of $\Delta(n)$ and $ \widetilde{\Delta}(n) $, respectively, under the transformations $ C_{n} \in GL(2,\mathbb{Z}) $ given by

\begin{center}
$C_{0} =
\left( \begin{array}{cc}
-1 & 0 \\
0 & 1 \\
\end{array} \right)$, $C_{2k} =
\left( \begin{array}{cc}
1 & 0 \\
-k & 1 \\
\end{array} \right)$ \mbox{ and } $C_{2k-1} =
\left( \begin{array}{cc}
1-k & k \\
1 & -1 \\
\end{array} \right)$.
\end{center}
Under the blow-down map, $ \widetilde{T}(n) $ is sent to the maximal torus of $ K(n) $ for all $ n \geq 0 $. By \cite{LalPin} $ \text{Symp}(  \widetilde{M}_{\mu,c_{1}} ) $ is connected, hence the choices involved in these identifications give the same maps up to homotopy.

\begin{figure}[thp]
\begin{center}
\begin{tikzpicture}[scale=0.8,roundnode/.style={circle, draw=black!80, thick, minimum size=7mm}, font=\footnotesize]

	\node[roundnode] at (0,3.5) (maintopic) {6} ;
 
    \draw[->][black!70] (0.8,0) -- (8.4,0); 
    \draw[->][black!70] (3,-4.4) -- (3,3.2); 

    \draw[brown] (1,3) -- (8.2,-4.2);
    \draw[brown] (8.2,-4.2) -- (7.6,-3.9);
    \draw[brown] (7.6,-3.9) -- (5.5,-2.5);
    \draw[brown] (5.5,-2.5) -- (3.5,-0.5);
    \draw[brown] (3.5,-0.5) -- (2,1.5);
    \draw[brown] (2,1.5) -- (1,3);

	\filldraw [black] (1,3) circle (2pt);
	\filldraw [black] (8.2,-4.2) circle (2pt);
	\filldraw [black] (7.6,-3.9) circle (2pt);
	\filldraw [black] (5.5,-2.5) circle (2pt);
	\filldraw [black] (3.5,-0.5) circle (2pt);
	\filldraw [black] (2,1.5) circle (2pt);

    \node at (1.2,3.3) {$(i)$};
    \node at (8.2,-4.6) {$(vi)$};
    \node at (7.4,-4.5) {$(v)$};
    \node at (5.3,-2.9) {$(iv)$};
    \node at (3.3,-0.9) {$(iii)$};
    \node at (1.6,1.5) {$(ii)$};
   
	\draw[dashed, black!60] (1,3) -- (3,3);
	\draw[dashed, black!60] (1,3) -- (1,0);
	\draw[dashed, black!60] (8.2,-4.2) -- (8.2,0);
	\draw[dashed, black!60] (8.2,-4.2) -- (3,-4.2);	
	\draw[dashed, black!60] (7.6,-3.9) -- (7.6,0);
	\draw[dashed, black!60] (7.6,-3.9) -- (3,-3.9);
	\draw[dashed, black!60] (5.5,-2.5) -- (5.5,0);
	\draw[dashed, black!60] (5.5,-2.5) -- (3,-2.5);
	\draw[dashed, black!60] (3.5,-0.5) -- (3,-0.5);
	\draw[dashed, black!60] (3.5,-0.5) -- (3.5,0);
	\draw[dashed, black!60] (2,1.5) -- (3,1.5);
	\draw[dashed, black!60] (2,1.5) -- (2,0);

	\node[rotate=45] at (3.2,3) {$k$};
	\node[rotate=45] at (3.6,1.8) {$k-k c_{1}$};
	\node[rotate=45] at (2.3,-0.8) {$-1+c_{1}$};
	\node[rotate=45] at (1.8,-3.4) {$- \mu -1 +k +c_{1}$};
	\node[rotate=45] at (2,-4.4) {$- \mu +c_{1} + c_{2}$};
	\node[rotate=45] at (1.7,-5.2) {$- \mu +c_{1} + (k-1)c_{2}$};
	\node[rotate=45] at (0.5,-0.7) {$ -k+1 $};
	\node[rotate=45] at (1,-1.3) {$ -k+1+(k-1)c_{1} $};
	\node[rotate=45] at (3.7,0.5) {$ 1-c_{1} $};
	\node[rotate=45] at (6.1,1) {$ \mu + 1 - k -c_{1} $};
	\node[rotate=45] at (8.4,1.2) {$ \mu + 1 - c_{1} - kc_{2} $};
	\node[rotate=45] at (9,1.1) {$ \mu + 1 - c_{1} - c_{2} $};
    
	\node at (4,-6) {($x_{2k-1,1,\cdot}$, $y_{2k-1,1, \cdot }$)};

\end{tikzpicture}
\end{center}
\caption{Delzant polytope corresponding to Configuration (10) of \cite{AnjPin} for $\mathbb{X}_{3}$, with the next blow-ups plotted}
\label{Config10 of AP}
\end{figure}
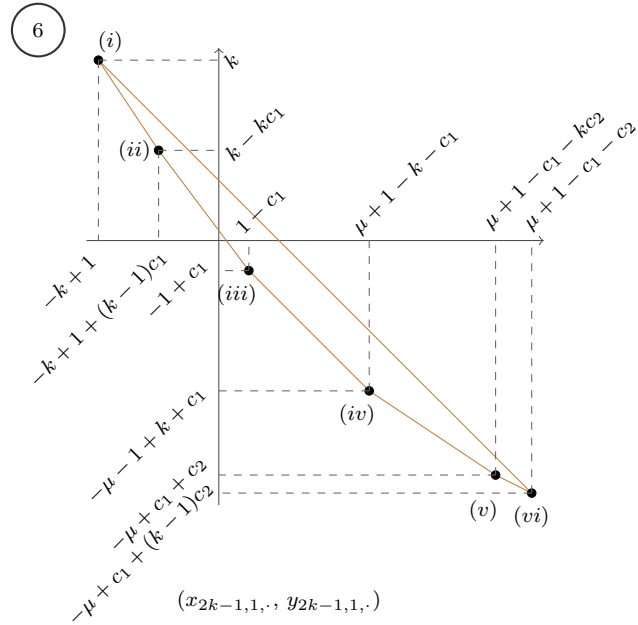

\begin{figure}[thp]
\begin{center}
\begin{tikzpicture}[scale=0.8,roundnode/.style={circle, draw=black!80, thick, minimum size=7mm}, font=\footnotesize]

	\node[roundnode] at (0,3.5) (maintopic) {7} ;
 
    \draw[->][black!70] (0.8,0) -- (8.6,0); 
    \draw[->][black!70] (3,-4.6) -- (3,3.2); 

    \draw[brown] (1,3) -- (8.5,-4.5);
    \draw[brown] (8.5,-4.5) -- (6.4,-3.1);
    \draw[brown] (6.4,-3.1) -- (5.2,-2.2);
    \draw[brown] (5.2,-2.2) -- (3.5,-0.5);
    \draw[brown] (3.5,-0.5) -- (2,1.5);
    \draw[brown] (2,1.5) -- (1,3);

	\filldraw [black] (1,3) circle (2pt);
	\filldraw [black] (8.5,-4.5) circle (2pt);
	\filldraw [black] (6.4,-3.1) circle (2pt);
	\filldraw [black] (5.2,-2.2) circle (2pt);
	\filldraw [black] (3.5,-0.5) circle (2pt);
	\filldraw [black] (2,1.5) circle (2pt);

    \node at (1.2,3.3) {$(i)$};
    \node at (8.4,-4.8) {$(vi)$};
    \node at (6.2,-3.6) {$(v)$};
    \node at (5.3,-2.6) {$(iv)$};
    \node at (3.3,-0.9) {$(iii)$};
    \node at (1.6,1.5) {$(ii)$};
   
	\draw[dashed, black!60] (1,3) -- (3,3);
	\draw[dashed, black!60] (1,3) -- (1,0);
	\draw[dashed, black!60] (8.5,-4.5) -- (8.5,0);
	\draw[dashed, black!60] (8.5,-4.5) -- (3,-4.5);	
	\draw[dashed, black!60] (6.4,-3.1) -- (6.4,0);
	\draw[dashed, black!60] (6.4,-3.1) -- (3,-3.1);
	\draw[dashed, black!60] (5.2,-2.2) -- (5.2,0);
	\draw[dashed, black!60] (5.2,-2.2) -- (3,-2.2);
	\draw[dashed, black!60] (3.5,-0.5) -- (3,-0.5);
	\draw[dashed, black!60] (3.5,-0.5) -- (3.5,0);
	\draw[dashed, black!60] (2,1.5) -- (3,1.5);
	\draw[dashed, black!60] (2,1.5) -- (2,0);

	\node[rotate=45] at (3.2,3) {$k$};
	\node[rotate=45] at (3.6,1.8) {$k-k c_{1}$};
	\node[rotate=45] at (2.3,-0.8) {$-1+c_{1}$};
	\node[rotate=45] at (1.6,-3.4) {$- \mu -1 +k +c_{1}+c_{2}$};
	\node[rotate=45] at (1.2,-4.6) {$- \mu -1+k +c_{1} -(k-1) c_{2}$};
	\node[rotate=45] at (2.2,-4.8) {$- \mu +c_{1} $};
	\node[rotate=45] at (0.5,-0.7) {$ -k+1 $};
	\node[rotate=45] at (0.9,-1.3) {$ -k+1+(k-1)c_{1} $};
	\node[rotate=45] at (3.7,0.5) {$ 1-c_{1} $};
	\node[rotate=45] at (6.1,1.3) {$ \mu + 1 - k -c_{1}-c_{2} $};
	\node[rotate=45] at (7.4,1.4) {$ \mu + 1 -k- c_{1} + kc_{2} $};
	\node[rotate=45] at (9,1) {$ \mu + 1 - c_{1}$};
    
	\node at (4,-6) {($x_{2k-1,2,\cdot}$, $y_{2k-1,2, \cdot }$)};

\end{tikzpicture}
\end{center}

\caption{Delzant polytopes with the next blow-ups plotted, continued}
\label{Config11 of AP}
\end{figure}

\begin{figure}[thp]
\begin{tikzpicture}[roundnode/.style={circle, draw=black!80, thick, minimum size=7mm}, font=\small]

	\node[roundnode] at (-1,6) (maintopic) {8} ;
 
    \draw[->][black!70] (-0.2,0) -- (1.2,0); 
    \draw[->][black!70] (0,-1.1) -- (0,4.9); 

    \draw[brown] (1,4.7) -- (1,-1);
    \draw[brown] (1,-1) -- (0.5,-0.5);
    \draw[brown] (0.5,-0.5) -- (0,0.5);
    \draw[brown] (0,0.5) -- (0,4);
    \draw[brown] (0,4) -- (0.7,4.6);
    \draw[brown] (0.7,4.6) -- (1,4.7);

	\filldraw [black] (1,4.7) circle (2pt);
	\filldraw [black] (1,-1) circle (2pt);
	\filldraw [black] (0.5,-0.5) circle (2pt);
	\filldraw [black] (0,0.5) circle (2pt);
	\filldraw [black] (0,4) circle (2pt);
	\filldraw [black] (0.7,4.6) circle (2pt);

    \node at (1.4,4.7) {$(i)$};
    \node at (1.4,-1) {$(vi)$};
    \node at (1.4,-0.5) {$(v)$};
    \node at (-0.4,0.9) {$(iv)$};
    \node at (0.3,3.8) {$(iii)$};
    \node at (0.7,5) {$(ii)$};

	\draw[dashed, black!60] (1,4.7) -- (0,4.7);
	\draw[dashed, black!60] (1,-1) -- (0,-1);
	\draw[dashed, black!60] (0.5,-0.5) -- (0.5,0);
	\draw[dashed, black!60] (0.5,-0.5) -- (0,-0.5);
	\draw[dashed, black!60] (0.7,4.6) -- (0.7,0);
	\draw[dashed, black!60] (0.7,4.6) -- (0,4.6);

	\node at (-0.9,4.8) {$\mu - c_{2}$};
	\node at (-0.9,4.4) {$\mu -k c_{2}$};
	\node at (-0.9,3.9) {$\mu -k $};
	\node at (-0.5,0.5) {$c_{1}$};
	\node at (-0.5,-0.5) {$-kc_{1}$};
	\node at (-0.5,-1) {$-k$};
	\node[rotate=45] at (1.3,0.5) {$ 1-c_{2} $};
	\node[rotate=45] at (0.6,0.3) {$ c_{1} $};
    
	\node at (0,-2) {($x_{2k,1,\cdot}$, $y_{2k,1, \cdot }$)};

	\node[roundnode] at (7,6) (maintopic) {9} ;
 
    \draw[->][black!70] (7.8,0) -- (9.2,0); 
    \draw[->][black!70] (8,-1.1) -- (8,5.2); 

    \draw[brown] (9,5) -- (9,-1);
    \draw[brown] (9,-1) -- (8.5,-0.5);
    \draw[brown] (8.5,-0.5) -- (8,0.5);
    \draw[brown] (8,0.5) -- (8,3.7);
    \draw[brown] (8,3.7) -- (8.3,4.3);
    \draw[brown] (8.3,4.3) -- (9,5);

	\filldraw [black] (9,5) circle (2pt);
	\filldraw [black] (9,-1) circle (2pt);
	\filldraw [black] (8.5,-0.5) circle (2pt);
	\filldraw [black] (8,0.5) circle (2pt);
	\filldraw [black] (8,3.7) circle (2pt);
	\filldraw [black] (8.3,4.3) circle (2pt);

    \node at (9.4,4.7) {$(i)$};
    \node at (9.4,-1) {$(vi)$};
    \node at (9.4,-0.5) {$(v)$};
    \node at (7.6,0.9) {$(iv)$};
    \node at (7.5,3.2) {$(iii)$};
    \node at (8.7,4.4) {$(ii)$};

	\draw[dashed, black!60] (9,5) -- (8,5);
	\draw[dashed, black!60] (8.5,-0.5) -- (8.5,0);
	\draw[dashed, black!60] (8.5,-0.5) -- (8,-0.5);
	\draw[dashed, black!60] (8.3,4.3) -- (8.3,0);
	\draw[dashed, black!60] (8.3,4.3) -- (8,4.3);
	\draw[dashed, black!60] (9,-1) -- (8,-1);

	\node at (7.7,5) {$\mu$};
	\node at (7.1,4.3) {$\mu -k + c_{2}$};
	\node at (6.9,3.7) {$\mu -k - c_{2} $};
	\node at (7.5,0.5) {$c_{1}$};
	\node at (7.5,-0.5) {$-kc_{1}$};
	\node at (7.5,-1) {$-k$};
	\node[rotate=45] at (8.3,0.3) {$ c_{2} $};
	\node[rotate=45] at (8.8,0.3) {$ c_{1} $};
    
	\node at (8,-2) {($x_{2k,2,\cdot}$, $y_{2k,2, \cdot }$)};

\end{tikzpicture}

\caption{Delzant polytopes with the next blow-ups plotted, continued}\label{Config16,18 of AP}
\bigskip
\end{figure}

\begin{figure}[thp]
\begin{minipage}{.5\textwidth}
\begin{center}
\begin{tikzpicture}[roundnode/.style={circle, draw=black!80, thick, minimum size=7mm}, font=\small]

	\node[roundnode] at (-1,6) (maintopic) {10} ;
 
    \draw[->][black!70] (-0.2,0) -- (1.2,0); 
    \draw[->][black!70] (0,-1.1) -- (0,5.2); 

    \draw[brown] (1,5) -- (1,-0.7);
    \draw[brown] (1,-0.7) -- (0.7,-0.6);
    \draw[brown] (0.7,-0.6) -- (0.5,-0.5);
    \draw[brown] (0.5,-0.5) -- (0,0.5);
    \draw[brown] (0,0.5) -- (0,4);
    \draw[brown] (0,4) -- (1,5);

	\filldraw [black] (1,5) circle (2pt);
	\filldraw [black] (1,-0.7) circle (2pt);
	\filldraw [black] (0.7,-0.6) circle (2pt);
	\filldraw [black] (0.5,-0.5) circle (2pt);
	\filldraw [black] (0,0.5) circle (2pt);
	\filldraw [black] (0,4) circle (2pt);

    \node at (1.4,4.7) {$(i)$};
    \node at (1.4,-0.6) {$(vi)$};
    \node at (0.6,-1.1) {$(iv) (v)$};
    \node at (-0.4,1.2) {$(iii)$};
    \node at (0.5,4) {$(ii)$};

	\draw[dashed, black!60] (1,5) -- (0,5);
	\draw[dashed, black!60] (1,-0.7) -- (0,-0.7);
	\draw[dashed, black!60] (0.7,-0.6) -- (0.7,0);
	\draw[dashed, black!60] (0.7,-0.6) -- (0,-0.6);
	\draw[dashed, black!60] (0.5,-0.5) -- (0.5,0);
	\draw[dashed, black!60] (0.5,-0.5) -- (0,-0.5);

	\node at (-0.3,5) {$\mu$};
	\node at (-0.7,4) {$\mu - k $};
	\node at (-0.5,0.5) {$c_{1}$};
	\node at (-0.6,-0.2) {$-kc_{1}$};
	\node at (-1,-0.6) {$-k-kc_{2}$};
	\node at (-0.8,-1) {$-k+c_{2}$};
	\node[rotate=45] at (1.3,0.5) {$ 1-c_{2} $};
	\node[rotate=45] at (0.6,0.3) {$ c_{1} $};
    
	\node at (0,-2) {($x_{2k,5,\cdot}$, $y_{2k,5, \cdot }$)};

\end{tikzpicture}
\end{center}
\end{minipage}%
\begin{minipage}{.5\textwidth}
\begin{center}
\begin{tikzpicture}[roundnode/.style={circle, draw=black!80, thick, minimum size=7mm}, font=\small]

	\node[roundnode] at (-1,6) (maintopic) {11} ;
 
    \draw[->][black!70] (-0.2,0) -- (1.2,0); 
    \draw[->][black!70] (0,-1.1) -- (0,5.2); 

    \draw[brown] (1,5) -- (1,-1);
    \draw[brown] (1,-1) -- (0.5,-0.5);
    \draw[brown] (0.5,-0.5) -- (0.3,-0.1);
    \draw[brown] (0.3,-0.1) -- (0,0.8);
    \draw[brown] (0,0.8) -- (0,4);
    \draw[brown] (0,4) -- (1,5);

	\filldraw [black] (1,5) circle (2pt);
	\filldraw [black] (1,-1) circle (2pt);
	\filldraw [black] (0.5,-0.5) circle (2pt);
	\filldraw [black] (0.3,-0.1) circle (2pt);
	\filldraw [black] (0,0.8) circle (2pt);
	\filldraw [black] (0,4) circle (2pt);

    \node at (1.4,4.7) {$(i)$};
    \node at (1.4,-1) {$(vi)$};
    \node at (1.4,-0.6) {$(v)$};
    \node at (1.4,-0.2) {$(iv)$};
    \node at (-0.4,1.2) {$(iii)$};
    \node at (0.5,4) {$(ii)$};

	\draw[dashed, black!60] (1,5) -- (0,5);
	\draw[dashed, black!60] (0.5,-0.5) -- (0.5,0);
	\draw[dashed, black!60] (0.5,-0.5) -- (0,-0.5);
	\draw[dashed, black!60] (0.3,-0.1) -- (0.3,0);
	\draw[dashed, black!60] (0.3,-0.1) -- (0,-0.1);
	\draw[dashed, black!60] (1,-1) -- (0,-1);

	\node at (-0.3,5) {$\mu$};
	\node at (-0.7,4) {$\mu - k $};
	\node at (-0.8,0.8) {$c_{1} + c_{2} $};
	\node at (-0.8,-0.1) {$c_{1} - 2 c_{2}$};
	\node at (-0.5,-0.5) {$-kc_{1}$};
	\node[rotate=45] at (0.3,0.3) {$ c_{2} $};
	\node[rotate=45] at (0.7,0.3) {$ c_{1} $};
	\node at (-0.5,-1) {$-k$};
    
	\node at (0,-2) {($x_{2k,3, \cdot }$, $y_{2k,3, \cdot}$)};

\end{tikzpicture}
\end{center}
\end{minipage}
\caption{Delzant polytopes with the next blow-ups plotted, continued}\label{Config13 of AP}
\end{figure}

We identify the symplectic blow-up $ \widetilde{M}_{\mu,c_{1},c_{2}} $ with the equivariant two blow-up of the Hirzebruch surfaces $ \mathbb{F}^{\mu}_{n} $ and obtain inequivalent toric structures. We define the torus actions $ \widetilde{T}_{i}(2k) $, $\widetilde{T_{i}} ( 2k - 1 )$, $i = 1, ... , 5$ as the equivariant blow-ups of the toric action of $\widetilde{T}(n)$ on $ \widetilde{\mathbb{F}}^{\mu}_{2k} $ and $ \widetilde{\mathbb{F}}^{\mu - c_{1}}_{2k-1} $ respectively, with capacity $c_{2}$, at each one of the five fixed points, which correspond to the vertices of the moment polygon $ \widetilde{\Delta} (n) $.
Furthemore, we blow-up each of the resulting toric actions in their six fixed points and obtain $ \widetilde{T}_{i,j}(2k) $, $\widetilde{T}_{i,j} ( 2k - 1 )$, $i = 1, ... , 5$, $j=1,...,6$. These toric pictures arise from the ones described in Section 4.2 of \cite{AnjPin}, by blowing up once more at a ball of capacity $ c_{3} $.   In each case, we will further pick two Hamiltonian $ S^{1} $-actions, ($ x_{2k-1,i,j}, y_{2k-1,i,j} $) or ($ x_{2k,i,j}, y_{2k,i,j} $), as we did in Section \ref{sec m1: generators}.

Since the maps  $ T_{i,j}(n) \rightarrow G_{\mu, c_{1}, c_{2}, c_{3}} $ induce injective maps of fundamental groups (see Remark \ref{injectivehomotopy}), it follows that we can see the actions $x_{n,i,j}$, $y_{n,i,j}$ as elements of $ \pi_{1}(G_{\mu, c_{1}, c_{2}, c_{3}}) $. Then, using Karshon's classification, we can find several relations between these elements. In fact, an easy but long calculation shows that exactly one more generator, for instance $ y_{1,1,1} $, is necessary and sufficient to produce all the other elements.

\end{appendices}

\end{document}